\documentclass[12pt]{article}

\usepackage{appendix}
\usepackage{longtable}
\usepackage{booktabs}
\usepackage{mathrsfs}

\usepackage{amsmath,amsthm,amsfonts,color,graphicx,xcolor,epstopdf}
\usepackage{pdfsync,float,tcolorbox,mathabx,hyperref, halloweenmath}
\usepackage{amssymb}
\usepackage[normalem]{ulem}
\usepackage{subfigure}

\usepackage{algorithm}  
\usepackage{algpseudocode}    

\DeclareUnicodeCharacter{02B9}{'}
\DeclareMathAlphabet\mathbfcal{OMS}{cmsy}{b}{n}

\numberwithin{equation}{section}
\hypersetup{
    colorlinks=true,
    linkcolor=blue,
    filecolor=magenta,      
    urlcolor=magenta,
  }

\newtheorem{remark}{Remark}[section]
\newtheorem{lemma}{Lemma}[section]
\newtheorem{theorem}{Theorem}[section]

\newtheorem{proposition}{Proposition}[section]

\begin{document}
\title{Dispersion relation reconstruction for 2D Photonic Crystals
based on polynomial interpolation}
\date{}\author{Yueqi Wang\thanks{Department of Mathematics, The University of Hong Kong, Pokfulam Road, Hong Kong. Email: \tt{u3007895@connect.hku.hk}} ~ and ~ Guanglian Li\thanks{Department of Mathematics, The University of Hong Kong, Pokfulam Road, Hong Kong. Email: \tt{lotusli@maths.hku.hk}}
}
\maketitle
\begin{abstract}
Dispersion relation reflects the dependence of wave frequency on its wave vector when the wave passes through certain materials. It demonstrates the properties of this material and thus it is critical. However, dispersion relation reconstruction is very time-consuming and expensive. To address this bottleneck, we propose in this paper an efficient dispersion relation reconstruction scheme based on global polynomial interpolation for the approximation of two-dimensional photonic band functions. Our method relies on the fact that the band functions are piecewise analytic with respect to the wave vector in the Brillouin zone. We utilize suitable sampling points in the Brillouin zone at which we solve the eigenvalue problem involved in the band function calculation, and then employ Lagrange interpolation to approximate the band functions on the whole Brillouin zone. Numerical results show that our proposed method can significantly improve the computational efficiency. 
\end{abstract}
\section{Introduction}
Photonic Crystals (PhCs) are periodic dielectric materials with size of their period comparable to the wavelength \cite{joannopoulos2008molding}. The propagation of electromagnetic waves inside such materials depends heavily on their frequencies. Furthermore, electromagnetic waves within a certain frequency range cannot propagate in certain PhCs. This forbidden frequency range is the so-called {\it band gap}, which motivates many important applications, including optical transistors, photonic fibers and low-loss optical mirrors \cite{yanik2003all,russell2003photonic,labilloy1997demonstration,wang2023analytical}. In this paper, we focus on two-dimensional (2D) PhCs which are periodic in the $xy$ plane and homogeneous along the $z$ axis with dielectric columns or holes spaced in dielectric materials.

To fully understand PhCs, research interest falls on the propagating frequency as well as the band gap. 
The periodicity of PhCs allows using Bloch's theorem so that the original Helmholtz eigenvalue problem on the whole space is transformed into a family of Helmholtz eigenvalue problems defined on the unit cell $\Omega$ parameterized by the wave vector $\mathbf{k}$ varying in the irreducible Brillouin zone (IBZ) $\mathcal{B}_{\text{red}}$ \cite{kuchment1993floquet}. The frequency $\omega_n$ which is a scaling of the square root of $n$th largest eigenvalue, regarded as a function of the wave vector $\mathbf{k}$ is the so-called $n$th band function for all $n\in\mathbb{N}^+$. The band gap is the distance between two adjacent band functions. Consequently, the calculation of the $n$th band function $\omega_n(\mathbf{k})$ involves solving infinite number of the Helmholtz eigenvalue problems defined on the unit cell parameterized by the wave vector $\mathbf{k}\in \mathcal{B}_{\text{red}}$. Since the permittivity takes different values in the inclusion and the background of one cell, and the ratio between these values, the so-called {\it contrast}, should be large to generate the band gap. This leads to Helmholtz eigenvalue problems with high-contrast and piecewise constant coefficients, which is numerically challenging. To reduce this computational cost, a natural approach is to decrease the number of parameters $\mathbf{k}$ by limiting them to $\partial\mathcal{B}_{\text{red}}$. A practical approach is to discretize $\partial\mathcal{B}_{\text{red}}$ uniformly to generate the parameters. Although there is no rigorous theoretical foundation, this approach demonstrates its accuracy for many numerical tests on 2D PhCs. To further reduce the number of parameters, several sampling algorithms have been proposed. In specific, Hussein introduced the model order reduction method to band gap calculation and proposed to use the high symmetry points and the intermediate points centrally intersecting the straight lines joining these high symmetry points as the sampling points \cite{hussein2009reduced}. Klindworth proposed to use Taylor expansions to approximate the reordered band functions based on the fact that band functions can be reordered so that they are analytic functions of $\mathbf{k}$ and an adaptive step size controlling was proposed to determine the sampling points \cite{klindworth2014efficient}. In addition, some improvements to these methods have also been proposed in recent years \cite{scheiber2011model,jorkowski2017higher,jorkowski2018mode}.

\subsection{Motivation for sampling inside $\mathcal{B}_{\text{red}}$}
However, there is no guarantee that band gaps can be calculated accurately using only the information of the wave vector over $\partial\mathcal{B}_{\text{red}}$ \cite{hinuma2017band}. As an example, we depict in Figures \ref{material} two 2D PhCs that exhibit a periodic arrangement of dielectric columns within a dielectric material. Their unit cells, Brillouin zones and IBZs are illustrated in Figures \ref{lattice} and \ref{lattice3}, respectively. Recall that the unit cell of a photonic crystal is denoted as $\Omega$.
\begin{figure}[H]
\centering
\subfigure[PhCs with square lattice]{\label{square material}
\includegraphics[width=4cm]{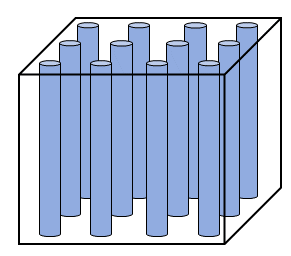}
}%
\subfigure[PhCs with hexagonal lattice]{\label{hex material}
\includegraphics[width=5cm]{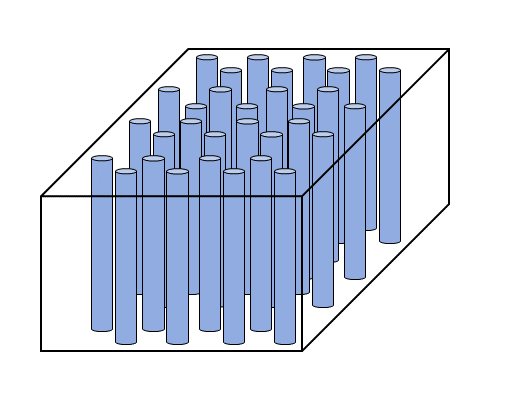}
}%
\centering
\caption{Two 2D PhCs. The materials have dielectric columns homogeneous along the
$z$ direction and periodic along $x$ and $y$.}\label{material}
\end{figure}
Figure \ref{minmax} illustrates the extrema of the first six band functions, which obviously do not always occur over $\partial\mathcal{B}_{\text{red}}$. For example, Figure \ref{minmax4} depicts one extremum inside $\mathcal{B}_{\text{red}}$ which corresponds to the maximum value of the sixth band function. In Table \ref{tab1} we present this maximum value and compare it with the maximum value of this case obtained using only $\partial\mathcal{B}_{\text{red}}$, which demonstrates the importance of the interior information. Furthermore, some work has shown counterexamples that highlight the dangers of just using $\partial\mathcal{B}_{\text{red}}$ \cite{harrison2007occurrence,maurin2018probability,craster2012dangers}. Thus, it is critical to develop an accurate and efficient sampling algorithm in  $\mathcal{B}_{\text{red}}$.

\begin{figure}[H]
\centering
\subfigure{\label{lattice(1)}
\includegraphics[width = .3\textwidth,trim={14cm 4.5cm 12cm 4.5cm},clip]{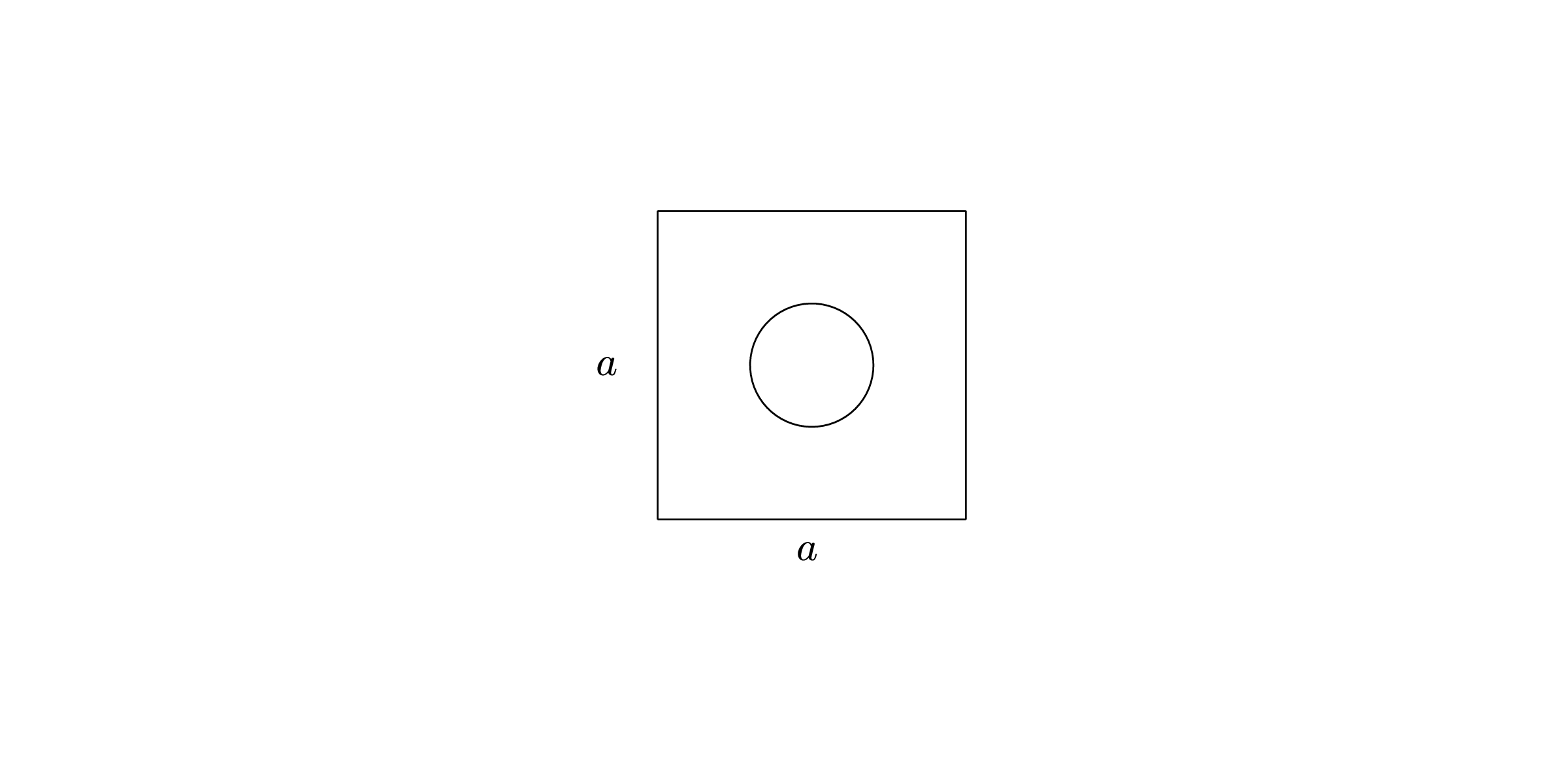}
}%
\subfigure{\label{lattice(2)}
\includegraphics[width = .3\textwidth,trim={14cm 4.5cm 12cm 4.5cm},clip]{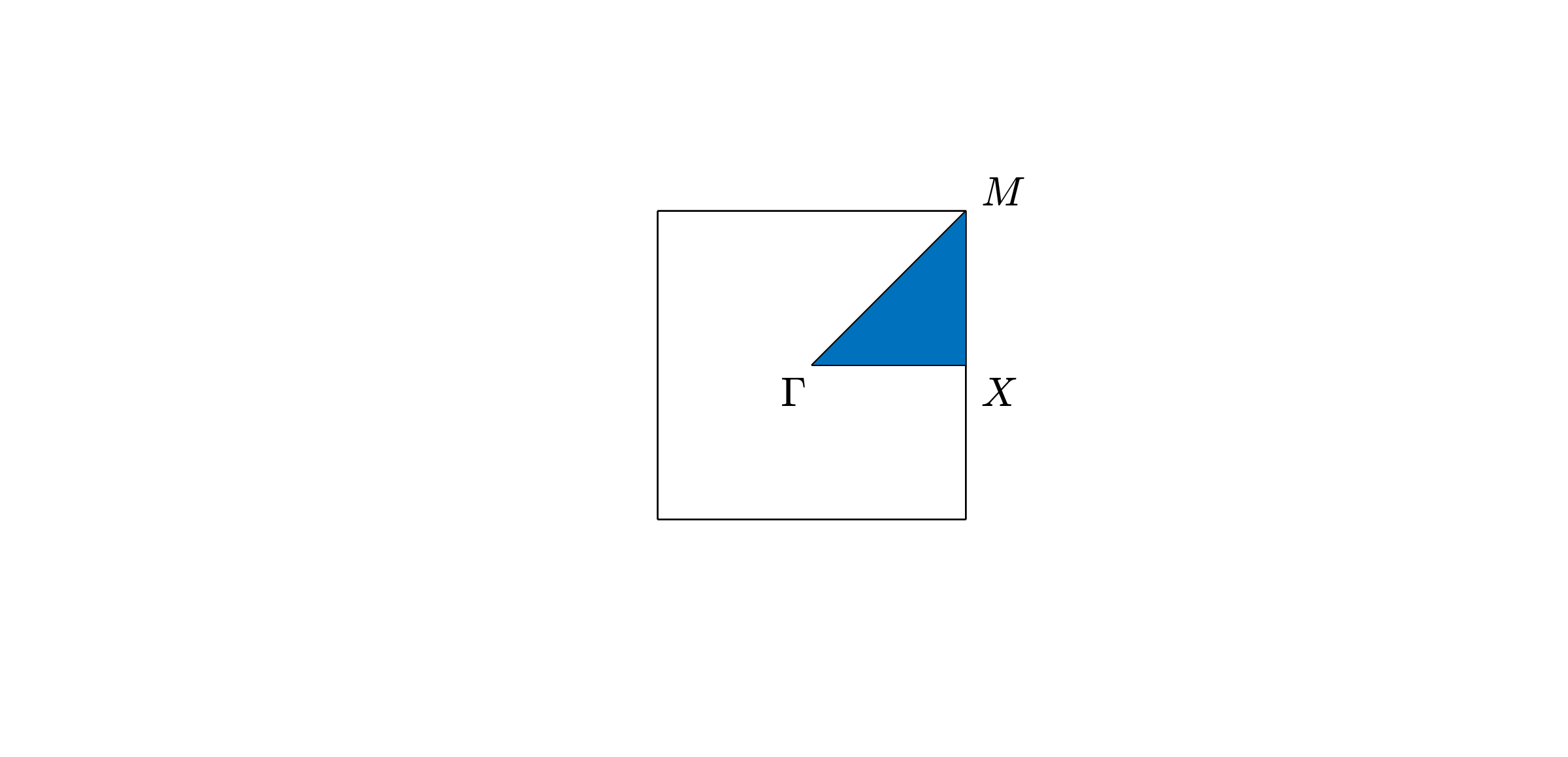}
}%
\centering
\caption{Square lattice: $\Omega$} (left) and the corresponding Brillouin zone (right). The IBZ $\mathcal{B}_{\text{red}}$ is the shaded triangle with vertices $\Gamma=(0,0)$, $X=\frac{1}{a}(\pi,0)$ and $M=\frac{1}{a}(\pi,\pi)$.\label{lattice}
\end{figure}

\begin{figure}[H]
\centering
\subfigure{\label{lattice2(1)}
\includegraphics[width = .27\textwidth,trim={14cm 4.5cm 13cm 4.5cm},clip]{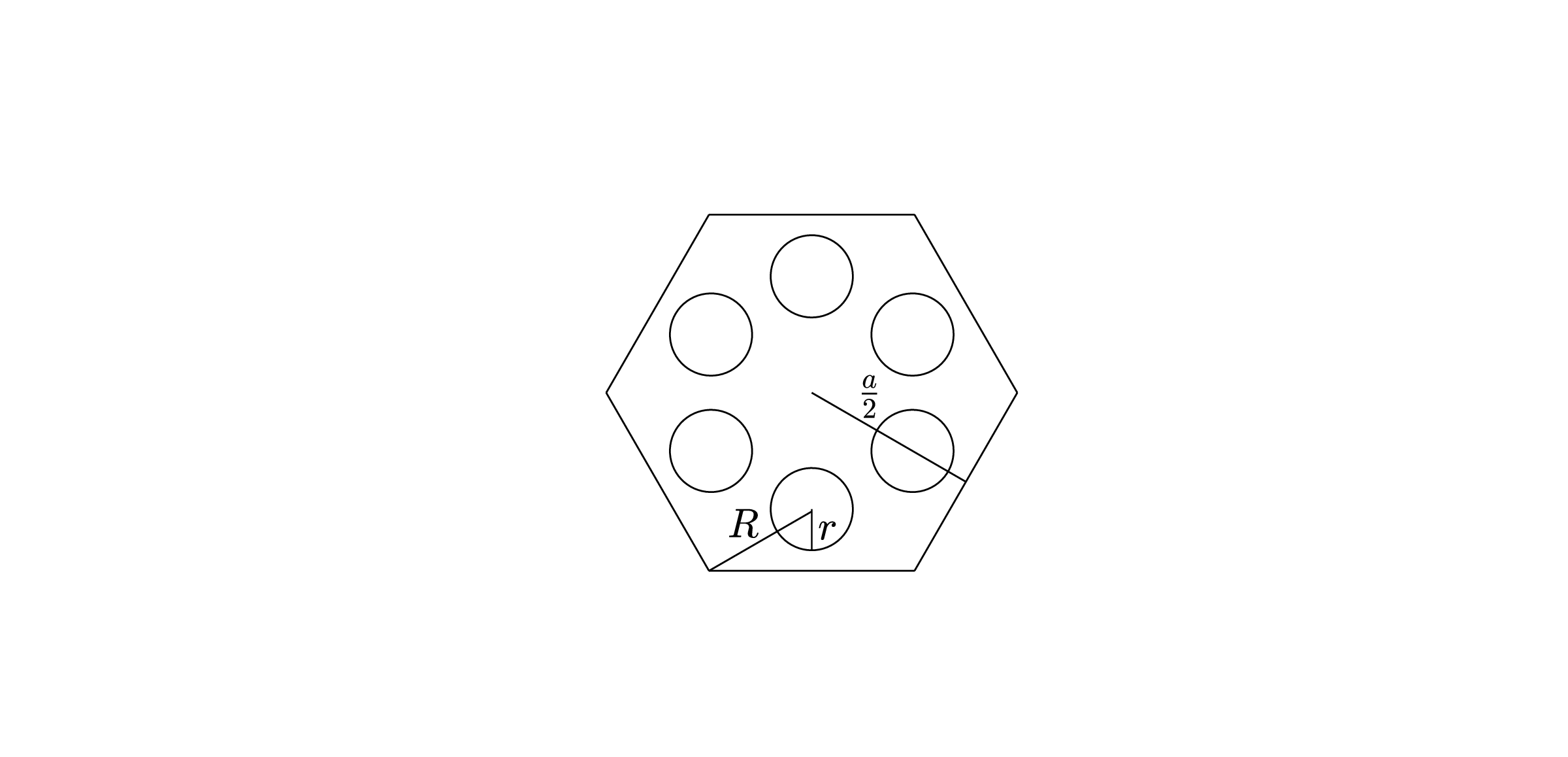}
}%
\subfigure{\label{lattice2(2)}
\includegraphics[width = .27\textwidth,trim={14cm 4.5cm 12cm 4.5cm},clip]{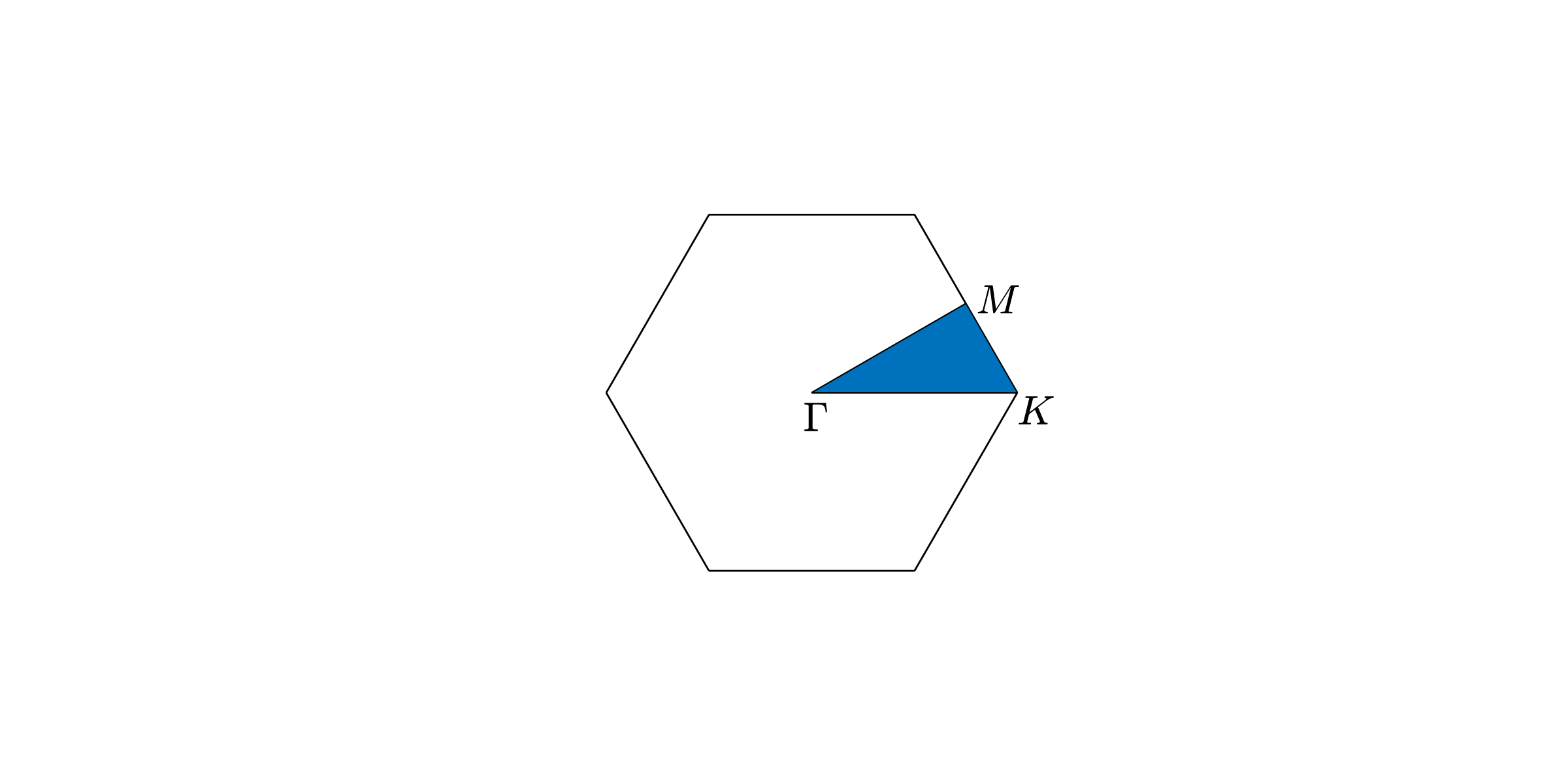}
}%
\centering
\caption{Hexagonal lattice: $\Omega$} (left) and the corresponding Brillouin zone (right). The IBZ $\mathcal{B}_{\text{red}}$ is the shaded triangle with vertices $\Gamma=(0,0)$, $K=\frac{1}{a}(\frac{4}{3}\pi,0)$ and $M=\frac{1}{a}(\pi,\frac{\sqrt{3}}{3}\pi)$.\label{lattice3}
\end{figure}
Nevertheless, to the best of our knowledge, there have been no efficient numerical algorithms for band function reconstruction using the interior of $\mathcal{B}_{\text{red}}$. To fill this vacancy, we propose in this paper the use of several efficient sampling algorithms that can be combined with Lagrange interpolants to approximate band functions in $\mathcal{B}_{\text{red}}$. This work is built upon several well-developed sampling algorithms, which are widely used in numerical analysis \cite{blyth2006lobatto,taylor2000algorithm,bos2010computing,salzer1972lagrangian,chen1995approximate}. 
\begin{table}[H]
\begin{center}
\begin{tabular}{ccc|cc}
\toprule
&\multicolumn{2}{c|}{$\mathcal{B}_{\operatorname{red}}$}&\multicolumn{2}{c}{$\partial\mathcal{B}_{\operatorname{red}}$}\\
\midrule
Band number&$\mathbf{k}$&Extrema &$\mathbf{k}$&Extrema\\
\midrule
6 & $(0.719762,0.049867)$& $\text{Maxi}=1.071796$& $(0.690971,0)
$& $\text{Maxi}=1.069602$
\\
\bottomrule
\end{tabular}
\end{center}
\caption{Hexagonal lattice (TE mode): band function's extrema obtained using $\mathcal{B}_{\operatorname{red}}$ and $\partial\mathcal{B}_{\operatorname{red}}$.}
\label{tab1}
\end{table}
\subsection{Main contributions}
 Our main contributions are threefold. On the one hand, we analyze and summarize key regularity properties of band functions, cf. Theorems \ref{thm:piecewise-analytic} and \ref{thm:lipchitz}: First, band functions are piecewise analytic; Second, singularities occur only on branch points, i.e., crossings of adjacent band functions, and at the origin. These two results are proved by showing that $\{\mathbf{k},(\frac{\omega_n(\mathbf{k})}{c})^2\}_{n=1}^\infty$, which are referred to as the Bloch variety \eqref{bloch variety}, are zeros of a real analytic function in $\mathbb{R}^3$. The locations of singular points are confirmed via implicit mapping theorem; Third, we give the formula for calculating the first-order partial derivatives of band functions at non-singularities, and thus point out that the first-order partial derivatives of band functions may be discontinuous at branch points where the dimension of the corresponding eigenspace is greater than one. This formula further reveals that band functions are Lipschitz continuous.
\begin{figure}[H]\label{points exist the maximum or minimum eigenvalue}
\centering
\subfigure[Square lattice: TE]{\label{minmax2}
\includegraphics[width=.2\textwidth,trim={1cm 0.2cm 0.9cm 0.4cm},clip]{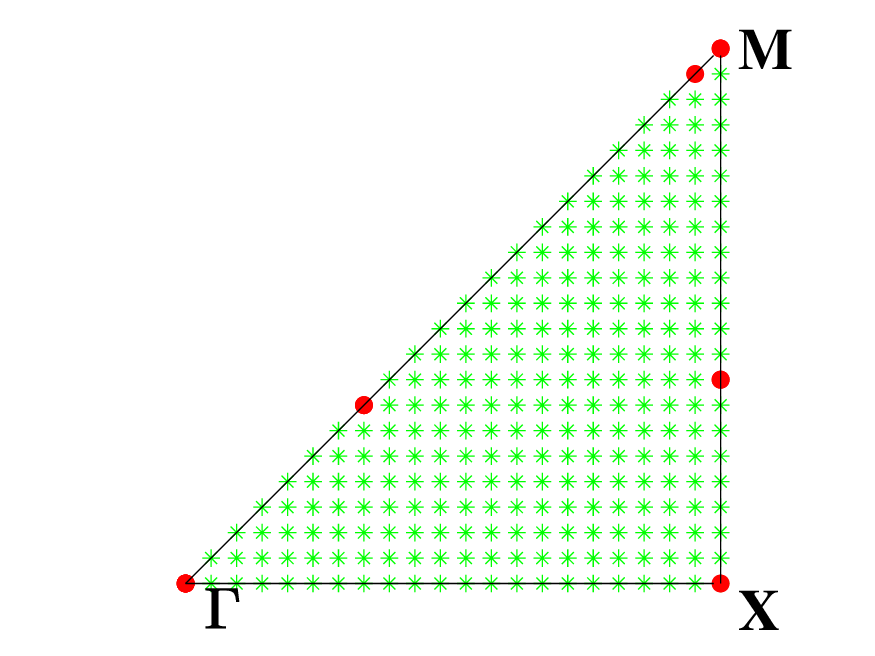}
}%
\subfigure[Square lattice: TM]{\label{minmax1}
\includegraphics[width=.2\textwidth,trim={0.9cm 0.2cm 1cm 0.4cm},clip]{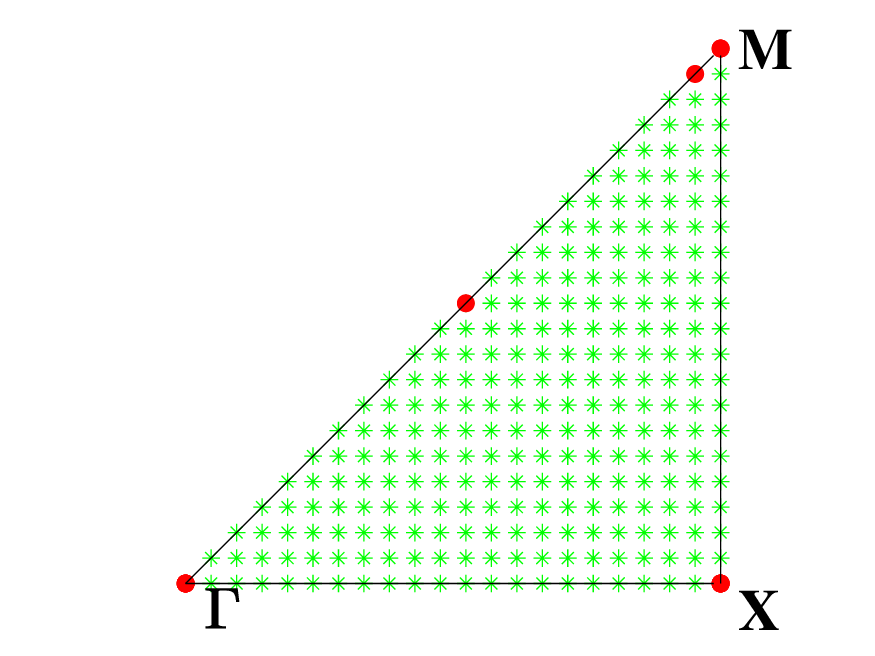}
}%
\subfigure[Hexagonal lattice: TE]{\label{minmax4}
\includegraphics[width=.27\textwidth,trim={1cm 0cm 1cm 1cm},clip]{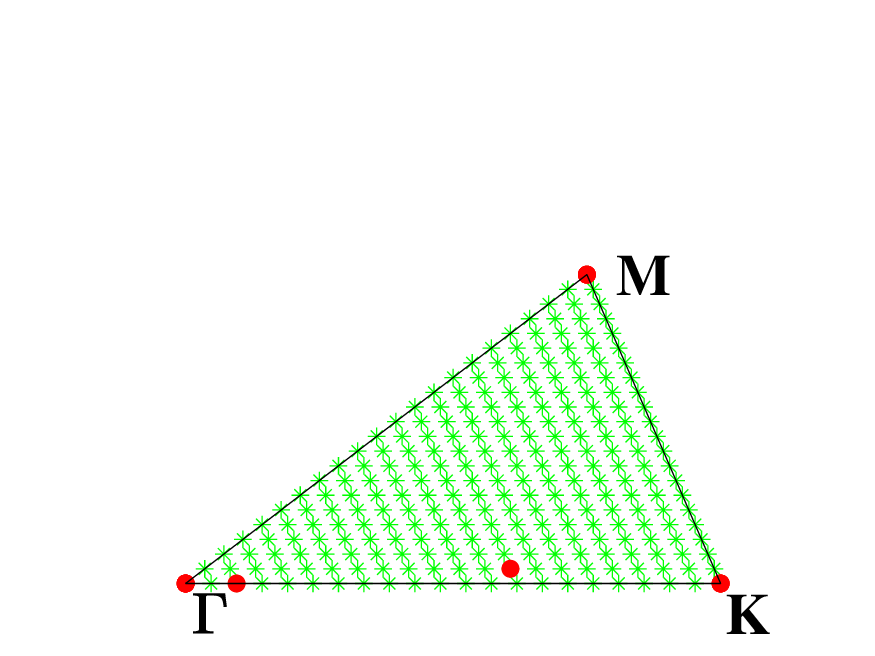}
}%
\subfigure[Hexagonal lattice: TM]{\label{minmax3}
\includegraphics[width=.27\textwidth,trim={1cm 0cm 1cm 1cm},clip]{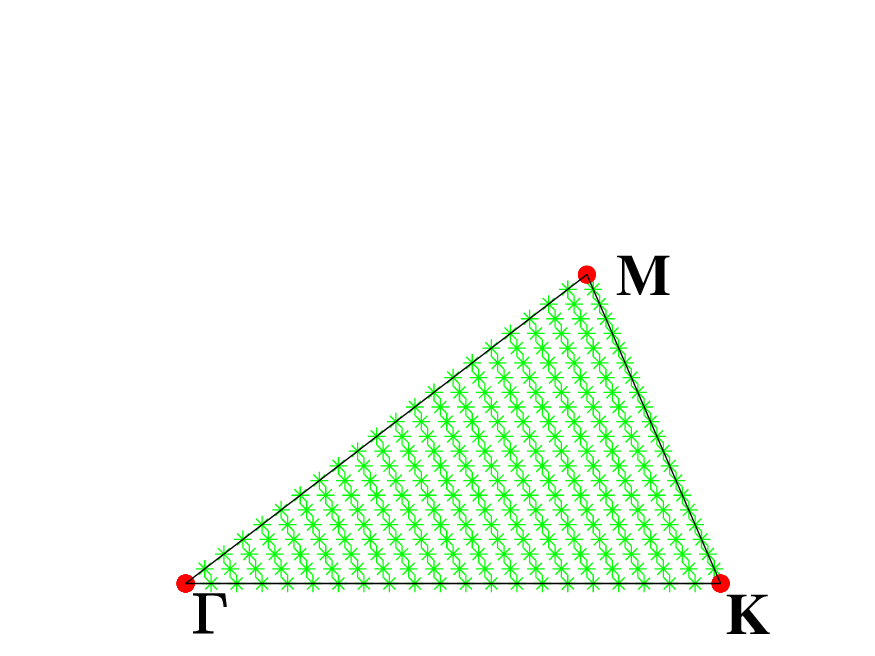}
}%
\centering
\caption{Extrema exist at points marked with red circles, and green points are used to measure the extrema.}\label{minmax}
\end{figure}
On the other hand, our proposed method can be utilized without reordering band functions which is crucial for the Taylor expansion-based method as developed in \cite{klindworth2014efficient}.
Without {reordering}, we are allowed to approximate the first few band functions only, which are the interest of many practical applications, while the aforementioned approach can only approximate the whole band functions simultaneously. Moreover, our method approximates band functions in the whole IBZ, and we can further incorporate it into many other numerical methods, e.g., adaptive FEM \cite{giani2012adaptive} and hp FEM \cite{schmidt2009computation}, to calculate band functions with high accuracy and efficiency.

This paper is organized as follows. In Section \ref{sec:preli}, we describe the time harmonic Maxwell equations involved in band structure calculation and elaborate on the derivation of the parameterized Helmholtz eigenvalue problem through Bloch's theorem as well as the two modes in 2D PhCs.
Section \ref{sec:property} is concerned with the smoothness of the band functions, based upon which we introduce the numerical schemes to reconstruct these band functions in Section \ref{sec:numerical-scheme}.  Extensive numerical experiments are illustrated in Section \ref{sec:experiment} to support our theoretical findings. Finally, we present in Section \ref{sec:conclusion} conclusions and future work. 

\section{Preliminaries}\label{sec:preli}
To study the propagation of light in photonic crystals, we begin with the macroscopic Maxwell equations. In SI convention, the Maxwell equations are composed of the following four equations \cite{jackson1999classical},
\begin{equation*}
\begin{split}
    \nabla \times \mathbfcal{E}+\frac{\partial{\mathbfcal{B}}}{\partial{t}}&=0,\\
    \nabla \times \mathbfcal{H}-\frac{\partial{\mathbfcal{D}}}{\partial{t}}&=\mathbfcal{J},\\
    \nabla \cdot \mathbfcal{D}&=\rho,\\
    \nabla \cdot\mathbfcal{B}&=0,
\end{split}
\end{equation*}
where $\mathbfcal{E}$, $\mathbfcal{H}$, $\mathbfcal{D}$ and $\mathbfcal{B}$ denote the electric field, magnetic field, electric displacement field and the magnetic induction field, respectively. Moreover, $\mathbfcal{J}$ is the free current density and $\rho$ is the free charge density. Assume that there are no sources of light, then we can set $\mathbfcal{J}=\mathbf{0}$ and $\rho=0$.

Conventionally, it is assumed that photonic materials are linear, non-dispersive, and nonmagnetic, which are isotropic with respect to light propagation. Therefore, the electric displacement field $\mathbfcal{D}$, the electric field $\mathbfcal{E}$, the magnetic induction field $\mathbfcal{B}$, and the magnetic field $\mathbfcal{H}$ obey the following constitutive relations,
\begin{equation*}
\begin{split}  \mathbfcal{D}&=\epsilon_{0}\epsilon\mathbfcal{E},\\
    \mathbfcal{B}&=\mu_{0}\mathbfcal{H}, 
\end{split}
\end{equation*}
where $\epsilon_0$ denotes the vacuum permittivity, $\epsilon\in L^{\infty}(\mathbb{R}^3;\mathbb{R}^+)$ is the relative permittivity, and $\mu_0$ is the vacuum permeability.  

After introducing all of the above assumptions, the Maxwell equations can be formulated as
\begin{subequations}
\begin{align}
\nabla\times\mathbfcal{E}+\mu_{0}\frac{\partial\mathbfcal{H}}{\partial t}=0, \label{11} \\
\nabla\times\mathbfcal{H}-\epsilon_{0}\epsilon\frac{\partial\mathbfcal{E}}{\partial t}=0,\label{22}\\
\nabla\cdot(\epsilon\mathbfcal{E})=0,\label{3}\\
\nabla\cdot\mathbfcal{H}=0.\label{4}
\end{align}
\end{subequations}
Here both $\mathbfcal{E}$ and $\mathbfcal{H}$ are functions of time and space, i.e., $\mathbfcal{E}=\mathbfcal{E}(\mathbf{x},t)$, $\mathbfcal{H}=\mathbfcal{H}(\mathbf{x},t)$  with $\mathbf{x}\in\mathbb{R}^3$ and $t\geq 0$. Since the coefficients in \eqref{11}-\eqref{4} are time-independent, we can seek for time harmonic solutions of the following form,
\begin{align*}
\mathbfcal{E}(\mathbf{x},t)&=\mathbf{E}(\mathbf{x})e^{-i\omega t}\\
\mathbfcal{H}(\mathbf{x},t)&=\mathbf{H}(\mathbf{x})e^{-i\omega t}.
\end{align*}

Together with \eqref{11}-\eqref{4}, we derive the time harmonic Maxwell equations,
\begin{subequations}
\begin{align}
    \nabla\times\mathbf{E}(\mathbf{x})-i\omega\mu_{0}\mathbf{H}(\mathbf{x})=0,\label{harm1}    \\
    \nabla\times\mathbf{H}(\mathbf{x})+i\omega\epsilon_{0}\epsilon(\mathbf{x})\mathbf{E}(\mathbf{x})=0 , \label{harm2}\\  \nabla\cdot\left(\epsilon(\mathbf{x})\mathbf{E}(\mathbf{x})\right)=0,\label{harm3}\\
    \nabla\cdot\mathbf{H}(\mathbf{x})=0,\label{harm4}
\end{align}
\end{subequations}
where $\mathbf{x}\in\mathbb{R}^3$ and $\omega\geq 0$.
\begin{remark}
Note that $\mathbf{E}$ and $\mathbf{H}$ are complex-valued fields, we only take their real part to obtain the physical fields.
\end{remark}

Now applying the curl operator to \eqref{harm1} and using \eqref{harm2}, we obtain
\begin{equation}\label{E}  \nabla\times\left(\nabla\times\mathbf{E}(\mathbf{x})\right)-\left(\frac{\omega}{c}\right)^2\epsilon(\mathbf{x})\mathbf{E}(\mathbf{x})=0,\quad \mathbf{x}
\in\mathbb{R}^3, 
\end{equation}
where $\epsilon_{0}\mu_{0}=c^{-2}$ with $c$ denoting the vacuum speed of light. Similarly, applying the curl operator to \eqref{harm2} and using \eqref{harm1}, we obtain
\begin{equation}\label{H}
    \nabla\times\left(\left(\frac{1}{\epsilon(\mathbf{x})}\right)\nabla\times\mathbf{H}(\mathbf{x})\right)-\left(\frac{\omega}{c}\right)^2\mathbf{H}(\mathbf{x})=0,\quad \mathbf{x}
\in\mathbb{R}^3. 
\end{equation}

For a given frequency $\omega\geq 0$, we can obtain the spatial fields $\mathbf{E}(\mathbf{x})$ and $\mathbf{H}(\mathbf{x})$ through \eqref{E} and \eqref{H}. In fact, we only need to consider one of the above equations, since we can derive from \eqref{harm1} and \eqref{harm2} that 
\begin{subequations}
\begin{align}
    \mathbf{H}(\mathbf{x})&=-\frac{i}{\omega\mu_{0}}\nabla\times\mathbf{E}(\mathbf{x}),  \label{rel1}\\
    \mathbf{E}(\mathbf{x})&=\frac{i}{\omega\epsilon_{0}\epsilon(\mathbf{x})}\nabla\times\mathbf{H}(\mathbf{x}). \label{rel2}
\end{align}
\end{subequations}
\begin{remark}
The two divergence equations \eqref{harm3} and \eqref{harm4} are implicitly satisfied, which can easily be seen by applying the divergence operator to equations \eqref{harm1} and \eqref{harm2}, and considering the fact that $\omega>0$. Hence, now we only focus on the other two of the time harmonic Maxwell equations as long as we drop those ``spurious modes" existing at $\omega=0$.\label{re1.2}
\end{remark}

To summarize, the eigenvalue problems \eqref{E} and \eqref{H} are two crucial parts of studying the propagation of electromagnetic waves in PhCs and our aim is to find the eigenpairs $(\omega,\mathbf{E})$ and $(\omega,\mathbf{H})$ satisfying these two equations, respectively. 

\subsection{Modes in 2D Photonic Crystals}
In this work we focus on two 2D PhCs as depicted in Figure \ref{material}. They have finite extensions in practice. Nevertheless, it is commonly assumed that the material extends infinitely in the plane perpendicular to the columns. Hence, the permittivity $\epsilon$ satisfies $\epsilon(\mathbf{x})=\epsilon(x,y,0)$, for all $\mathbf{x}=(x,y,z)\in \mathbb{R}^{3}$. So we can also restrict our electric field and magnetic field to the $xy$ plane, i.e.,
\begin{align*}
    \mathbf{E}(\mathbf{x})&=\mathbf{E}(x,y,0)=(E_{1}(x,y,0),E_{2}(x,y,0),E_{3}(x,y,0)),
    \\\mathbf{H}(\mathbf{x})&=\mathbf{H}(x,y,0)=(H_{1}(x,y,0),H_{2}(x,y,0),H_{3}(x,y,0)),
\end{align*}
for all $\mathbf{x}=(x,y,z)\in \mathbb{R}^{3}$. Then a straightforward calculation leads to
\begin{align}
\nabla\times \mathbf{E}&=\left(\frac{\partial E_{3}}{\partial y}\right)\mathbf{i}+\left(-\frac{\partial E_{3}}{\partial x}\right)\mathbf{j}+\left(\frac{\partial E_{2}}{\partial x}-\frac{\partial E_{1}}{\partial y}\right)\mathbf{k},\label{tuidao1}\\
\nabla\times(\nabla\times \mathbf{E})
&=\left(\frac{\partial ^{2} E_{2}}{\partial y \partial x}-\frac{\partial ^{2} E_{1}}{\partial y^2}\right)\mathbf{i}+\left(-\frac{\partial ^{2} E_{2}}{\partial x ^2}+\frac{\partial ^{2} E_{1}}{\partial x \partial y}\right)\mathbf{j}+\left(-\frac{\partial^2 E_{3}}{\partial x^2}-\frac{\partial ^2 E_{3}}{\partial y^2}\right)\mathbf{k}, \label{tuidao2}\\
\nabla\times \mathbf{H}&=\left(\frac{\partial H_{3}}{\partial y}\right)\mathbf{i}+\left(-\frac{\partial H_{3}}{\partial x}\right)\mathbf{j}+\left(\frac{\partial H_{2}}{\partial x}-\frac{\partial H_{1}}{\partial y}\right)\mathbf{k},\label{tuidao3}\\
\nabla\times\left(\frac{1}{\epsilon(\mathbf{x})} \nabla\times \mathbf{H}\right)
&=\left(\frac{\partial} {\partial y }\frac{1}{\epsilon(\mathbf{x})}\frac{\partial H_{2}}{\partial x}-\frac{\partial} {\partial y}\frac{1}{\epsilon(\mathbf{x})}\frac{\partial H_{1}}{\partial y}\right)\mathbf{i} 
+\left(-\frac{\partial} {\partial x }\frac{1}{\epsilon(\mathbf{x})}\frac{\partial H_{2}}{\partial x}+\frac{\partial} {\partial x}\frac{1}{\epsilon(\mathbf{x})}\frac{\partial H_{1}}{\partial y}\right)\mathbf{j} + \nonumber\\
&\left(-\frac{\partial} {\partial x }\frac{1}{\epsilon(\mathbf{x})}\frac{\partial H_{3}}{\partial x}-\frac{\partial} {\partial y}\frac{1}{\epsilon(\mathbf{x})}\frac{\partial H_{3}}{\partial y}\right)\mathbf{k}.\label{tuidao4}
\end{align}
Plugging \eqref{tuidao2} into \eqref{E}, we obtain
\begin{align*}
-\Delta E_3(\mathbf{x})-\left(\frac{\omega}{c}\right)^2 \epsilon(\mathbf{x})E_3(\mathbf{x})=0,\quad \mathbf{x}
\in\mathbb{R}^2.
\end{align*}
Here and in the following, we consider $\mathbf{x}\in \mathbb{R}^2$. Combining with \eqref{tuidao1}, \eqref{rel1} implies 
\begin{equation*}
\begin{aligned}
    H_1(\mathbf{x})&=-\frac{i}{\omega \mu _0}\frac{\partial}{\partial y}E_3(\mathbf{x}),\\
    H_2(\mathbf{x})&=\frac{i}{\omega \mu _0}\frac{\partial}{\partial x}E_3(\mathbf{x}).
\end{aligned}    
\end{equation*}
Analogously, we derive
\begin{align*}
    -\nabla \cdot \frac{1}{\epsilon(\mathbf{x})}\nabla H_3(\mathbf{x})-\left(\frac{\omega}{c}\right)^2 H_3(\mathbf{x})=0,\quad \mathbf{x}
\in\mathbb{R}^2.
\end{align*}
Following \eqref{tuidao3} and \eqref{rel2}, we get
\begin{equation*}
\begin{aligned}
    E_1(\mathbf{x})&=\frac{i}{\omega \epsilon _0 \epsilon(\mathbf{x})}\frac{\partial}{\partial y}H_3(\mathbf{x}),\\
    E_2(\mathbf{x})&=-\frac{i}{\omega \epsilon _0 \epsilon(\mathbf{x})}\frac{\partial}{\partial x}H_3(\mathbf{x}),
    \end{aligned}
\end{equation*}
The above derivation yields two scalar eigenvalue problems and we also deduce that the components of the electric field and magnetic field are not independent. Indeed, $H_1,H_2$ are related to $E_3$, and $E_1,E_2$ are related to $H_3$. Thus, we can classify the electromagnetic waves in terms of whether $E_3$ or $H_3$ equals to zero. which is often referred to as Transverse Electric (TE) mode and Transverse Magnetic (TM) mode respectively. In other words, in TE mode, the magnetic field is directed along the $z$ axis and the electric field is perpendicular to this axis, while TM mode consists of electric field along $z$ axis and magnetic field perpendicular to the $z$ axis.

To conclude, the eigenvalue problems in 2D PhCs reduce to
\begin{align}
    -\nabla \cdot \frac{1}{\epsilon(\mathbf{x})}\nabla H(\mathbf{x})-\left(\frac{\omega}{c}\right)^2 H(\mathbf{x})&=0,\quad \mathbf{x}
\in\mathbb{R}^2&&\textbf{(TE mode)},\label{TE}\\
    -\Delta E(\mathbf{x})-\left(\frac{\omega}{c}\right)^2 \epsilon(\mathbf{x})E(\mathbf{x})&=0,\quad \mathbf{x}
\in\mathbb{R}^2&& \textbf{(TM mode)}.\label{TM}
\end{align}

\subsection{Bloch's theorem}

2D PhCs possess a discrete translational symmetry \cite{joannopoulos2008molding}, i.e., there exist two linearly independent vectors such that the permittivity $\epsilon(\mathbf{x})$ satisfies
\begin{equation*}
    \epsilon(\mathbf{x}+c_1\mathbf{a}_1+c_2\mathbf{a}_2)=\epsilon(\mathbf{x})
    \text{ for all }\mathbf{x}\in\mathbb{R}^2\text{ and }c_1,c_2\in\mathbb{R}.
\end{equation*}
The primitive lattice vectors, denoted as $\mathbf{a}_1,\mathbf{a}_2$, are the shortest possible vectors that satisfy this condition. The square lattice 
in Figure \ref{lattice} has primitive lattice vectors $\mathbf{a}_i = a{e}_i$ for $i=1,2$, and the hexagonal lattice in Figure \ref{lattice3} has primitive lattice vectors $\mathbf{a}_1 = \frac{a}{4}({e}_1+\sqrt{3}{e}_2)$ and $\mathbf{a}_2 = \frac{a}{4}({e}_1-\sqrt{3}{e}_2)$. Here, $(e_i)_{i=1,2}$ is the canonical basis in $\mathbb{R}^2$. The reciprocal lattice vectors $\mathbf{b}_1$ and $\mathbf{b}_2$ are defined by the property $\mathbf{b}_i \cdot \mathbf{a}_j=2\pi\delta_{ij}$ for $i,j=1,2$.

Bloch's theorem \cite{kittel2018introduction} states that in periodic crystals, wave functions take the form of a plane wave modulated by a periodic function. Mathematically, they can be written as
 \begin{align*}  \Psi(\mathbf{x})=e^{i\mathbf{k}\cdot\mathbf{x}}u(\mathbf{x}),
 \end{align*}
where $\Psi$ is the wave function, $u(\mathbf{x})$ is a periodic function with the same periodicity as the crystal lattice and $\mathbf{k}$ is the wave vector. Functions of this form are known as Bloch functions or Bloch states. 
\begin{proposition}
The periodic condition of $u(\mathbf{x})$ implies that each Bloch state can be determined by its values within the unit cell $\Omega$, which is spanned by the primitive lattice vectors, as illustrated in the left of Figures \ref{lattice} and \ref{lattice3}. 
\end{proposition}

\begin{proposition}\label{IBZ propertity}
In the two-dimensional case, the wave vector $\mathbf{k}$ can be restricted into the so-called (first) Brillouin zone $\mathcal{B}$, which is the region closer to a certain lattice point than to any other lattice points in the reciprocal lattice. In addition, note that in some cases, further utilization of symmetry can even restrict $\mathbf{k}$ to the triangular irreducible Brillouin zone $\mathcal{B}_{\operatorname{red}}$ \cite{joannopoulos2008molding}. The Brillouin zones and irreducible Brillouin zones of the square lattice and the hexagonal lattice are depicted on the right of Figures \ref{lattice} and \ref{lattice3}.
\end{proposition}
Bloch's theorem implies the solutions to \eqref{TM} and \eqref{TE} can be expressed as $E(\mathbf{x}) =e^{i\mathbf{k}\cdot \mathbf{x}}u_1(\mathbf{x})$ and  $H(\mathbf{x})=e^{i\mathbf{k}\cdot \mathbf{x}}u_2(\mathbf{x})$ for some periodic functions $u_1(\mathbf{x})$ and $u_1(\mathbf{x})$ in the unit cell $\Omega$. Consequently, \eqref{TM} and \eqref{TE} are reduced to
\begin{subequations}
\begin{align}  
-(\nabla+i\mathbf{k})\cdot\left(\frac{1}{\epsilon(\mathbf{x})}(\nabla+i\mathbf{k}) u_2(\mathbf{x})\right)-\left(\frac{\omega}{c}\right)^{2} u_2(\mathbf{x})&=0,\quad \mathbf{x}\in\Omega && \textbf{(TE mode)},\label{TE2}\\
   -(\nabla+i\mathbf{k})\cdot\left((\nabla+i\mathbf{k}) u_1(\mathbf{x})\right)-\left(\frac{\omega}{c}\right)^{2}\epsilon(\mathbf{x})u_1(\mathbf{x})&=0,\quad \mathbf{x}\in\Omega&& \textbf{(TM mode)},\label{TM2} 
\end{align}
\end{subequations}
where $\mathbf{x}\in\Omega\subset\mathbb{R}^2$, $\mathbf{k}$ varies in the Brillouin zone, and $u_i(\mathbf{x})$ satisfies
 the periodic boundary conditions $u_i(\mathbf{x})=u_i(\mathbf{x}+\mathbf{a_j})$ with $\mathbf{a_j}$ being the primitive lattice vector for $i,j=1,2$.

Now we consider both modes simultaneously by
\begin{equation}\label{both}
    -(\nabla+i\mathbf{k})\cdot \alpha(\mathbf{x})(\nabla+i\mathbf{k})u(\mathbf{x})-\lambda\beta(\mathbf{x})u(\mathbf{x})=0, \quad \mathbf{x}\in\Omega, 
\end{equation}
with $\mathbf{x}\in\Omega\subset\mathbb{R}^2$, $\mathbf{k} \in \mathcal{B}$, and $\lambda =\left(\frac{\omega}{c}\right)^2$. In TE mode, $u$ describes the magnetic field $H$ in $z$-direction and the coefficients $\alpha(\mathbf{x})$ and $\beta(\mathbf{x})$ are
\begin{align*}
    \alpha(\mathbf{x}):=\frac{1}{\epsilon(\mathbf{x})}, \quad\beta(\mathbf{x}):=1.
 \end{align*}
Similarly, in TM mode, $u$ describes the electric field $E$ in $z$-direction and the coefficients $\alpha(\mathbf{x})$ and $\beta(\mathbf{x})$ are 
\begin{align*}   \alpha(\mathbf{x}):=1,\quad\beta(\mathbf{x}):=\epsilon(\mathbf{x}). 
\end{align*}

\section{Regularity of band functions and eigenfunctions}\label{sec:property}
To further analyze the properties of band functions, we first define some function spaces. Let $L^2_{\beta}(\Omega)$ denote the space of weighted square integrable
functions equipped with the norm,

 \begin{align*}
\|f\|_{L^2_{\beta}(\Omega)}:=\left(\int_{\Omega} \beta(\mathbf{x})|f(\mathbf{x})|^2\mathrm{d}\mathbf{x}\right)^{\frac{1}{2}}.
 \end{align*}
 
Let $H^1(\Omega)\subset L^2_{\beta}(\Omega)$ with square integrable gradient be equipped with the standard $H^1$ norm. $H^1_{\pi}(\Omega)\subset H^1(\Omega)$ is composed of functions with periodic boundary conditions on $\partial\Omega$. Moreover, let 
\begin{align*}
H^1_{\pi}(\Omega,\Delta,\alpha):=\left\{v\in H^1_{\pi}(\Omega): \Delta v\in L^2(\Omega), \alpha\partial_{\mathbf{n}_L}v|_{L}=-\alpha\partial_{\mathbf{n}_R}v|_{R}\text{ and } \alpha\partial_{\mathbf{n}_T}u|_{T}=-\alpha\partial_{\mathbf{n}_B}u|_{B}
\right\}
\end{align*}
with $\partial_{\mathbf{n}_L},\partial_{\mathbf{n}_R},\partial_{\mathbf{n}_B},\partial_{\mathbf{n}_T}$ denoting the outward normal derivatives on the left, right, bottom and top boundaries of $\Omega$, respectively. 

Bloch's theorem expands the original operator $\mathcal{L}:=-\frac{1}{\beta(\mathbf{x})}\nabla \cdot \alpha(\mathbf{x})\nabla$ defined on Sobolev space $H^2(\mathbb{R}^2)$ into a new set of operators $\mathcal{L}_{\mathbf{k}}:=-\frac{1}{\beta(\mathbf{x})}(\nabla+i\mathbf{k}) \cdot \alpha(\mathbf{x})(\nabla+i\mathbf{k})$ defined on $H^1_{\pi}(\Omega,\Delta,\alpha)$.
The following theorem proved in \cite{glazman1965direct} represents the spectrum of the operator $\mathcal{L}$ using that of $\mathcal{L}_{\mathbf{k}}$.
\begin{theorem}\label{thm:spectrum}
For all $\mathbf{k}\in \mathcal{B}$, $\mathcal{L}_{\mathbf{k}}$ has a non-negative discrete spectrum. We can enumerate these eigenvalues in a non-decreasing manner and repeat according to their finite multiplicities as
\begin{align*}  0\leq\lambda_1(\mathbf{k})\leq\lambda_2(\mathbf{k})\leq\cdots \leq \lambda_n(\mathbf{k})\leq\cdots\leq \infty.
\end{align*}
$\{\lambda_n(\mathbf{k})\}_{n=1}^{\infty}$ is an infinite sequence with $\lambda_n(\mathbf{k})$ being a continuous function with respect to the wave vector $\mathbf{k}$ and $\lambda_n(\mathbf{k})\to \infty$ when $n \to \infty$.
Moreover, the spectrum $\sigma(\mathcal{L})$ of the operator $\mathcal{L}$ is connected to the spectrum $\sigma(\mathcal{L}_{\mathbf{k}})$ of the operators $\mathcal{L}_{\mathbf{k}}$ through
\begin{align*}
    \sigma(\mathcal{L})=\bigcup _{\mathbf{k}\in \mathcal{B}}\sigma(\mathcal{L}_{\mathbf{k}}).
\end{align*} \label{theo3.1}
\end{theorem}

In view that $\lambda_n(\mathbf{k})=\left(\frac{\omega_n(\mathbf{k})}{c}\right)^2$, Theorem \ref{thm:spectrum} implies that each band function $\omega_n(\mathbf{k})$ is continuous for all band number $n\in \mathbb{N}^+$.

Next we introduce one of the most important properties of band functions, which lays the main foundation for our proposed interpolation method. 
\begin{theorem}[Piecewise analyticity of the band functions]\label{thm:piecewise-analytic}
2D PhCs band functions are piecewise analytic in $\mathcal{B}$. In specific, each band function $\omega_n(\mathbf{k})$ is analytic in $\mathcal{B}\backslash{X}
_n$, where $X_n\subset \mathcal{B}$  consists of branch points and the origin, which has zero Lebesgue measure.
\end{theorem}
\begin{proof}
Our proof is mainly based upon the analyticity of the Bloch variety that was proved in \cite[Theorem 4.4.2]{kuchment1993floquet}. For the sake of completeness, we repeat it in Theorem \ref{kuchment1}. Our proof is inspired by the procedure used in \cite{wilcox1978theory}, wherein the Bloch wave of Schr\"{o}dinger equation with periodic potential was considered. 

The Bloch variety is defined as 
\begin{equation}\label{bloch variety}
    B(\mathcal{L}_{\mathbf{k}})=\left\{(\mathbf{k},\lambda)\in \mathbb{R}^3:\mathcal{L}_{\mathbf{k}}u=\lambda u\text{ admits a nonzero function } u\in H^1_{\pi}(\Omega,\Delta,\alpha) 
    \right\}.
\end{equation}
Theorem \ref{kuchment1} implies that there is an analytic function $D(\mathbf{k},\lambda)$ on $\mathbb{R}^3$ such that $B(\mathcal{L}_{\mathbf{k}})$ is its set of zeros, i.e., 
\begin{align*}
B(\mathcal{L}_{\mathbf{k}})=\left\{ (\mathbf{k},\lambda)\in \mathbb{R}^3|D(\mathbf{k},\lambda)=0\right\}. 
\end{align*}
We now decompose $B(\mathcal{L}_{\mathbf{k}})$ into two types of sets, where the first type is
\begin{align}
 B^r:= \Big\{
 (\mathbf{k}_0,\lambda_0)\in B(\mathcal{L}_{\mathbf{k}}):
 &\frac{\partial^{m-1}D}{\partial \lambda^{m-1}}=0 \text{ in a neighborhood of }(\mathbf{k}_0,\lambda_0)\nonumber\\
 \text{ and }&\frac{\partial^mD}{\partial \lambda^m}|_{(\mathbf{k}_0,\lambda_0)}\neq0 \text{ for some }m\in\mathbb{N}^+\Big\}, \label{defn:br}
 \end{align}
and the second type is $B^s:=B(\mathcal{L}_{\mathbf{k}})\backslash B^r$. 

Note that $\lambda_n(\mathbf{k})\to \infty$ when $n\to \infty$ for all $\mathbf{k}\in \mathcal{B}$. To the aim of defining a bounded subset of $B(\mathcal{L}_{\mathbf{k}})$, we introduce
 \begin{align*}
 B_n:=\bigcup_{l=1}^n\left\{(\mathbf{k},\lambda_l(\mathbf{k}))|(\mathbf{k},\lambda_l(\mathbf{k}))\in B(\mathcal{L}_{\mathbf{k}})\right\}.
 \end{align*}
 Analogously, this leads to bounded subsets of $B^r$ and $B^s$ given by 
 \begin{align*}
 B_n^r:=B_n\cap B^r\quad\text{ and }\quad B_n^s:=B_n\cap B^s.
 \end{align*}
 Furthermore, let $X^s_n$ be the projection of $B_n^s$ onto $\lambda=0$ defined by 
 \begin{align}\label{branch point}
 X^s_n:=\left\{\mathbf{k}\in\mathcal{B}|(\mathbf{k},\lambda)\in B_n^s\text{ for some } \lambda\right\}.
 \end{align}
The definition of $B_n^s$ implies that $X^s_n$, also known as the branch point set, consists of points of degeneracy for the first $n$ bands, i.e., where certain band functions intersect. Moreover, for any  $(\mathbf{k}_0,\lambda_0)\in B_n^s$, there is an integer $m\in\mathbb{N}^+$ such that $(\mathbf{k}_0,\lambda_0)\in \{(\mathbf{k},\lambda):D(\mathbf{k},\lambda)=0\}\cap\{(\mathbf{k},\lambda):\frac{\partial^mD}{\partial \lambda^m}=0\}$, which is a one-dimensional variety. Hence $X^s_n$, the projection of $B_n^s$ onto the hyperplane $\lambda=0$, is a subset of $\mathcal{B}$ with Lebesgue measure zero.  Let $\mathbf{k}_0\in \mathcal{B}\backslash X^s_n$, i.e., $(\mathbf{k}_0,\lambda_l(\mathbf{k}_0))\in B_n^r$ for $l=1,\cdots,n$. Due to the implicit mapping theorem for analytic functions \cite{kaup2011holomorphic}, there is a neighborhood $N(\mathbf{k}_0)$ in which $\lambda_l(\mathbf{k})$ is the unique analytic solution of $\frac{\partial^{m-1}D}{\partial \lambda^{m-1}}=0$ for some $m\in\mathbb{N}^+$. This implies that each positive band function $\omega_l(\mathbf{k})=c\cdot\sqrt{\lambda_l(\mathbf{k})}$ is analytic in $\mathcal{B}\backslash X^s_n$. Besides, it is well known that only at the origin the first eigenvalue equals to 0. Thus, we have $X_n=X_n^s\cup\{\mathbf{0}\}$.
\end{proof}

\begin{theorem} {\textnormal{(\cite[Theorem 4.4.2]{kuchment1993floquet})}}
Let $L$ be a general periodic elliptic operator in $\mathbb{R}^n$, then the complex Bloch variety of $L$, 
 \begin{align*}
 B(L)=\{(\mathbf{k},\lambda)\in \mathbb{C}^n\times\mathbb{C}|\text{ the equation }Lu=\lambda u\text{ has a non-zero Bloch function with }\mathbf{k}\},
 \end{align*}
is the set of all zeros of an entire function on $\mathbb{C}^{n+1}$. 
\label{kuchment1}
\end{theorem}

\begin{theorem}[mentioned in \cite{klindworth2014efficient} and proved by perturbation theory \cite{kato2013perturbation}]\label{piece1}
If we only consider one component $k_i$ of the vector $\mathbf{k}=(k_1,k_2)$, then all the positive band functions can be reordered when crossing the branch points such that they are analytic functions with respect to $k_i$.
\end{theorem}

Theorem \ref{thm:piecewise-analytic} states that band functions are piecewise analytic with possible singularities occurring on branch points and at the origin. Theorem \ref{piece1} further shows the analytic continuation of band functions in one variable through branch points. In the following, we will discuss the properties of these singular points and the smoothness of band functions. To this end, we first study the limit of eigenfunctions along any band functions. 

The variational formulation of \eqref{both} is: for $\mathbf{k}\in\mathcal{B}$, find non-trivial eigenpair $(\lambda,u) \in (\mathbb{R},H^1_{\pi}(\Omega))$ satisfying
\begin{equation}\label{variational}
\left\{
\begin{aligned}   \int_{\Omega}\alpha(\nabla+i\mathbf{k})u\cdot(\nabla-i\mathbf{k})\Bar{v}-\lambda\beta u\Bar{v}\mathrm{d}x&=0, \text{ for all }v \in H^1_{\pi}(\Omega)\\
\|u\|_{L^2_{\beta}(\Omega)}&=1.
    \end{aligned}
\right.
\end{equation}
Using the sesquilinear forms
\begin{align*}
    a(u,v)&:=\int_{\Omega}\alpha (\nabla+i\mathbf{k})u\cdot(\nabla-i\mathbf{k})\Bar{v}\,\mathrm{d}x,\\
    b(u,v)&:=\int_{\Omega}\beta u\Bar{v}\,\mathrm{d}x,
\end{align*}
\eqref{variational} reads: for $\mathbf{k}\in\mathcal{B}$, find non-trivial eigenpair $(\lambda,u) \in (\mathbb{R},H^1_{\pi}(\Omega))$ such that
\begin{equation}\label{simply}
    \left\{
\begin{aligned}
a(u,v)&=\lambda b(u,v), \text{ for all }v \in H^1_{\pi}(\Omega) \\
b(u,u)&=1.
\end{aligned}
\right.
\end{equation}
Suppose that we fix a particular $\mathbf{k}$, then we denote the eigenspace of one of the corresponding eigenvalues $\lambda$ as $E(\lambda)$ and let $\Tilde{\mathcal{L}}_{\mathbf{k}}$ be the operator defined on the quotient space $\Tilde{H}_{\pi}^1(\Omega,\Delta,\alpha):=H_{\pi}^1(\Omega,\Delta,\alpha)/E(\lambda)$ with the same form as $\mathcal{L}_{\mathbf{k}}$. Then the unique different value between the resolvent sets of $\Tilde{\mathcal{L}}_{\mathbf{k}}$ and $\mathcal{L}_{\mathbf{k}}$ is $\lambda$. More precisely, they have the relation $\rho(\Tilde{\mathcal{L}}_{\mathbf{k}})=\rho({\mathcal{L}}_{\mathbf{k}})\cup\{ \lambda\}$, i.e., $\lambda$ is in the resolvent set of $\Tilde{\mathcal{L}}_{\mathbf{k}}$.
In the following, we investigate the regularity of the eigenfunctions.
\begin{theorem}[Continuity of the eigenfunctions in $\mathcal{B}\backslash X_n$]\label{thm:eigenfun}
Let $\omega_n(\mathbf{k})$ be the $n$th band function for $n\in\mathbb{N}^+$, and let $u(x;\mathbf{k})$ be one of its corresponding normalized eigenfunctions if the corresponding eigenvalue $\lambda_n(\mathbf{k})$ has multiplicity larger than one,
then the following statements hold,
\begin{itemize}
\item[(i)] $u(x;\mathbf{k})$ can be defined such that $u(x;\mathbf{k})$ is continuous with respect to the wave vector $\mathbf{k}$ for $\mathbf{k}\notin X_n$. 
\item[(ii)] If $\mathbf{k}\in X_n$ and the multiplicity of the eigenvalue $\lambda_n(\mathbf{k})$ is $M\geq 2$ with $\{u_q(\mathbf{x};\mathbf{k})\}_{q=1}^M$ being its associated eigenfunctions, then there may be a normalized eigenfunction $u(x;\mathbf{k})$ that admits jump discontinuity, i.e., there is $\{c_{i\pm}^q\}_{q=1}^M\subset\mathbb{C}$, satisfying
\begin{align*}
\lim_{\delta\to 0}u(\mathbf{x};\mathbf{k}\pm \delta e_i)&= \sum_{q=1}^M c_{i\pm}^q u_q(\mathbf{x};\mathbf{k})\text{ and } \left\|\sum_{q=1}^M c_{i\pm}^q u_q(\mathbf{x};\mathbf{k})\right\|=1,\\
\text{but } \sum_{q=1}^M c_{i+}^q u_q(\mathbf{x};\mathbf{k})&\neq \sum_{q=1}^M c_{i-}^q u_q(\mathbf{x};\mathbf{k}).
\end{align*} 
Here, $\delta>0$ is a parameter such that $\mathbf{k}\pm\delta e_i\in\mathcal{B}$ for $i=1,2$ and $(e_i)_{i=1,2}$ is the canonical basis in $\mathbb{R}^2$. 
\end{itemize}
\end{theorem}
\begin{proof}
For the sake of simplicity, we drop the band number for the moment. Let $\omega(\mathbf{k})$ be the band function with $u(\mathbf{x};\mathbf{k})$ being the associated eigenfunction for any $\mathbf{k}\in \mathcal{B}$. Let the error function be 
\begin{align*}
e_i(\delta):=u(\mathbf{x};\mathbf{k}+\delta e_i)-u(\mathbf{x};\mathbf{k}),
\end{align*}
with $\delta>0$ being a parameter such that $\mathbf{k}+\delta e_i\in\mathcal{B}$. 

By an application of \eqref{both} and \eqref{simply}, we deduce that the error function $e_i(\delta)$ satisfies the strong formulation
\begin{equation}\label{ei 1}
  \left(\mathcal{L}_\mathbf{k}-\lambda(\mathbf{k})\right)e_i(\delta)=
  f(\mathbf{x};\delta),
\end{equation}
where 
\begin{align*}
    f(\mathbf{x};\delta)=
    &\left(\lambda(\mathbf{k}+\delta e_i)-\lambda(\mathbf{k})\right)u(\mathbf{x};\mathbf{k}+\delta e_i)
    -\delta(2k_i+\delta)\frac{1}{\epsilon(\mathbf{x})}u(\mathbf{x};\mathbf{k}+\delta e_i)\\
    &+\delta i\frac{1}{\beta(\mathbf{x})}\frac{\partial}{\partial x_i}\left(\alpha(\mathbf{x})u(\mathbf{x};\mathbf{k}+\delta e_i)\right)+\delta i\frac{1}{\epsilon(\mathbf{x})}\frac{\partial}{\partial x_i}u(\mathbf{x};\mathbf{k}+\delta e_i).
\end{align*}
The corresponding weak formulation is 
\begin{equation*}
a(e_i(\delta),v)-\lambda(\mathbf{k}) b(e_i(\delta),v)=g(v;\mathbf{k},\lambda,u), 
\text{ for all } v \in H^1_{\pi}(\Omega), 
\end{equation*}
with $g(v;\mathbf{k},\lambda,u)$ being 
\begin{equation*}
\begin{aligned}
g(v;\mathbf{k},\lambda,u)=&\left(\lambda(\mathbf{k}+\delta e_i)-\lambda(\mathbf{k})\right)b(u(\mathbf{x};\mathbf{k}+\delta e_i),v)\\
&-\delta(2k_i+\delta)m_{\alpha}(u(\mathbf{x};\mathbf{k}+\delta e_i),v)-\delta m_{\alpha i}(u(\mathbf{x};\mathbf{k}+\delta e_i),v),
\end{aligned}
\end{equation*}
and 
\begin{subequations}\label{eq:m-def}
\begin{align}
&m_{\alpha i}(u,v)=\int_{\Omega}i\alpha\left(u
\frac{\partial\Bar{v}}{\partial x_i}-\Bar{v} \frac{\partial u}{\partial x_i}\right)\,\mathrm{d}\mathbf{x}, \quad i=1,2,\\
&m_{\alpha}(u,v)=\int_{\Omega}\alpha u \Bar{v}\,\mathrm{d}\mathbf{x}.
\end{align}
\end{subequations}
Following the Fredholm–Riesz–Schauder theory \cite{sauter2010boundary}, we can derive that for a given $\mathbf{k}$ and a corresponding eigenvalue $\lambda(\mathbf{k})$, Problem \eqref{ei 1} has a unique solution $e_i(\delta)=(\Tilde{\mathcal{L}}_{\mathbf{k}}-\lambda \mathbf{I})^{-1}f(\mathbf{x};\delta)$ in the quotient space $\Tilde{H}_{\pi}^1(\Omega,\Delta,\alpha)$ which is bounded by
\begin{equation}\label{eq:xxxxx}
    \Vert e_i(\delta)\Vert_{\Tilde{H}_{\pi}^1(\Omega)}\leq 
    \Vert (\Tilde{\mathcal{L}}_{\mathbf{k}}-\lambda \mathbf{I})^{-1}
    \Vert_{L^2_{\beta}(\Omega)\to \Tilde{H}_{\pi}^1(\Omega,\Delta,\alpha)}\cdot\Vert f(\mathbf{x};\delta)\Vert_{L^2_{\beta}(\Omega)}.
\end{equation}

Note that $\lambda(\mathbf{k})$ is a continuous function and $u(\mathbf{x};\mathbf{k})\in H_{\pi}^1(\Omega,\Delta,\alpha)$, which lead to $\|f(\mathbf{x};\delta)\|_{L^2_{\beta}(\Omega)}\to 0$ as $\delta\to 0$. Together with \eqref{eq:xxxxx}, the error function $e_i(\delta)$ converges to some function in $E(\lambda(\mathbf{k}))$ as $\delta\to 0$. In a similar manner, we can show that $u(\mathbf{x};\mathbf{k}-\delta e_i)-u(\mathbf{x};\mathbf{k})$ converges to some function in $E(\lambda(\mathbf{k}))$ as $\delta\to 0$, i.e., 
\begin{align}\label{eq:111111}
\lim_{\delta\to 0}u(\mathbf{x};\mathbf{k}\pm \delta e_i)\in E(\lambda(\mathbf{k})).
\end{align}
Next, we discuss case by case whether a given wave vector $\mathbf{k}$ belongs to the singular set $X_n$ as defined in  
Theorem \ref{thm:piecewise-analytic}.
If $\mathbf{k}\notin X_n$ and the multiplicity of $\lambda$ is $1\leq M\in\mathbb{N}^{+}$, then the definition of $B^r$ \eqref{defn:br} implies the existence of a neighborhood $N(\mathbf{k})$ such that for any $\mathbf{k}\pm\delta e_i\in N(\mathbf{k})$, the eigenvalues  $\lambda(\mathbf{k}\pm\delta e_i)$ have the same multiplicity. Together with \eqref{eq:111111}, this implies the corresponding eigenfunction satisfying 
\begin{align*}
\lim_{\delta\to 0}u(\mathbf{x};\mathbf{k}\pm \delta e_i)&= \sum_{q=1}^M c_{i\pm}^q u_q(\mathbf{x};\mathbf{k}),\\
\left\|\sum_{q=1}^M c_{i\pm}^q u_q(\mathbf{x};\mathbf{k})\right\|_{L^2_{\beta}(\Omega)}&=1.
\end{align*} 
Here, $\{c_{i\pm}^q\}_{q=1}^M\subset\mathbb{C}$ are some constant. Thus, we can always find a combination of these $M$ normalized eigenfunctions $\{u_q(\mathbf{x};\mathbf{k})\}_{q=1}^M$ such that there is a normalized eigenfunction $u(\mathbf{x};\mathbf{k})$, satisfying 
\begin{equation*}
   \lim_{\delta\to 0} u(\mathbf{x};\mathbf{k}+\delta e_i)=\lim_{\delta\to 0} u(\mathbf{x};\mathbf{k}-\delta e_i)=u(\mathbf{x};\mathbf{k}),\quad \text{ for } i=1,2.  
\end{equation*}
When $\mathbf{k}\in X_n$ is a singular point and $\lambda(\mathbf{k})$ has multiplicity $M\geq 2$, there is no such kind of neighborhood $N(\mathbf{k})$ such that for any $\mathbf{k}\pm\delta e_i\in N(\mathbf{k})$, the multiplicity of $\lambda(\mathbf{k}\pm\delta e_i)$ is also $M$ since the Lebesgue measure of $X_n$ vanishes. Consequently, there is no guarantee we can construct such kind of normalized eigenfunctions to ensure the continuity at $\mathbf{k}$. We only have for any normalized eigenfunction $u(\mathbf{x};\mathbf{k}\pm \delta e_i)$ at $\mathbf{k}\pm \delta e_i$,
\begin{equation*}
    \lim_{\delta\to 0}u(\mathbf{x};\mathbf{k}\pm \delta e_i)= \sum_{q=1}^{M} c_{i\pm}^q u_q(\mathbf{x};\mathbf{k}),
\end{equation*}
with some constant $\{c_{i\pm}^q\}_{q=1}^M\subset\mathbb{C}$ such that $\|\sum_{q=1}^M c_{i\pm}^q u_q(\mathbf{x};\mathbf{k})\|_{L^2_{\beta}(\Omega)}=1$ for $i=1,2$. This completes the proof.
\end{proof}
\begin{remark}[Differentiability of the eigenfunctions in $\mathcal{B}\backslash X_n$]
 As is shown in \cite{klindworth2014efficient}, we can further prove that the eigenfunctions corresponding to the $n$th band function can indeed be organized to be continuously differentiable of any degree in $\mathcal{B}\backslash X_n$. However, the properties of eigenfunctions at $X_n$ are not discussed in the aforementioned paper. The authors conjectured that there may exist an ordering of band functions such that the corresponding eigenfunctions are also continuously differentiable at $X_n$. In contrast, Theorem \ref{thm:eigenfun} indicates the possibility of discontinuity of the eigenfunctions in $X_n$. 
 \end{remark}
Now, we are able to prove the regularity of band functions.
\begin{theorem}[Lipschitz continuity of the band functions]\label{thm:lipchitz}
For the case of 2D PhCs, $\omega_n(\mathbf{k})\in Lip(\mathcal{B})\cap \mathring{A}(\mathcal{B})$ for all $n\in\mathbb{N}^+$. Here, $ Lip(\mathcal{B})$ is the space of Lipschitz continuous functions in the Brillouin zone $\mathcal{B}$ and $\mathring{A}(\mathcal{B})$ denotes the space composed of piecewise analytic functions with their singular point sets having a zero Lebesgue measure.
\end{theorem}
\begin{proof}
On the one hand, let $\mathbf{k}\in \mathcal{B}\backslash X_n$, and suppose the multiplicity of the eigenvalue $\lambda_n(\mathbf{k})$ is $M\geq 1$ for some $n\in\mathbb{N}^{+}$. For simplicity, we drop $n$ and $\mathbf{k}$ in the proof. Theorem \ref{thm:eigenfun} guarantees the existence of a normalized eigenfunction $u(\mathbf{x};\mathbf{k})$ which is continuous in a neighborhood of $\mathbf{k}$. Taking the partial derivative with respect to $k_i$ for $i=1,2$ at $\mathbf{k}$ on both sides of \eqref{simply}, this leads to
\begin{equation*}\label{partial1}
    a(\frac{\partial u}{\partial k_i},v)-\lambda b(\frac{\partial u}{\partial k_i},v)=f^{(1)}(v;\mathbf{k},u,\frac{\partial \lambda}{\partial k_i}),
\end{equation*}
with
\begin{align*}
f^{(1)}(v;\mathbf{k},u,\frac{\partial \lambda}{\partial k_i}):=-m_{\alpha i}(u,v)-2k_im_{\alpha}(u,v)+\frac{\partial \lambda}{\partial k_i}b(u,v).
\end{align*}
Here, the bilinear forms $m_{\alpha i}(\cdot,\cdot)$ and $m_{\alpha}(\cdot,\cdot)$ are defined in \eqref{eq:m-def}.

Note that the operator in the equation above $(\mathcal{L}_k-\lambda \mathbf{I})$ has the eigenspace $E(\lambda)$ as its kernel. 
Let $\{u_p\}_{p=1}^M$ be a set of basis in $E(\lambda)$. 
As a consequence of the Fredholm–Riesz–Schauder theory \cite{sauter2010boundary}, adding additional orthogonality with $E(\lambda)$ leads to the well-posedness of the following problem: seeking $\partial_{k_i}u\in H^1_{\pi}(\Omega)$, s.t., 
\begin{equation}\label{one}
    \left\{
\begin{aligned}
 a(\partial_{k_i}u,v)-\lambda b(\partial_{k_i}u,v)&=f^{(1)}(v;\mathbf{k},u,\frac{\partial \lambda}{\partial k_i})&&\text{ for all }v \in H^1_{\pi}(\Omega),\\
 b(\partial_{k_i}u,u_p)&=0&&\text{ for }p=1,\cdots ,M. 
\end{aligned}
\right.
\end{equation}
Let the test function $v:=u$, we obtain  
\begin{equation}\label{self adjoint}
    a(\partial_{k_i}u,u)-\lambda b(\partial_{k_i}u,u)=\overline{a(u,\partial_{k_i}u)-\lambda b(u,\partial_{k_i}u)}=f^{(1)}(u;\mathbf{k},u,\frac{\partial \lambda}{\partial k_i})=0.
\end{equation}
Therefore, $\frac{\partial \lambda}{\partial k_i}$ has to be the solution to the following problem, 
\begin{align*}
f^{(1)}(u;\mathbf{k},u,\frac{\partial \lambda}{\partial k_i})=0. 
\end{align*}
This results in
\begin{equation}\label{derivative}
  \frac{\partial \lambda}{\partial k_i}=2k_i m_\alpha(u,u)+m_{\alpha i}(u,u).
\end{equation}

Consequently, $\frac{\partial \lambda}{\partial k_i}$ is uniformly bounded. 

On the other hand, let $\mathbf{k}_0\in X_n$ and suppose the multiplicity of the eigenvalue $\lambda_n(\mathbf{k_0})$ is $M\geq 1$ for some $n\in\mathbb{N}^{+}$ and let its associated eigenfunctions be $\{u_q(\mathbf{x};\mathbf{k}_0)\}_{q=1}^M$. 
Then a combination of \eqref{derivative} and Theorem \ref{thm:eigenfun} leads to the left and right partial derivatives,  
\begin{equation*}
\begin{aligned}
    \left.\frac{\partial \lambda}{\partial k_i{^{\pm}}}\right|_{\mathbf{k}=\mathbf{k}_0}&=\lim_{\delta\to 0^{\pm}}\left.\left(2k_i m_\alpha(u,u)+m_{\alpha i}(u,u)\right)\right|_{\mathbf{k}=\mathbf{k}_0+\delta e_i}\\
    &=2k_i \sum_{q=1}^M\sum_{p=1}^M c^q_{i\pm}c^p_{i\pm}\left(m_\alpha( u_q(\mathbf{x};\mathbf{k}), u_p(\mathbf{x};\mathbf{k}))+m_{\alpha i}( u_q(\mathbf{x};\mathbf{k}),u_p(\mathbf{x};\mathbf{k}))\right). 
\end{aligned}
\end{equation*}
Here, the complex values $\{c_{i\pm}^q\}_{q=1}^M\subset\mathbb{C}$ satisfy Theorem \ref{thm:eigenfun} for $i=1,2$. Consequently, $\left.\frac{\partial \lambda}{\partial k_i{^{\pm}}}\right|_{\mathbf{k}=\mathbf{k}_0}$ are bounded but may take different values.  

Finally, the first partial derivative of each positive band function at $\mathbf{k}\in \mathcal{B}\backslash X_n$ can be easily derived from the relation $\lambda(\mathbf{k})=\left(\frac{\omega(\mathbf{k})}{c}\right)^2$. Since the first band function vanishes at $\mathbf{k}=0$, $\left.\frac{\partial \omega_1}{\partial k_i{^{+}}}\right|_{\mathbf{k}=0}=\infty$, which is unbounded. Consequently, we proved $\omega_n(\mathbf{k})\in Lip(\mathcal{B})\cap \mathring{A}(\mathcal{B})$ for all $n\in\mathbb{N}^+$, and this completes our proof.      
\end{proof}
\section{Numerical schemes}\label{sec:numerical-scheme}
As shown in the previous section, the band functions of 2D PhCs are real-valued, non-negative, continuous, and piecewise analytic in $\mathcal{B}$. Besides, there is a triangular area within the Brillouin zone $\mathcal{B}$, which is the so-called irreducible Brillouin zone (IBZ) $\mathcal{B}_{\operatorname{red}}$, such that any point in $\mathcal{B}$ can be transformed into $\mathcal{B}_{\operatorname{red}}$ by mirror symmetry or rotational symmetry. This implies that band functions at these points take the same value as the corresponding points in $\mathcal{B}_{\operatorname{red}}$ \cite{joannopoulos2008molding}. Consequently, one only needs to approximate band functions in $\mathcal{B}_{\operatorname{red}}$. Alternatively, one can approximate band functions in the quadrilateral domain $\Tilde{\mathcal{B}}_{\operatorname{red}}$, which is composed of $\mathcal{B}_{\operatorname{red}}$ and its mirror symmetric region along one of its edges.

Based upon the properties of band functions we derived, we exploit in this work the band function reconstruction using Lagrange interpolation. In specific, given a set of $N$ distinct sampling points $\{\mathbf{k}_i\}_{i=1}^N\subset\mathcal{C}$ for $\mathcal{C}:=\mathcal{B}_{\operatorname{red}}$ or $\mathcal{C}:=\Tilde{\mathcal{B}}_{\operatorname{red}}$ and the corresponding band function values $\{\omega(\mathbf{k}_i)\}_{i=1}^N$ for a certain band number, the corresponding Lagrange interpolation $\mathbf{L}\omega$ is a linear combination of the Lagrange polynomials $\{l_i(\mathbf{k})\}_{i=1}^N$ for those sampling points satisfying $l_i(\mathbf{k_j})=\delta_{ij}$ such that it interpolates the data, i.e., 
\begin{align}\label{eq:interpolation}
 \mathbf{L}\omega(\mathbf{k})=\sum_{i=1}^N\omega(\mathbf{k}_i)l_i(\mathbf{k}).
 \end{align}
Here, $\{\omega(\mathbf{k}_i)\}_{i=1}^N$ are derived by solving the eigenvalue problem \eqref{both} numerically. We refer to \cite[Section 2]{canuto2007spectral} for more details on Lagrange interpolation.

 Theoretically, by computing the so-called Lebesgue constant
\begin{equation}\label{optimal}
   \Gamma_{N}(\{\mathbf{k}_i\}_{i=1}^N):=\max_{\mathbf{k}\in \mathcal{C}}\sum_{i=1}^N|l_i(\mathbf{k})|, 
\end{equation}
one can obtain a measure of how close the Lagrange approximation is to the best polynomial approximation by means of \cite[Theorem 15.1]{trefethen2019approximation}
\begin{equation*}
    \sup_{\mathbf{k}\in  \mathcal{C}}|\omega(\mathbf{k})-\mathbf{L}\omega(\mathbf{k})|\leq \left(1+\Gamma_{N}\left(\{\mathbf{k}_i\}_{i=1}^N\right)\right)\inf_{p_n\in P_n(\mathcal{C})}\sup_{\mathbf{k}\in \mathcal{C}}|\omega(\mathbf{k})-p_n(\mathbf{k})|.
\end{equation*}
Here, $P_n(\mathcal{C})$ is the polynomial space with degree at most $n$ and $N:=\dim(P_n(\mathcal{C}))$, where $N=(n+1)^2$ if $\mathcal{C}=\Tilde{\mathcal{B}}_{red}$ and $N=\frac{1}{2}(n+1)(n+2)$ if $\mathcal{C}=\mathcal{B}_{red}$.

Note that the value of the Lebesgue constant $\Gamma_{N}(\{\mathbf{k}_i\}_{i=1}^N)$ depends solely on the nodal set $\{\mathbf{k}_i\}_{i=1}^N$. However, minimizing \eqref{optimal} to obtain the optimal nodal set is computationally challenging for any type of domain $\mathcal{C}\subset\mathbb{R}^{d}$ with $d>1$. This is because the denominator of the Lagrange polynomials vanishes on a subset of $\mathcal{C}$ that the Lebesgue constant $\Gamma_{N}(\{\mathbf{k}_i\}_{i=1}^N)$, as a function of the nodal points $\{\mathbf{k}_i\}_{i=1}^N$, is not continuous in the $2N$-dimensional bounded domain $\mathcal{C}^N$. Besides, $\Gamma_{N}(\{\mathbf{k}_i\}_{i=1}^N)$ is very sensitive to the location of the interpolation points, which makes the minimization procedure subtle. There have been several attempts to generate the optimal nodal sets in the sense that they minimize the Lebesgue constant. For example, Heinrichs directly minimized the Lebesgue constant in a triangle with Fekete points as the initial guess \cite{heinrichs2005improved}. 
Babu{\v{s}}ka minimized the norm of the Lagrange interpolation operator, which yields small Lebesgue constant in a triangle \cite{chen1995approximate}. 
Sommariva provided the approximate Fekete and approximated optimal nodal points in both the square and the triangle by solving the corresponding optimization problems numerically \cite{briani2012computing}. However, how to find the optimal points is still an open question. 

Since in the multivariate case, the optimal Lagrange interpolant is hard to determine, our aim is to find a suitable sampling point set, on which our Lagrange interpolation performs well when dealing with the band function reconstruction. 

\subsection{Sampling points in triangle}\label{subsec:sampling points-t}
To standardize the problem, 
let the isosceles right triangle $T$ be the reference triangle, 
\[
T:=\{\mathbf{x}=(x,y): 0\leq x \leq 1, 0 \leq y \leq 1-x\}, 
\]
and let $P_n(T)$ be the space of polynomials on $T$ with degree at most $n$,
\begin{equation*}
    P_n(T)=\text{span}\{x^iy^j,\quad i+j\leq n\}.
\end{equation*}

In the following, we introduce several sampling methods on this reference triangle $T$.

\subsubsection*{Mean optimal points} 
The mean optimal nodal set is defined by minimizing a norm related to the Lagrange interpolation operator $\mathbf{L}$, which takes the form \cite{chen1995approximate}
\begin{align}\label{eq:l2}
\Vert \mathbf{L}\Vert _{2}:=\int _T\sum_{i=1}^N|l_i(\mathbf{x})|^2 \mathrm{d}\mathbf{x}. 
\end{align}
Here, $\{l_i(\mathbf{x})\}_{i=1}^N$ is defined in the same way as in \eqref{eq:interpolation}. Compared with \eqref{optimal}, the minimization of \eqref{eq:l2} involves less computational complexity and is thus more favorable. Indeed, suppose $\{p_i(\mathbf{x})\}_{i=1}^N$ is a set of standard orthogonal polynomials on the triangle $T$ with degree at most $n$, i.e., $\int _T p_i(\mathbf{x})p_j(\mathbf{x}) \mathrm{d}\mathbf{x}=\delta_{ij}
$ for all $i,j=1,\cdots,N$, then for any given point set $\{\mathbf{x}_j\}_{j=1}^N$, the Lagrange polynomials, if they exist, can be expressed as
$l_k(\mathbf{x})=\sum_{i=1}^Na_{ki}p_i(\mathbf{x})
$ with some constants $\{a_{ki}\}_{k,i=1}^N$, which allows \eqref{eq:l2} to be expressed as $\Vert \mathbf{L}\Vert _{2}=\sum_{k=1}^N \sum_{i=1}^N |a_{ki}|^2$. In comparison, the calculation of $\Gamma_{N}(\{\mathbf{x}_i\}_{i=1}^N)$ is equivalent to the calculation of $\max_{\mathbf{x}\in T}\sum_{k=1}^N |\sum_{i=1}^N |a_{ki}p_i(\mathbf{x})|$, which has more computational complexity. Numerical results show that $\Vert \mathbf{L}\Vert _{2}/ \Gamma_{N}\left(\{\mathbf{x}_i\}_{i=1}^N\right)$ is not large and the performance of this kind of point set is nearly optimal. 
\subsubsection*{Fekete points}
Fekete point set \cite{bos1991certain} is another type of nearly optimal nodal point set that maximizes the determinant of the generalized Vandermonde matrix $V(\mathbf{x}_1,\mathbf{x}_2,\cdots,\mathbf{x}_N)$, 
\begin{equation}\label{Vandermonde matrix}
 \begin{aligned}
\max_{\{\mathbf{x}_1,\mathbf{x}_2,\cdots,\mathbf{x}_N\}\subset T}\det(V(\mathbf{x}_1,\mathbf{x}_2,\cdots,\mathbf{x}_N)). 
\end{aligned}  
\end{equation}
Here, $V(\mathbf{x}_1,\mathbf{x}_2,\cdots,\mathbf{x}_N)$ is of size $N\times N$ with entries
\begin{align*}
V_{ij}=g_{j}(\mathbf{x}_i) \text{ for } i,j=1,\cdots,N.
\end{align*}
$\{g_i\}_{i=1}^N$ denotes a set of basis functions in $P_n(T)$. Note that $\det(V)$ can be regarded as a polynomial function of $(\mathbf{x}_1,\cdots,\mathbf{x}_n)$, which implies the existence of Fekete points for a given compact set $T$. Note also that Problem \eqref{Vandermonde matrix} involves less computational complexity than the minimization of \eqref{optimal}. As the first attempt, Bos \cite{bos1991certain} constructed the Fekete point set up to the 7th order, which was further extended up to degree 13 \cite{chen1995approximate} and 18 \cite{taylor2000algorithm} in a triangle, respectively.

\subsubsection*{Improved Lobatto grid}

The improved Lobatto grid
is proposed in \cite{blyth2006lobatto} as an improvement of the original Lobatto grid, which is composed of $(\xi_i,\eta_j)$ defined by 
\begin{align*}
\xi_i=\frac{1}{3}(1+2v_j-v_i-v_k),
\eta_j=\frac{1}{3}(1+2v_i-v_j-v_k) \text{ for }i=1,\cdots,n+1\text{ and }j=1,\cdots,n+2-i. 
\end{align*}
Here, $k:=n+3-i-j$ and $v_i:=\frac{1}{2}(1+t_{i})$ with $t_i$ denoting the zeros of the $n$th degree Labatto polynomials. 

This proposed
point set utilizes the zeros of Lobatto polynomials which are close to optimal for 1D interpolation \cite{fejer1932lagrangesche}. It is generated by deploying Lobatto interpolation nodes along the three edges of the triangle, and then
computing interior nodes by averaged intersections to achieve three-fold rotational symmetry. The symmetry of the distribution with respect to the three vertices is a significant improvement over the original Lobatto grid. Its straightforward implementation makes it an attractive choice, and numerical results show that the Lebesgue constant for this point set is competitive with the above-mentioned two point sets.

The following figures show the comparison of mean optimal points, Fekete points, and the improved Lobatto grid for $n=4,8$.
\begin{remark}[Lebesgue constants for mean optimal points, Fekete points and improved Lobatto grid]
Although there is no rigorous proof on the boundedness of the Lebesgue constants for mean optimal points, Fekete points and improved Lobatto grid, numerical evidence \cite{chen1995approximate,taylor2000algorithm,blyth2006lobatto} suggests that their Lebesgue constants are proportional to $\sqrt{N}$.
\end{remark}
\begin{figure}[H] 
    \centering  
\subfigure[$n=4$]{\label{Comparison1}
\includegraphics[width=.35\textwidth,trim={10cm 0cm 10cm 0cm},clip]{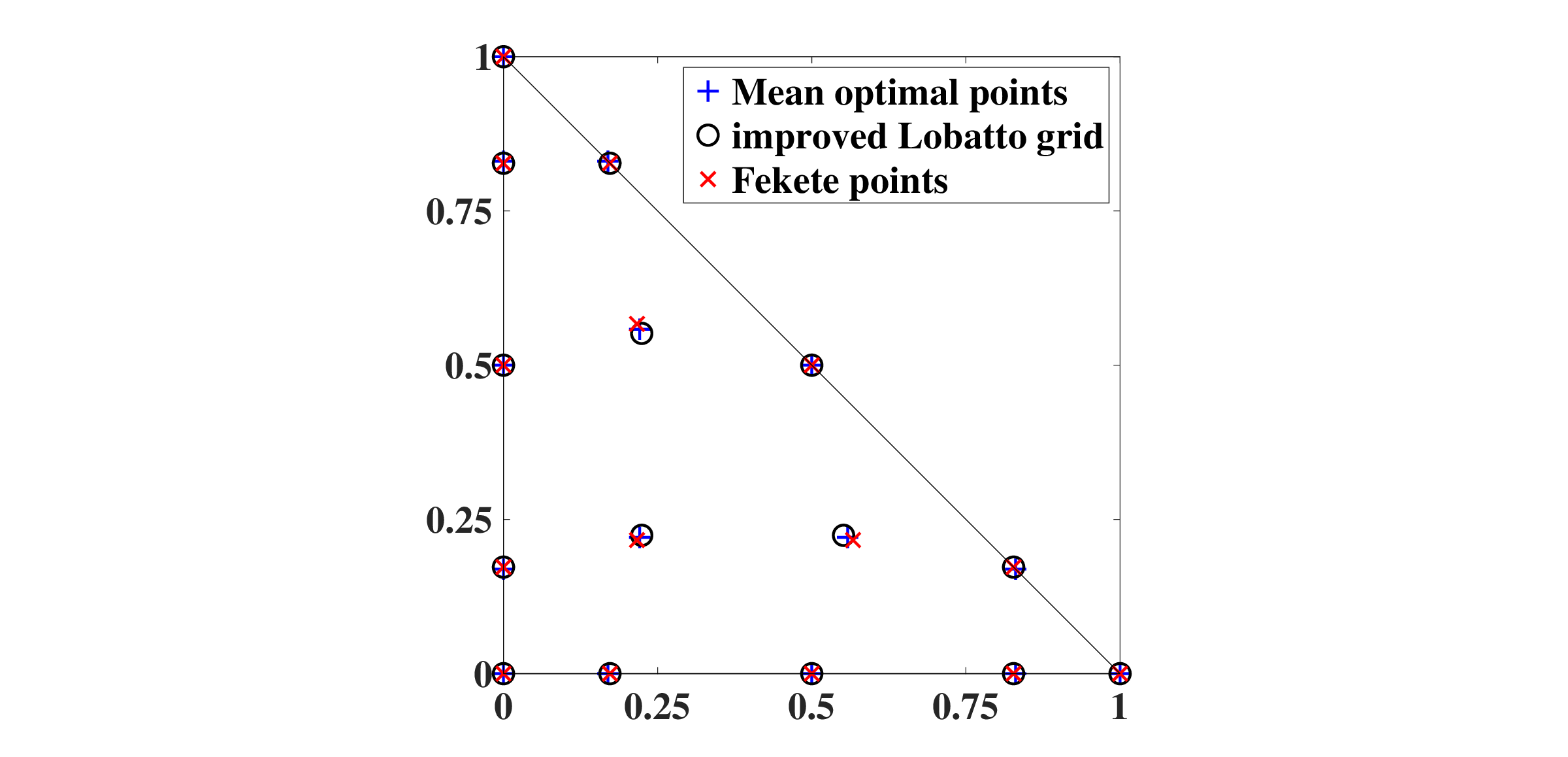}}
\subfigure[$n=8$]{\label{Comparison2}
\includegraphics[width=.35\textwidth,trim={10cm 0cm 10cm 0cm},clip]{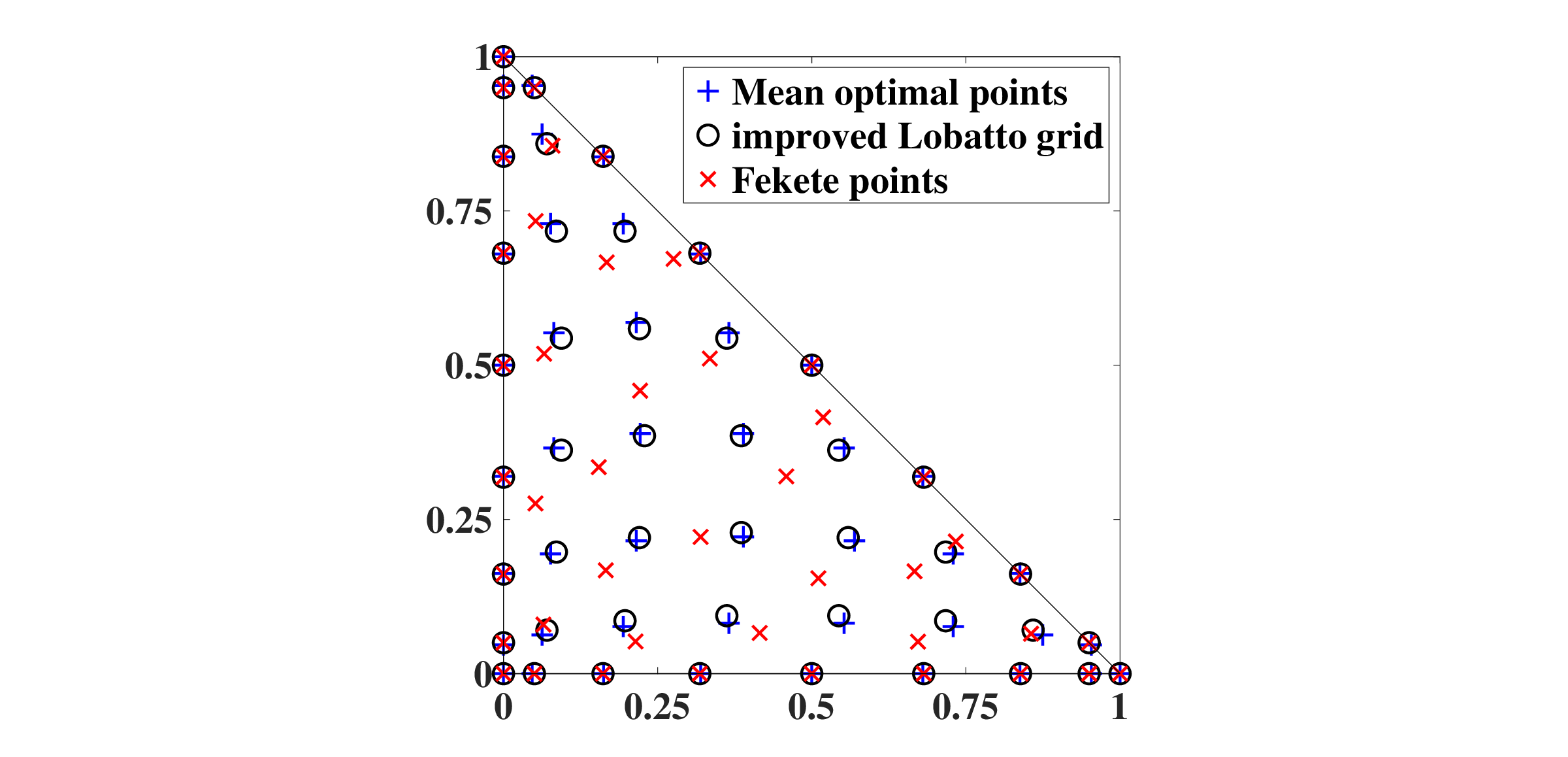}}
\caption{Mean optimal points, Fekete points, and improved Lobatto grid for $n=4,8$.}\label{Comparison}
    \end{figure} 
\subsection{Sampling points in quadrilateral}\label{subsec:sampling points-r}
For a quadrilateral, we will first project it into the unit square $S:=[-1,1]^2$ by projective mapping \cite{heckbert1999projective} and then consider the polynomial space with two variables and degree at most $n$ in each variable, i.e.,
\begin{equation*}
   P_n(S)=\text{span}\{x^iy^j,\quad i,j\leq n\}.  
\end{equation*}
\begin{figure}[H]
\centering
\subfigure[$n=4$]{\label{cheb1_25}
\includegraphics[width=.35\textwidth,trim={1cm 0.2cm 1cm 0.2cm},clip]{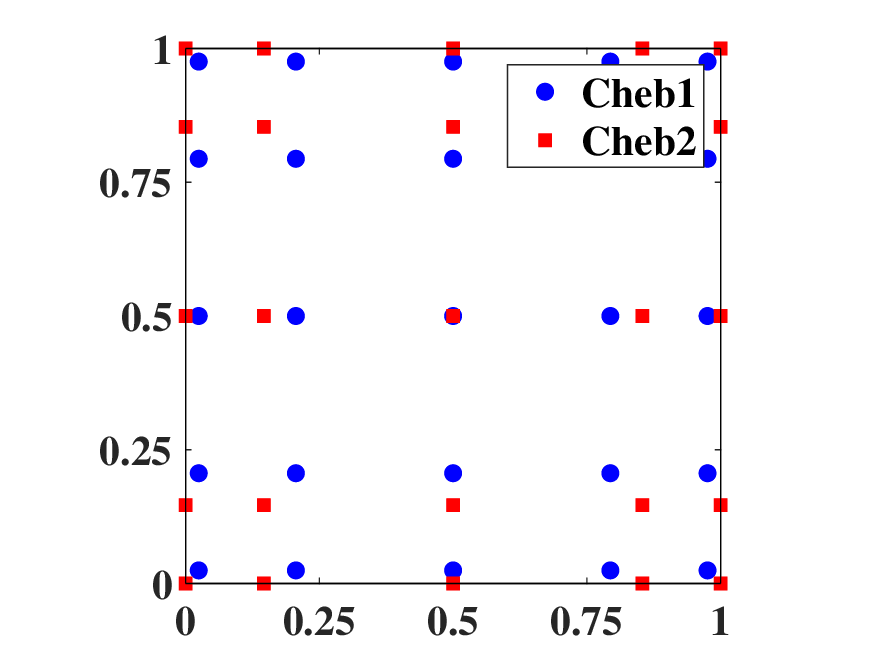}
}%
\quad
\subfigure[$n=8$]{\label{cheb1_81}
\includegraphics[width=.35\textwidth,trim={1cm 0.2cm 1cm 0.2cm},clip]{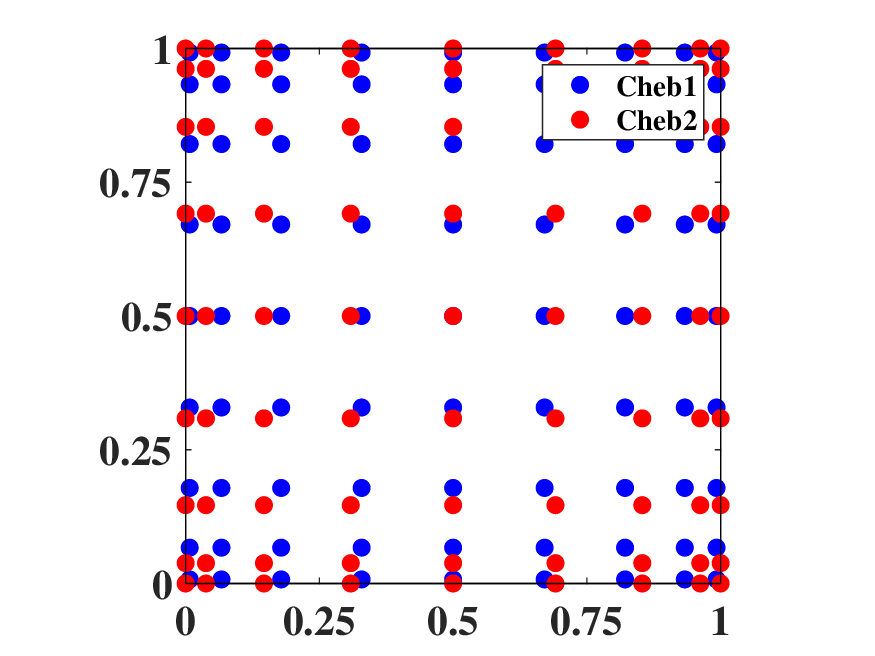}
}%
\centering
\caption{Cheb1 and Cheb2 with tensor product for $n=4,8$.}
\label{chebyyy}
\end{figure} 
It is well known that in the one-dimensional case, interpolation using the zeros of Chebyshev polynomials is close to optimal. So in the case of unit square, we mainly consider the following two sampling point sets with tensor product:

The Chebyshev points of the first kind (Cheb1) in the interval $[-1,1]$ are the zeros of the Chebyshev polynomial of the first kind $T_{n+1}(x)$, 
 \begin{align*}
x_k=\cos{\frac{2k+1}{2(n+1)}\pi},\quad k=0,\cdots,n.
 \end{align*}

The Chebyshev points of the second kind (Cheb2) in the interval $[-1,1]$ are the zeros of the Chebyshev polynomial of the second kind $U_{n-1}(x)$ times $(x^2-1)$, i.e.,  \begin{align*}
x_i=\cos{\left(\frac{i}{n}\pi\right)}, \quad i=0,\cdots,n. \end{align*}

Figure \ref{chebyyy} demonstrates the comparison of Cheb1 and Cheb2 for $n=4$ and $n=8$ after using tensor product to expand them to the unit square.
\begin{remark}[Lebesgue constants for $\operatorname{Cheb1}$ and $\operatorname{Cheb2}$]
Since the Lebesgue constants of $\operatorname{Cheb1}$ and $\operatorname{Cheb2}$ are both proportional to $\log(N)$ \cite{ibrahimoglu2016lebesgue}, where $N=n+1$ is the dimension of the polynomial space with polynomial degrees at most $n$ defined in $[-1,1]$, their Lebesgue constants on $[-1,1]^2$ are proportional to $(\log(N))^2$, where $N=(n+1)^2$ denotes the dimension of the polynomial space with polynomial degrees at most $n$ in each variable.
\end{remark}
\begin{figure}[H] 
    \centering  
\subfigure[Square lattice \#DOFs=2107]{\label{mesh1}
\includegraphics[width=.4\textwidth]{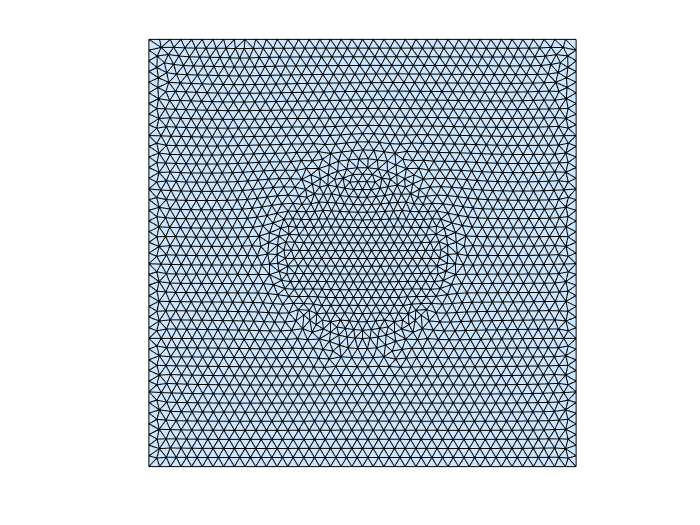}}\qquad
\subfigure[Hexagonal lattice \#DOFs=1781]{\label{mesh2}\includegraphics[width=.4\textwidth]{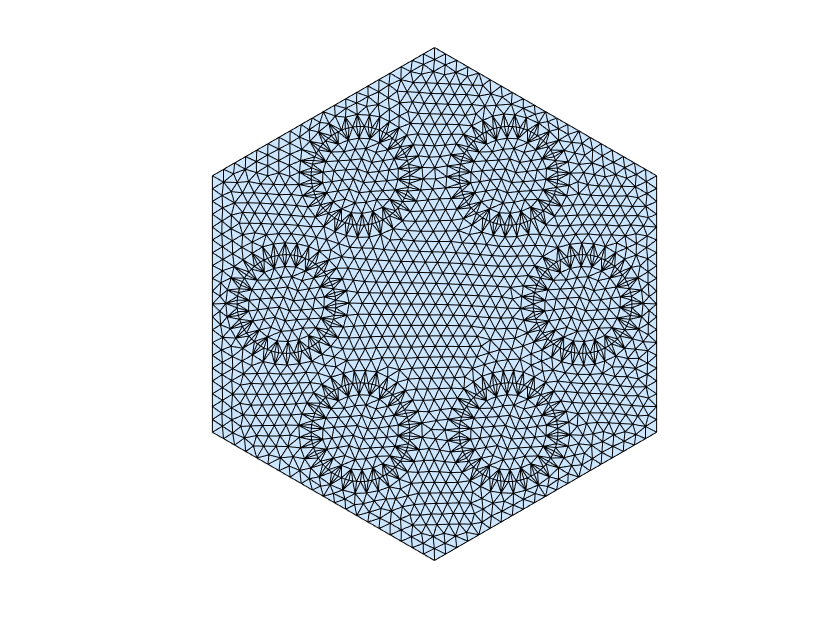}}
\caption{Discretization of the unit cell $\Omega$ by \textit{distmesh2d}.}
    \end{figure} 
\vspace{-1.5cm}
\begin{figure}[H] 
    \centering  
\subfigure[$253$ evenly distributed points in the IBZ of the square lattice.]{\label{point_in_B 1}
\includegraphics[width=0.28\textwidth,trim={1cm 0.2cm 1cm 0.5cm},clip]{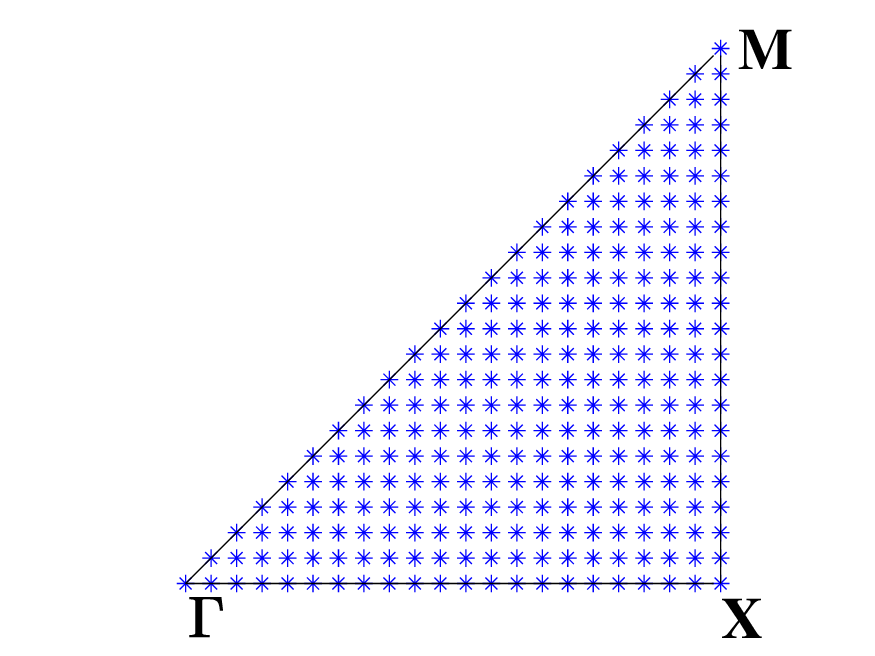}}
\subfigure[$253$ evenly distributed points in the IBZ of the hexagonal lattice.]{\label{point_in_B 2}\includegraphics[width=0.35\textwidth,trim={1cm 0.2cm 1cm 0.5cm},clip]{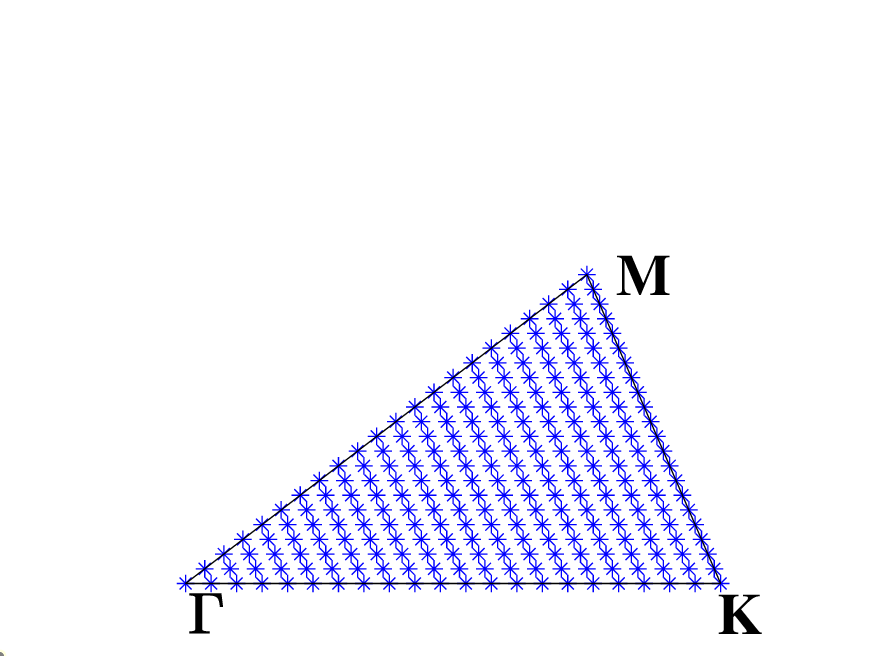}}
\caption{$\hat{\mathcal{B}}$.}
\label{fig:ref-pts}
\end{figure}
\begin{remark}[Computational complexity]
It is worth noticing that the computational complexity for \eqref{eq:interpolation} is consistent corresponding to all these five sampling methods, even though the number of Chebyshev points is almost twice as the number of those three kinds of sampling points within the triangular area, due to the mirror and rotational symmetry mentioned earlier. For example, in Figure \ref{cheb1_25}, we have 25 Chebyshev points, but we only need to compute eigenvalues at 15 of them within the isosceles right triangle.
\end{remark}

\begin{remark}

As depicted in Figure \ref{chebyyy}, $\operatorname{Cheb2}$ has nodes on the edges, whereas $\operatorname{Cheb1}$ does not. Furthermore, it is empirically known that the grid points over the edges are more important than the interior grid points for the band function reconstruction \cite{hinuma2017band}. Consequently, we anticipate that the performance of $\operatorname{Cheb2}$ will surpass that of $\operatorname{Cheb1}$, which is confirmed in Section \ref{sec:experiment}. 
\end{remark}
\section{Numerical experiments}\label{sec:experiment}
To demonstrate the performance of our proposed method \eqref{eq:interpolation} together with the sampling methods presented in Sections \ref{subsec:sampling points-t} and \ref{subsec:sampling points-r}, 
we mainly utilize a square lattice (Figure \ref{lattice}) and a hexagonal lattice (Figure \ref{lattice3}) as the test beds.
\begin{figure}[H]
\centering
\subfigure[TE mode]{\label{zhexian2}
\includegraphics[width=.35\textwidth,trim={11cm 0cm 12.5cm 1cm},clip]{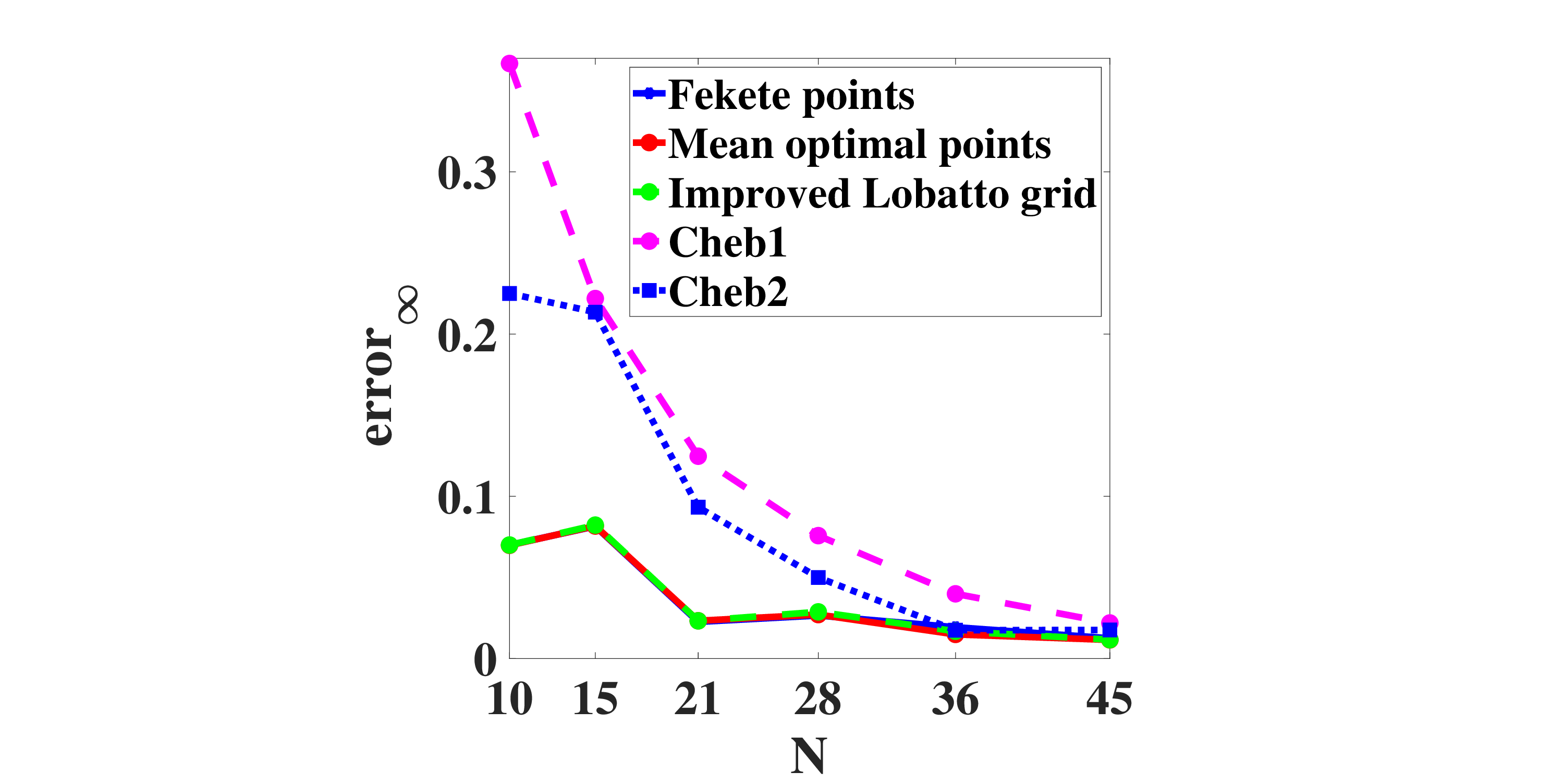}
}%
\quad
\subfigure[TM mode]{\label{zhexian1}
\includegraphics[width=.35\textwidth,trim={11cm 0cm 12.5cm 0cm},clip]{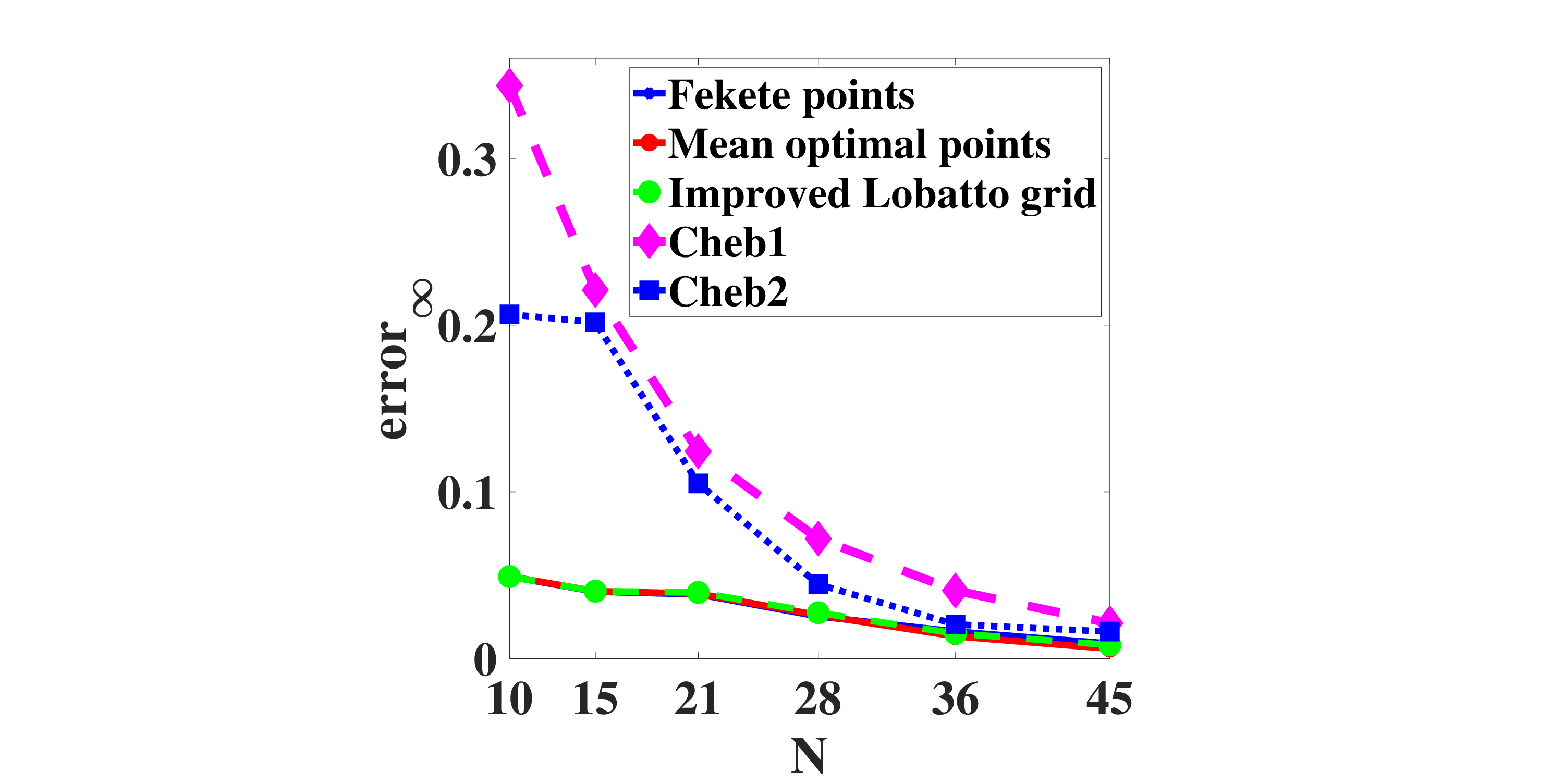}
}%
 \\          
\subfigure[TE mode]{\label{zhexian4}
\includegraphics[width=.35\textwidth,trim={11cm 0cm 12.5cm 0cm},clip]{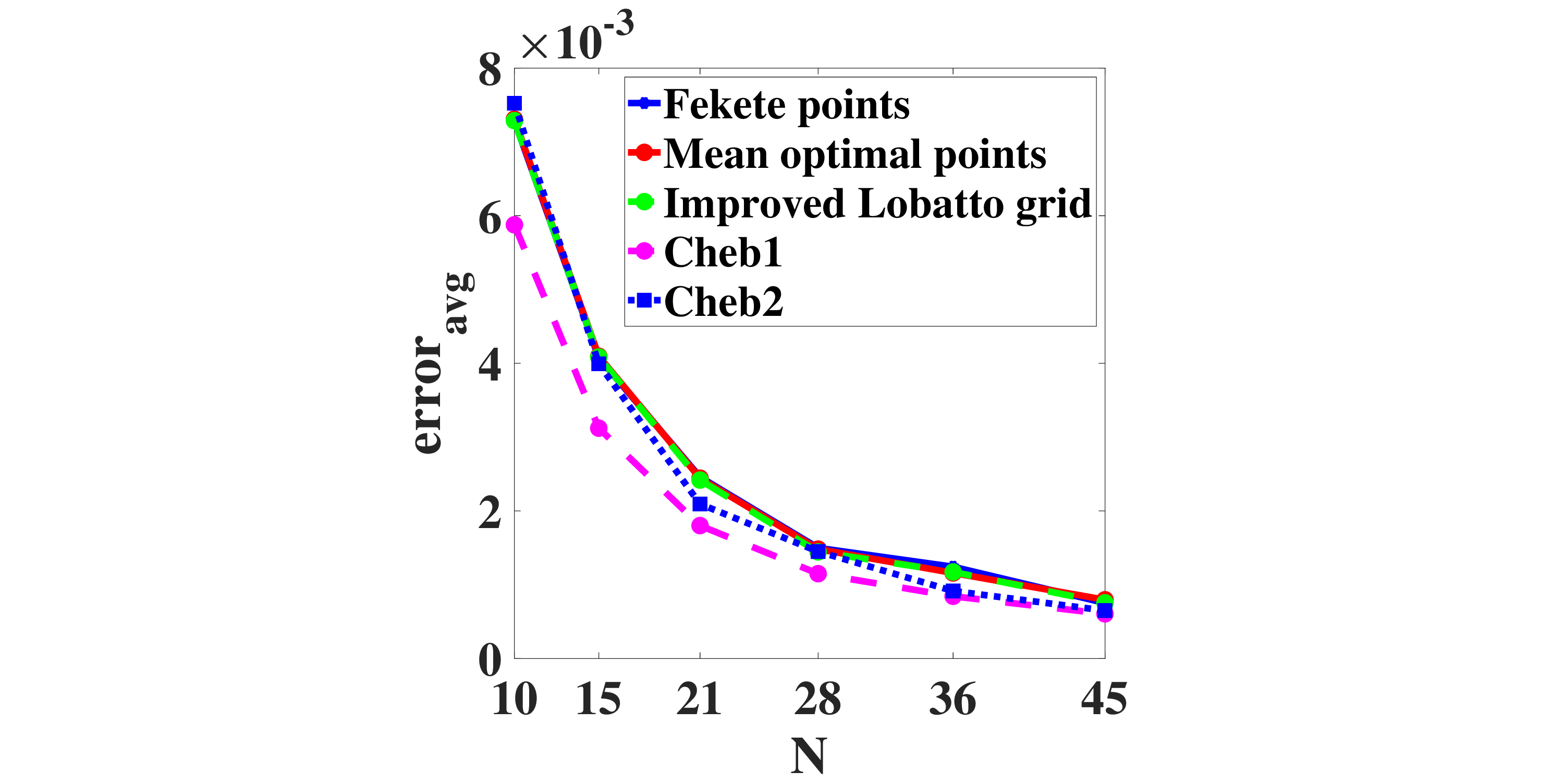}
}%
\quad
\subfigure[TM mode]{\label{zhexian3}
\includegraphics[width=.35\textwidth,trim={11cm 0cm 12.5cm 0cm},clip]{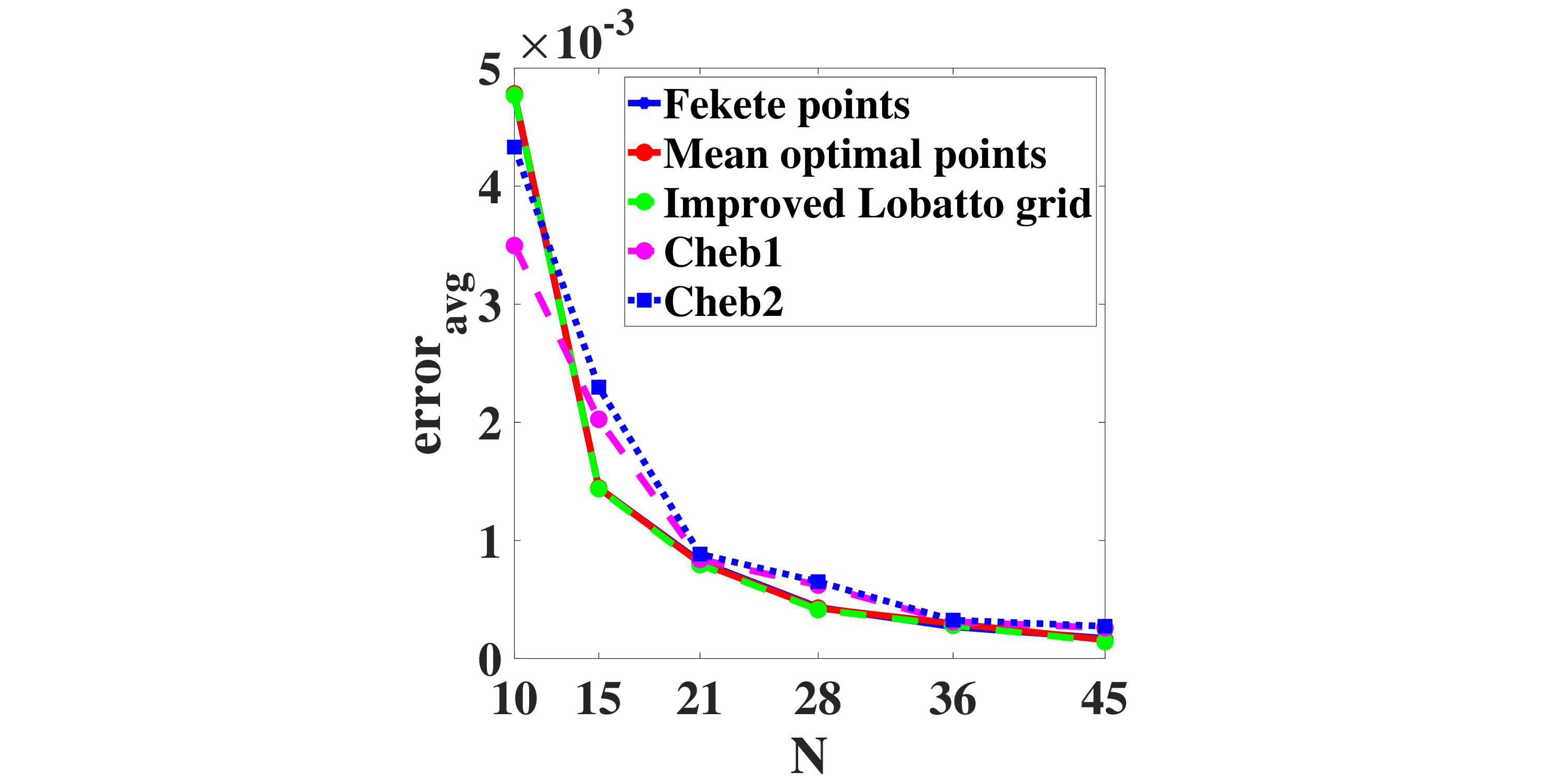}
}%
\centering
\caption{Square lattice: the performance of \eqref{eq:interpolation} under $\operatorname{error}_{\infty}$ and $\operatorname{error}_{\text{avg}}$.}\label{erroe square}
\end{figure}
We calculate the eigenvalue problem for a given sampling point using the conforming Galerkin Finite Element method. Due to the high contrast between the relative permittivity of the circular medium and its external medium, we utilize a fitted mesh $\mathcal{T}_h$ generated by \textit{distmesh2d} provided by Persson and Strang \cite{persson2004simple} to avoid the stabilization issue when an unfitted mesh is used, which is depicted in Figures \ref{mesh1} and \ref{mesh2}. Here, we choose a mesh size $h=0.025a$ for the square lattice and $h=0.05a$ for the hexagonal lattice. The associated conforming piecewise affine space is 
\begin{align*}
V_h=\{v_h\in C(\overline{\Omega}):v_h|_K\in P_1(K)\,\,\text{ for all }K\in\mathcal{T}_h\,\}.
\end{align*}
    
Given $\mathbf{k}\in\mathcal{B}_{\operatorname{red}}$, the conforming Galerkin Finite Element approximation to Problem \eqref{simply} reads as finding a non-trivial eigenpair $(\lambda_h,u_h) \in (\mathbb{R},V_h)$ satisfying 
\begin{equation}\label{fem variational}
    \left\{
\begin{aligned}
&a(u_h,v_h)=\lambda b(u_h,v_h) \text{ for all }v \in V_h \\
&b(u_h,u_h)=1.
\end{aligned}
\right.
\end{equation}
\begin{figure}[H]
\centering
\subfigure[TE mode]{\label{TEfekete45}
\includegraphics[width=0.32\textwidth,trim={10.5cm 1cm 12cm 1cm},clip]{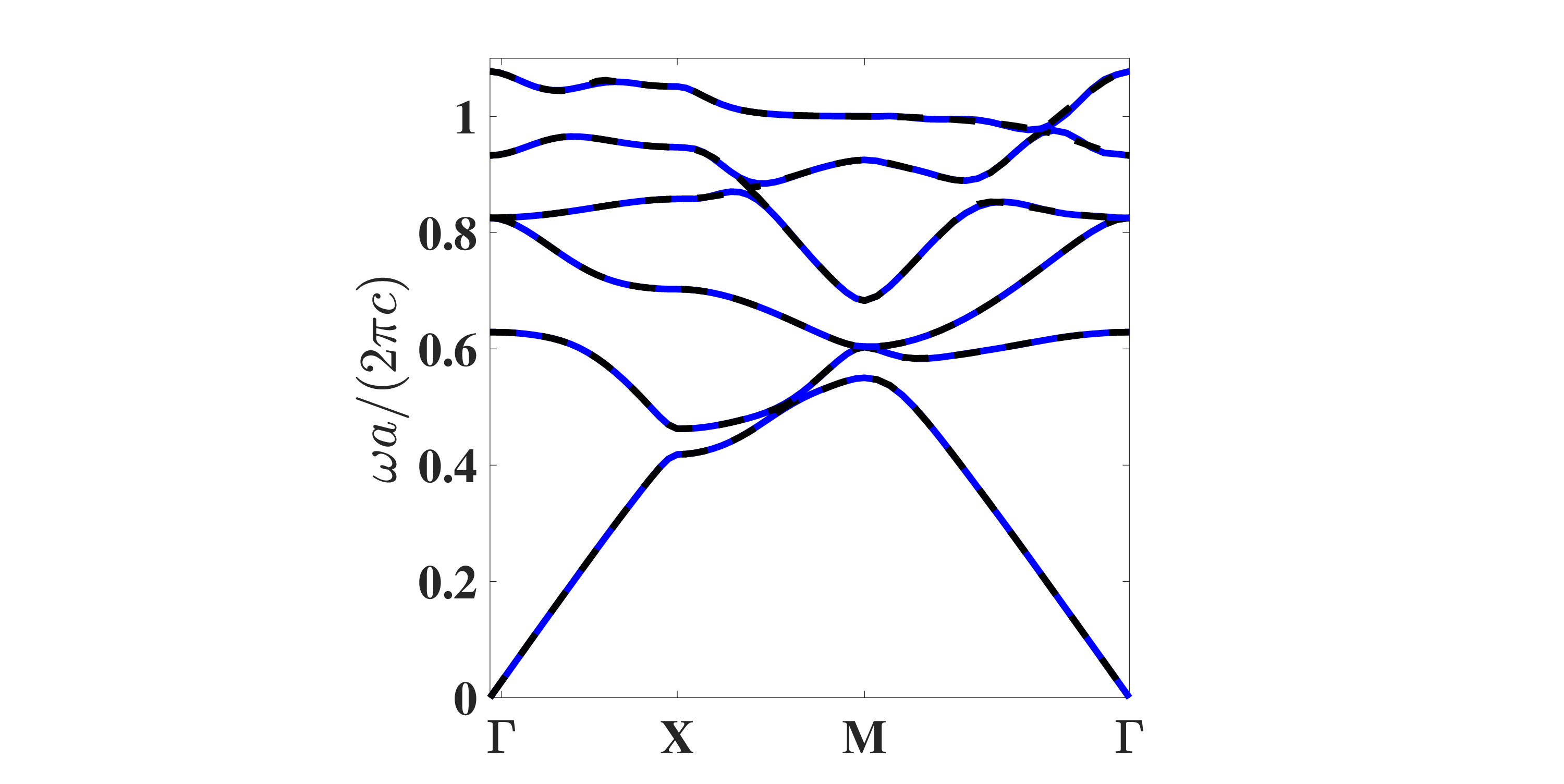}
}%
\quad
\subfigure[TM mode]{\label{TMfekete45}
\includegraphics[width=0.32\textwidth,trim={10.5cm 1cm 12cm 1cm},clip]{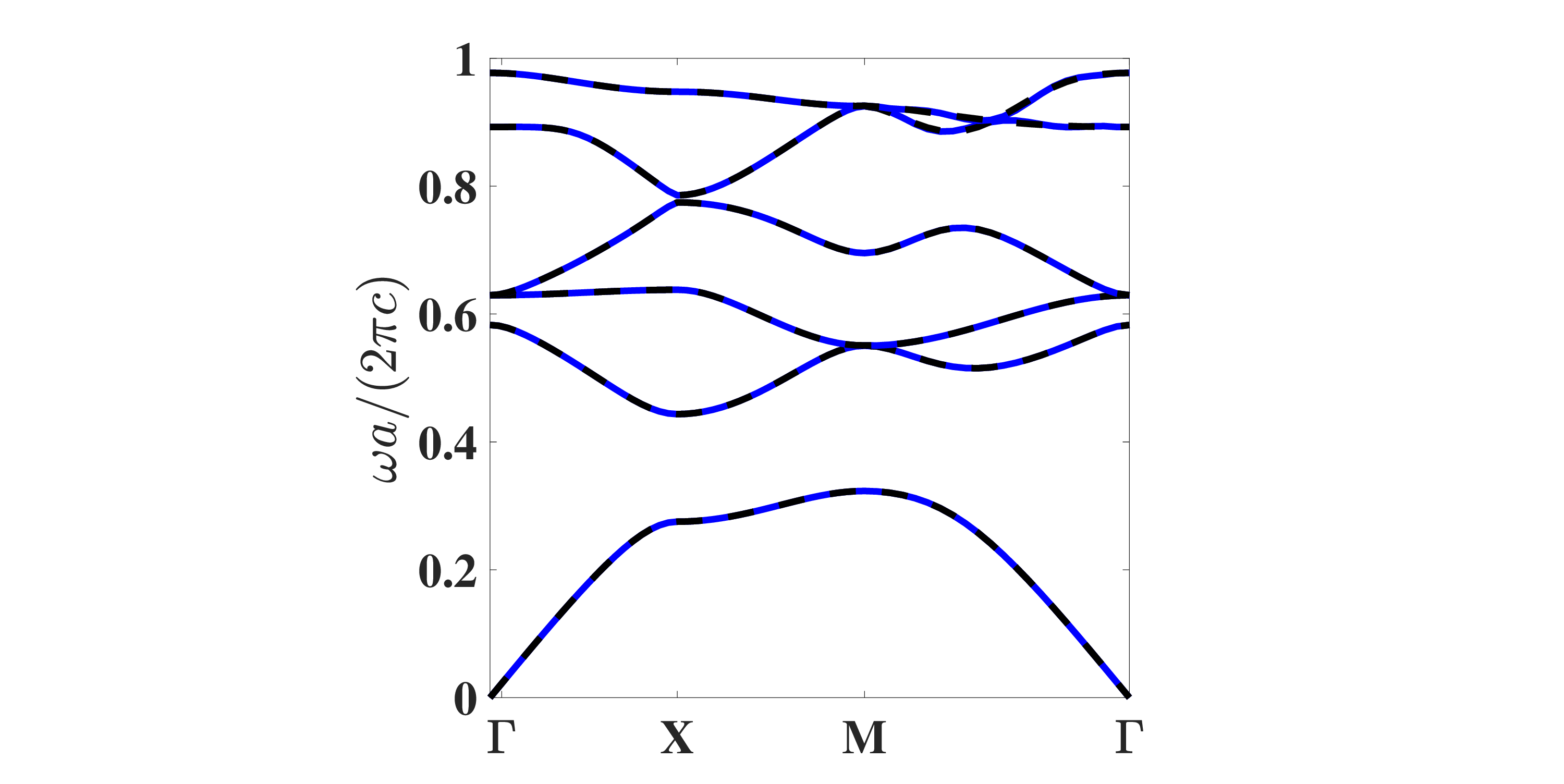}
}%
 \\            
\subfigure[TE mode]{\label{TEcheb2_45}
\includegraphics[width=0.32\textwidth,trim={10.5cm 1cm 12cm 1cm},clip]{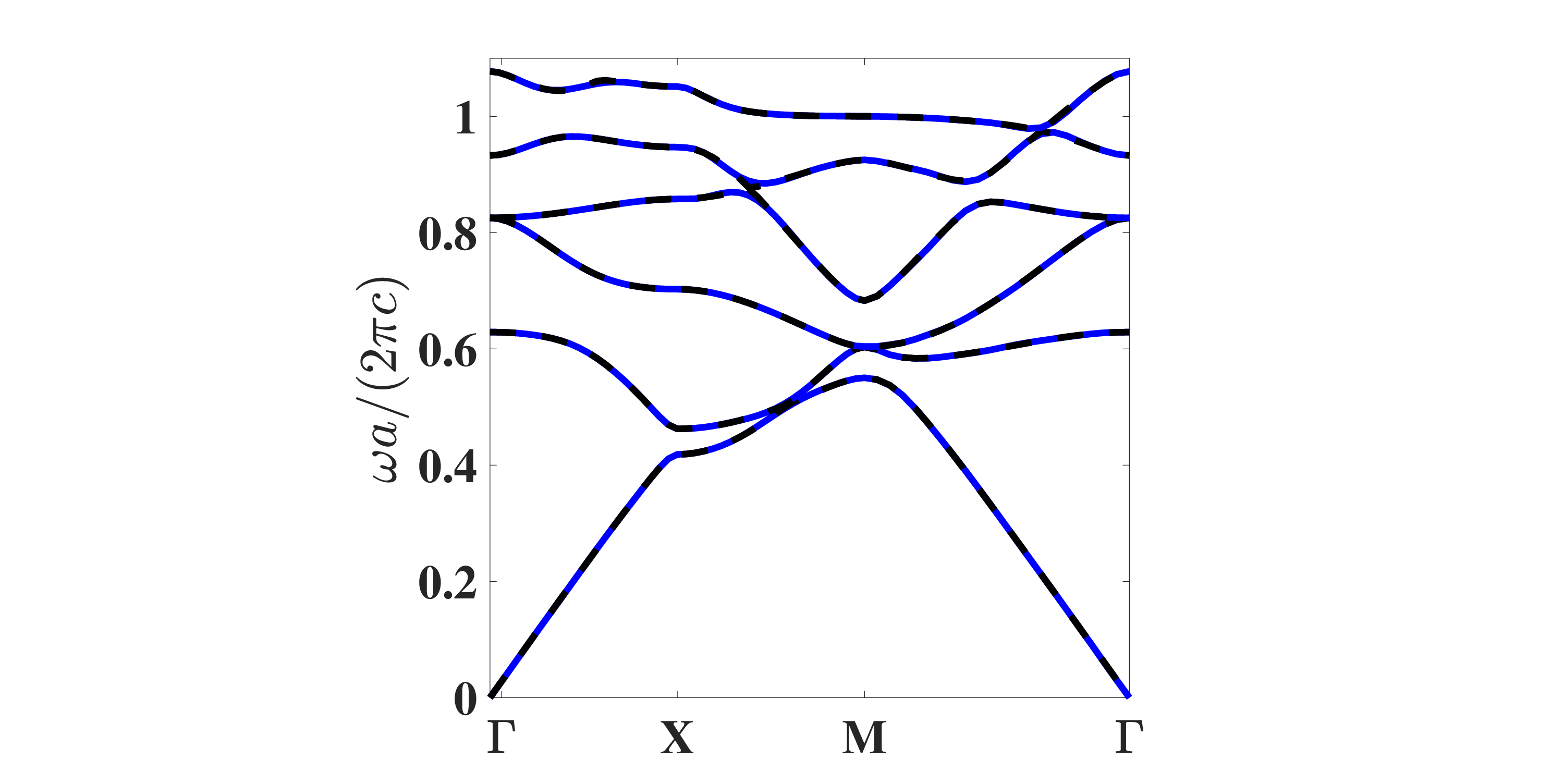}
}%
\quad
\subfigure[TM mode]{\label{TMcheb2_45}
\includegraphics[width=0.32\textwidth,trim={10.5cm 1cm 12cm 1cm},clip]{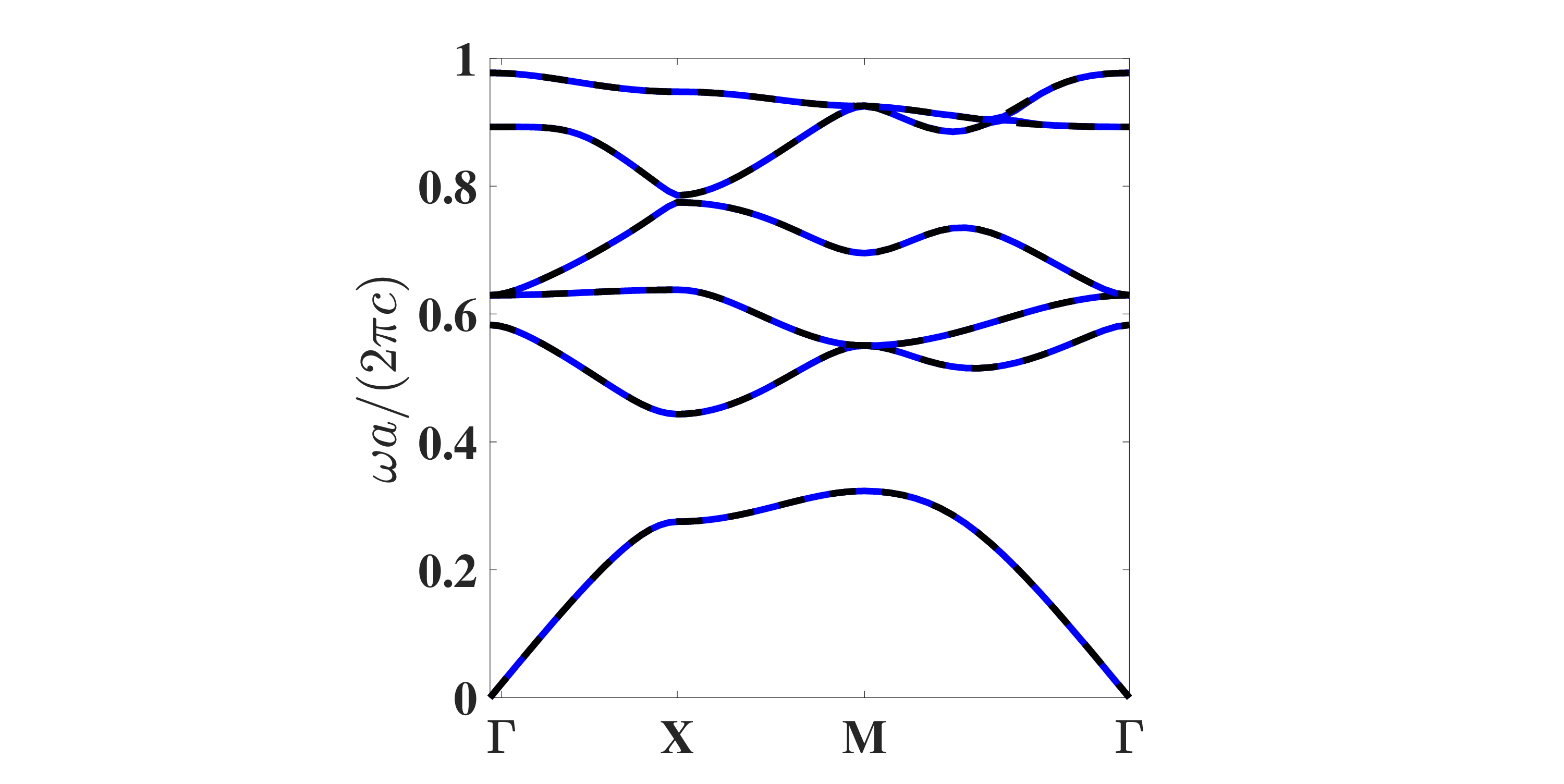}
}%
\caption{Square lattice: band functions along $\partial\mathcal{B}_{\operatorname{red}}$ using Fekete points (the first row) and Cheb2 (the second row) with degree $n=8$. The dashed black line is the reference band structures.}
\label{fig:bandStructure}
\end{figure}  
Note that $\mathcal{B}_{\operatorname{red}}$ is a triangle in both cases. 
To obtain the reference solution with sufficient accuracy, we discretize $\mathcal{B}_{\operatorname{red}}$ with 253 evenly distributed points as shown in Figure \ref{fig:ref-pts}, which is denoted as $\hat{\mathcal{B}}$.
The pointwise relative error is then evaluated on $\hat{\mathcal{B}}$ by, 
 \begin{align*}
    e_i(\mathbf{k}):=\frac{\vert \omega_i(\mathbf{k}) -\mathbf{L}\omega_i(\mathbf{k})\vert}{{\omega_i}(\mathbf{k})} \quad \text{ for }\mathbf{k}\in\hat{\mathcal{B}} \text{ and } i=1,\cdots, 6.
 \end{align*}
Here, $\omega_i(\mathbf{k})$ is the $i$th reference band function obtained directly by the conforming Galerkin Finite Element method over $\hat{\mathcal{B}}$ using the same mesh on the unit cell $\Omega$, and $\mathbf{L}\omega_i(\mathbf{k})$ is the Lagrange interpolation \eqref{eq:interpolation} with a certain sampling method. Given that the first few low-frequency band functions are typically of practical interest \cite{joannopoulos2008molding}, we demonstrate the performance of our numerical scheme using only the first six band functions, which is not the limitation of our scheme. Specifically, we use maximum relative error and average relative error to investigate the performance of our methods, which are defined by 
\begin{align*}
\operatorname{error}_{\infty}:=\max_{i=1,\cdots,6}\max_{\mathbf{k}\in\hat{\mathcal{B}}}\left|e_i(\mathbf{k})\right|\text{ and }
\operatorname{error}_{\text{avg}}:=\frac{1}{253}\sum_{\mathbf{k}\in\hat{\mathcal{B}}}\left(\frac{1}{6}\sum_{i=1}^6\left|e_i(\mathbf{k})\right|\right).
\end{align*}
\subsection{Numerical tests with square lattice in Figure \ref{lattice}}
In this section, we are concerned with the case of square lattice shown in Figure \ref{lattice}, wherein the primitive lattice vectors are $\mathbf{a}_{i}=ae_i$ for $i=1,2$ with a positive parameter $a$. The circle area has radius $r=0.2a$ and $\epsilon=8.9$ (as for alumina) which is embedded in air ($\epsilon=1$). In this case, due to the symmetry of the Brillouin zone $\mathcal{B}$, we can restrict the sampling points $\mathbf{k}$ to the triangle $\mathcal{B}_{\text{red}}$ and apply the sampling methods introduced in Subsection \ref{subsec:sampling points-t}. Alternatively, we can sample $\mathbf{k}$ in a quadrilateral $\Tilde{\mathcal{B}}_{\text{red}}$ composed of $\mathcal{B}_{\text{red}}$ and its mirror symmetry along its longest edge, then consider the sampling methods defined in Subsection \ref{subsec:sampling points-r}. 
\begin{figure}[H]
\centering
\subfigure[TE mode]{\label{zhexian22}
\includegraphics[width=.35\textwidth,trim={11cm 0cm 12.5cm 0cm},clip]{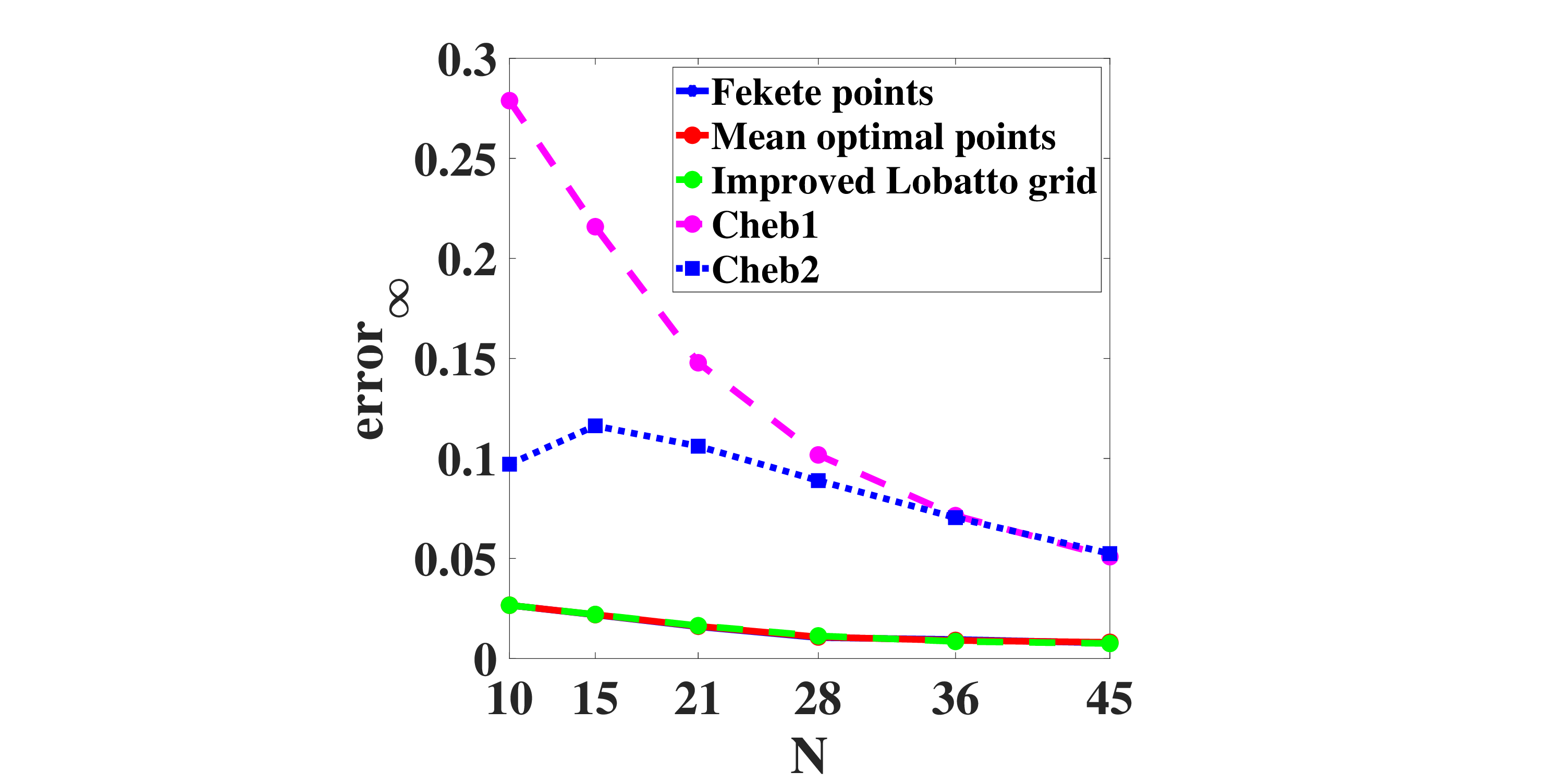}%
}%
\quad
\subfigure[TM mode]{\label{zhexian21}
\includegraphics[width=.35\textwidth,trim={11cm 0cm 12.5cm 0cm},clip]{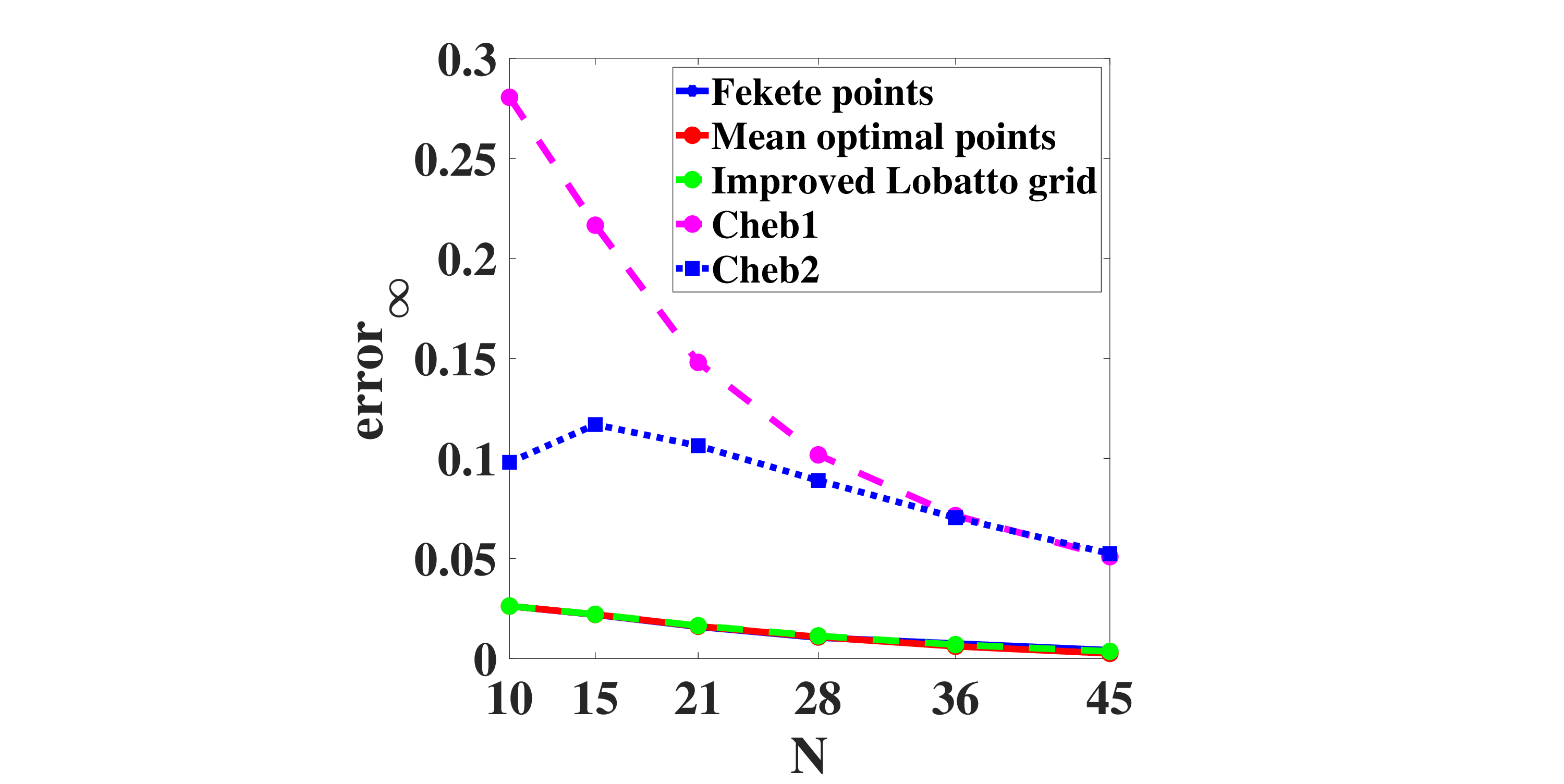}
}%
 \\             
\subfigure[TE mode]{\label{zhexian24}
\includegraphics[width=.35\textwidth,trim={11cm 0cm 12.5cm 0cm},clip]{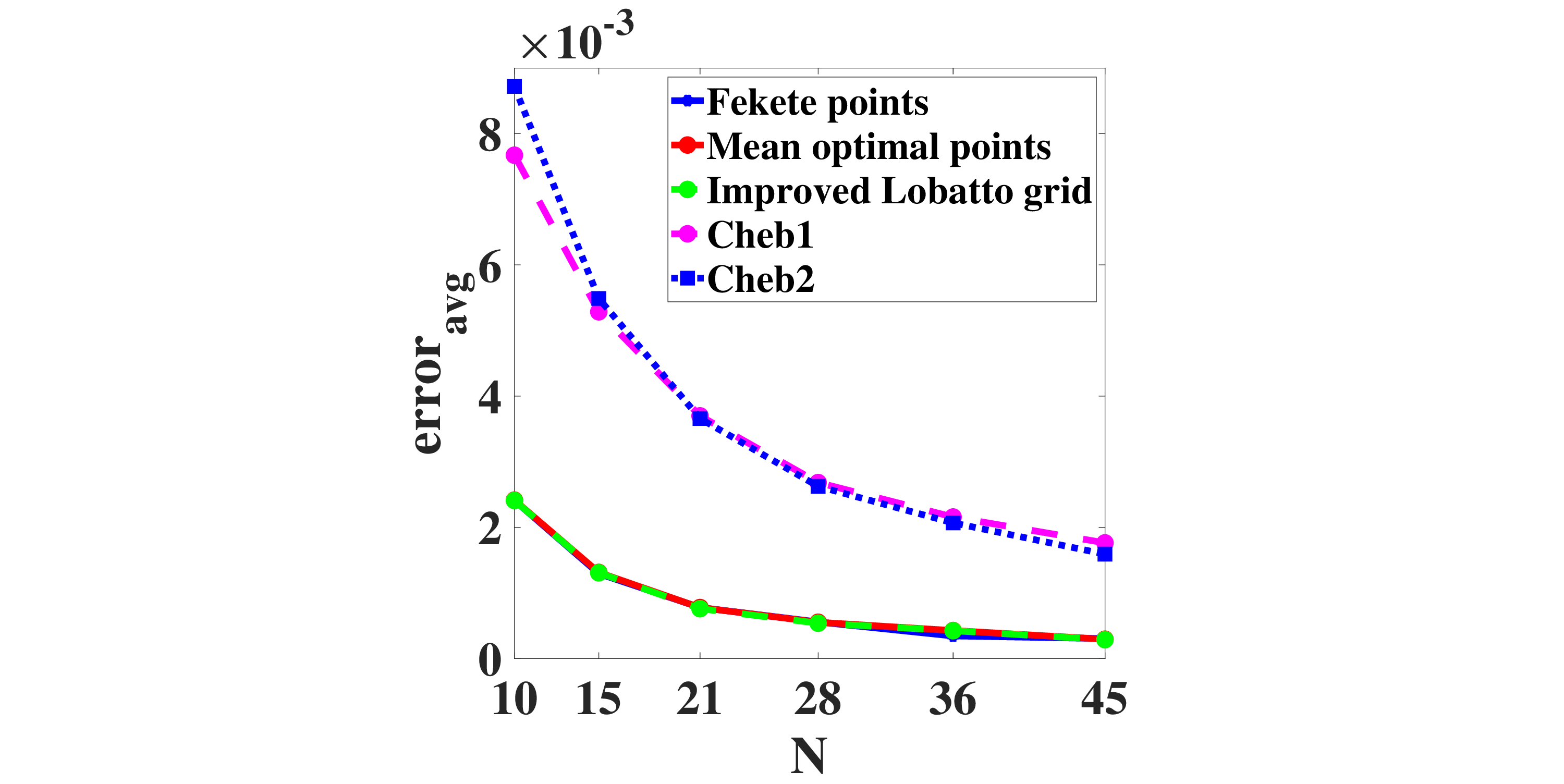}
}%
\quad
\subfigure[TM mode]{\label{zhexian23}
\includegraphics[width=.35\textwidth,trim={11cm 0cm 12.5cm 0cm},clip]{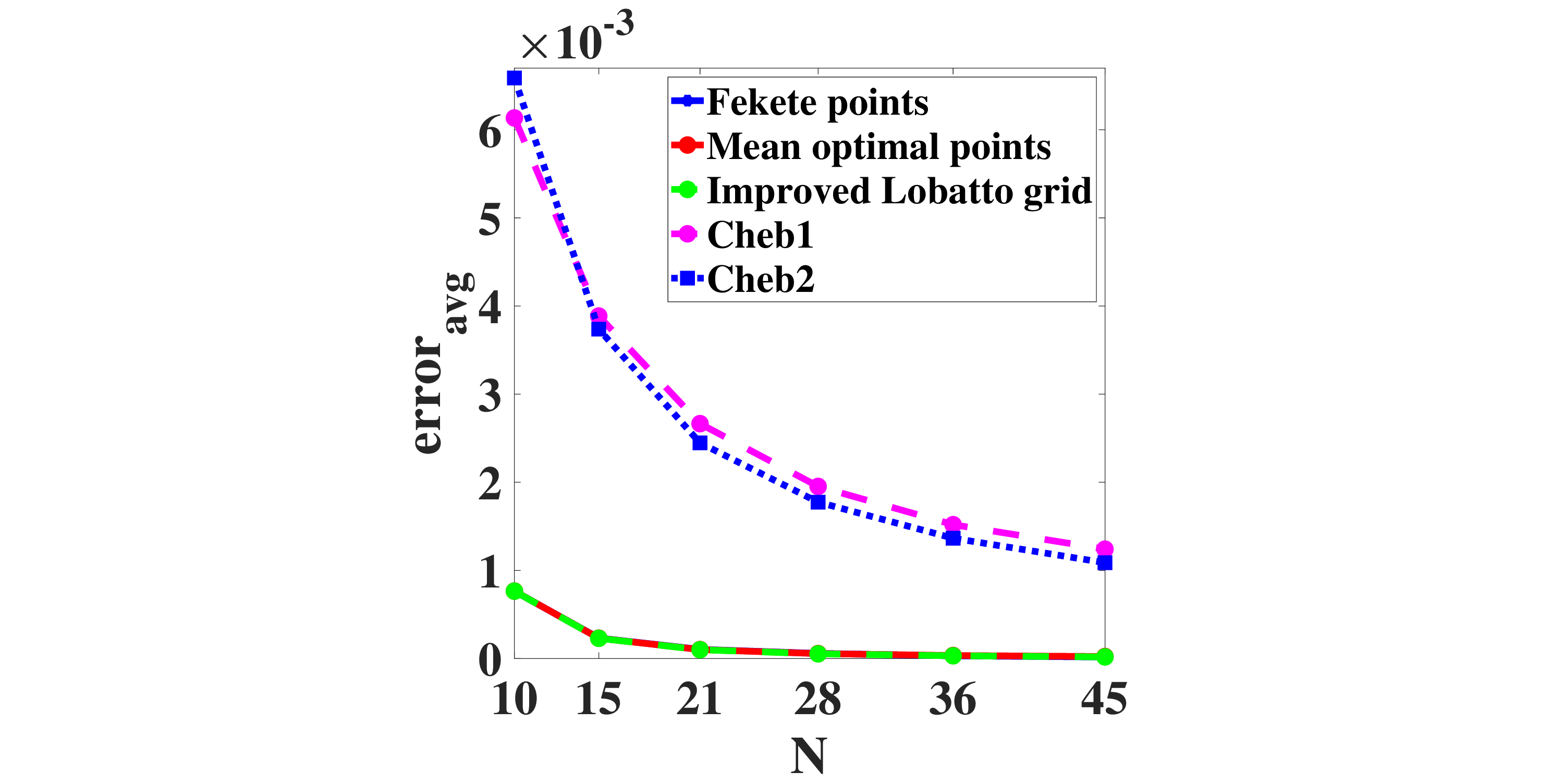}
}%
\caption{Hexagonal lattice: the performance of \eqref{eq:interpolation} under $\operatorname{error}_{\infty}$ and $\operatorname{error}_{\text{avg}}$.}\label{error hex}
\end{figure} 
The performance of Lagrange interpolation on those five kinds of sampling points is shown in Figure \ref{erroe square}. Note that the horizontal axis shows the number of sampling points inside $\mathcal{B}_{\text{red}}$. 
Figure \ref{fig:bandStructure} depicts the approximate band structure along the edges of the irreducible Brillouin zone using Fekete points and Cheb2 with degree $n=8$. 
\subsection{Numerical tests with hexagonal lattice in Figure \ref{lattice3}}
In this section, we focus on another kind of 2D PhCs which has infinite periodic hexagonal lattice. As shown in Figure \ref{lattice3}, its unit cell is composed of six cylinders of dielectric material with dielectric constant $\epsilon=8.9$ embedded in air. The primitive lattice vectors are $\mathbf{a}_1 = \frac{a}{4}({e}_1+\sqrt{3}{e}_2)$ and $\mathbf{a}_2 = \frac{a}{4}({e}_1-\sqrt{3}{e}_2)$, and the lattice constant $a=3R$. The radius of cylinders is $r=\frac{1}{3}R$. Here, the positive parameter $R$ denotes the length of hexagon edges. In this case, due to the symmetry of the first Brillouin zone, we can restrict the wave vector $\mathbf{k}$ to $\mathcal{B}_{\text{red}}$ or $\Tilde{\mathcal{B}}_{\text{red}}$ which is composed of $\mathcal{B}_{\text{red}}$ and its symmetry along its longest edge.
\begin{figure}[H]
\centering
\subfigure[TE mode]{\label{2TEfekete45}
\includegraphics[width=0.32\textwidth,trim={11cm 1cm 12.5cm 1cm},clip]{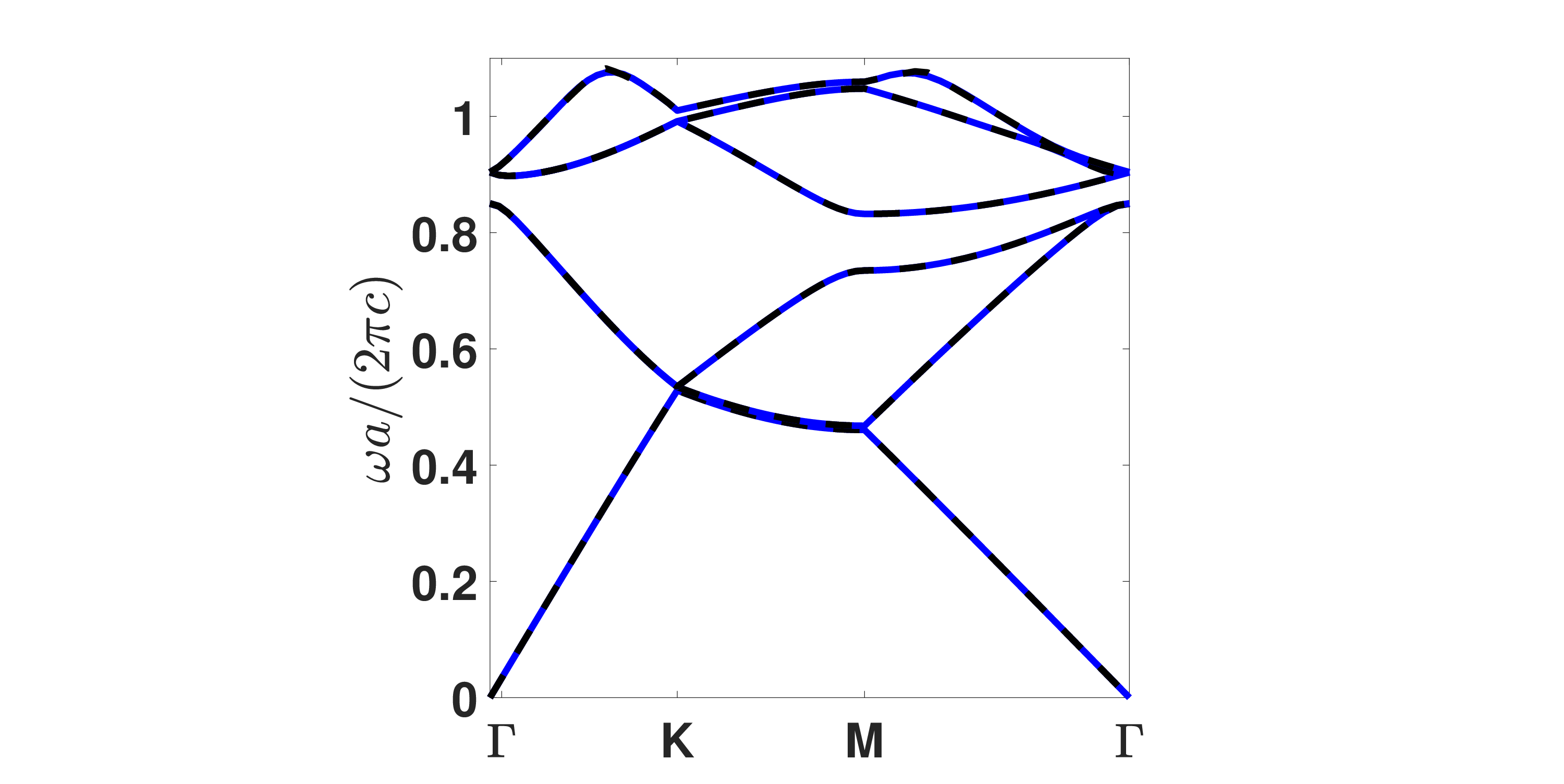}
}%
\quad
\subfigure[TM mode]{\label{2TMfekete45}
\includegraphics[width=0.32\textwidth,trim={11cm 1cm 12.5cm 1cm},clip]{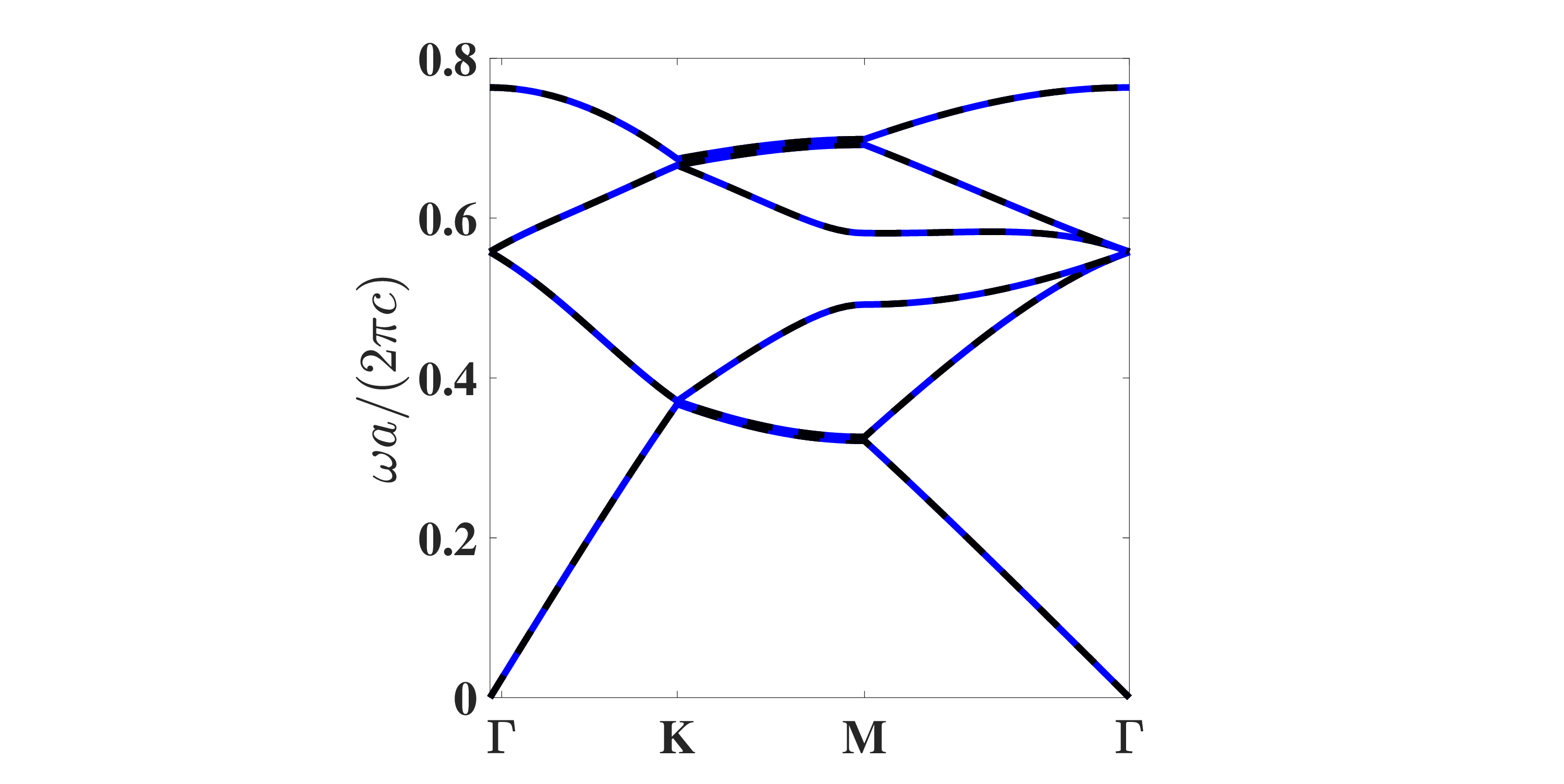}
}%
 \quad              
\subfigure[TE mode]{\label{2TEcheb2_45}
\includegraphics[width=0.32\textwidth,trim={11cm 1cm 12.5cm 1cm},clip]{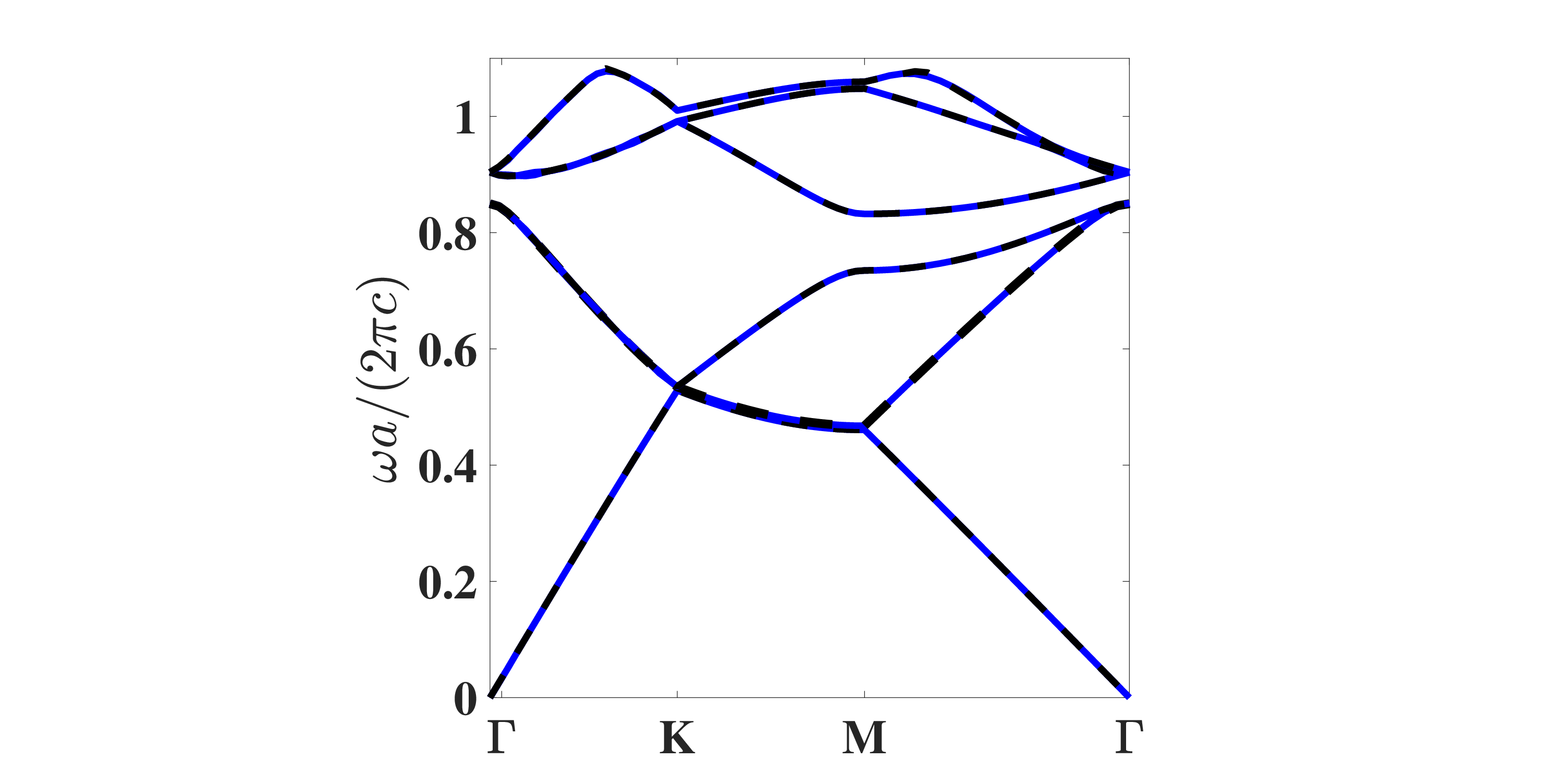}
}%
\quad
\subfigure[TM mode]{\label{2TMcheb2_45}
\includegraphics[width=0.32\textwidth,trim={11cm 1cm 12.5cm 1cm},clip]{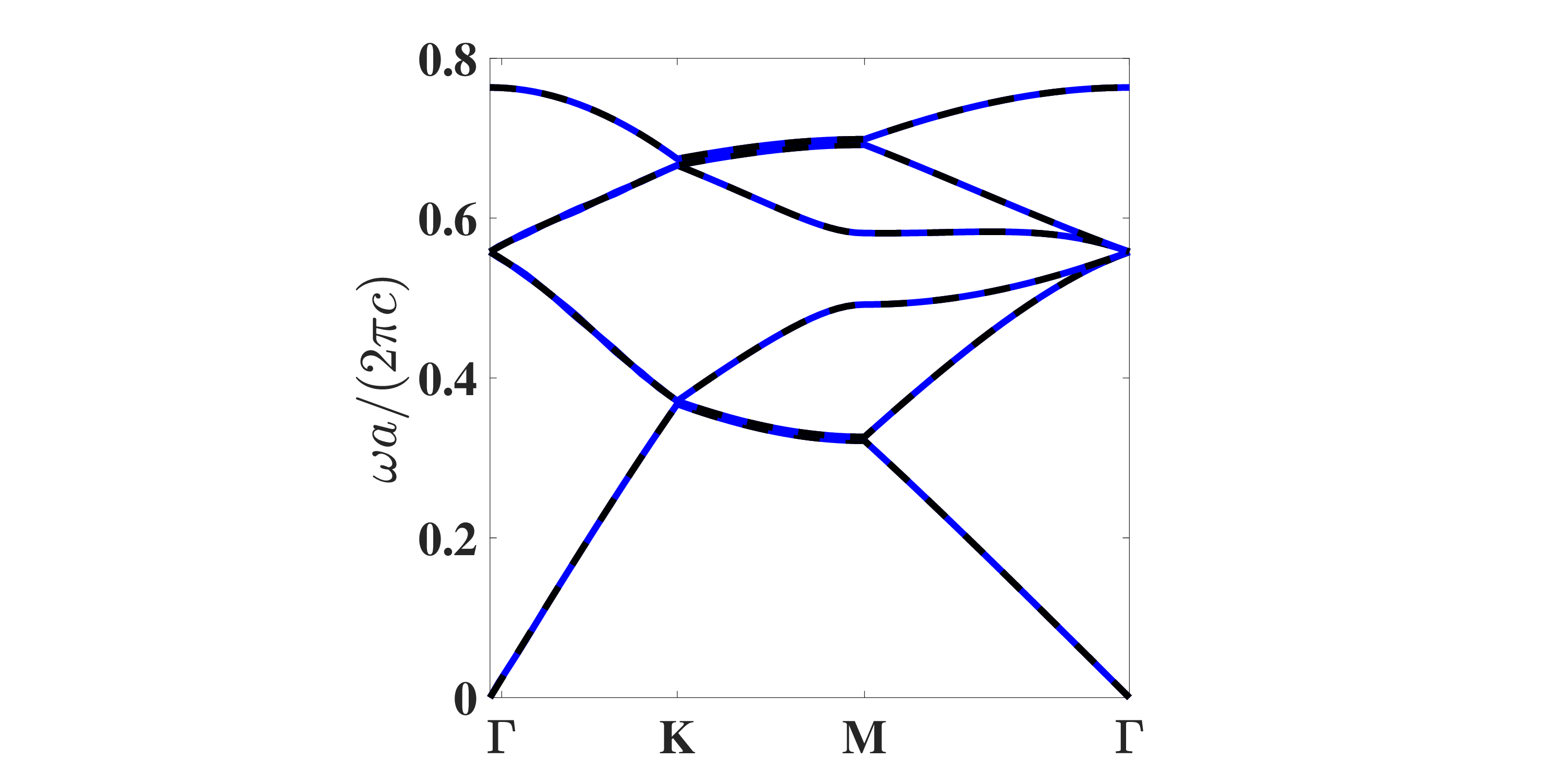}
}%
\caption{Hexagonal lattice: band functions along $\partial\mathcal{B}_{\text{red}}$} using Fekete points (the first row) and Cheb2 (the second row) with degree $n=8$. The dashed black line is the reference band structures.
\label{fig:edge}
\end{figure} 
In Figure \ref{error hex}, we display the performance of Lagrange interpolation methods \eqref{eq:interpolation} based upon those five types of sampling methods measured in $\operatorname{error}_{\infty}$ and $\operatorname{error}_{\operatorname{avg}}$ against the number of sampling points inside $\mathcal{B}_{\text{red}}$. One can observe excellent performance for all cases from Figures \ref{erroe square}
and \ref{error hex}. For instance, 45 sampling points inside $\mathcal{B}_{\text{red}}$ leads to $\operatorname{error}_{\infty}$ below $1\%$ for all sampling points in Section \ref{subsec:sampling points-t}. Note that the performance for hexagonal unit cell is typically better than that for the square unit cell since the first six band functions are smoother in the former case. One can infer the regularity of band functions along the edges of $\mathcal{B}_{\text{red}}$. We observe from Figures \ref{fig:edge} and \ref{fig:bandStructure} that the band functions with square lattice exhibit a larger and more diverse frequency distribution range, resulting in more singularities. Besides, we can also observe from Figures \ref{erroe square}
and \ref{error hex} that the performance of Lagrange interpolation based upon the first three sampling methods (Fekete points, mean optimal points and improved Lobatto grid) in both TE mode and TM mode, are similar or better than Cheb1 and Cheb2.
\begin{figure}[H]
\centering
\subfigure[Square lattice TE mode]{\label{zoom2}
\includegraphics[width=0.32\textwidth,trim={11cm 1cm 12cm 1cm},clip]{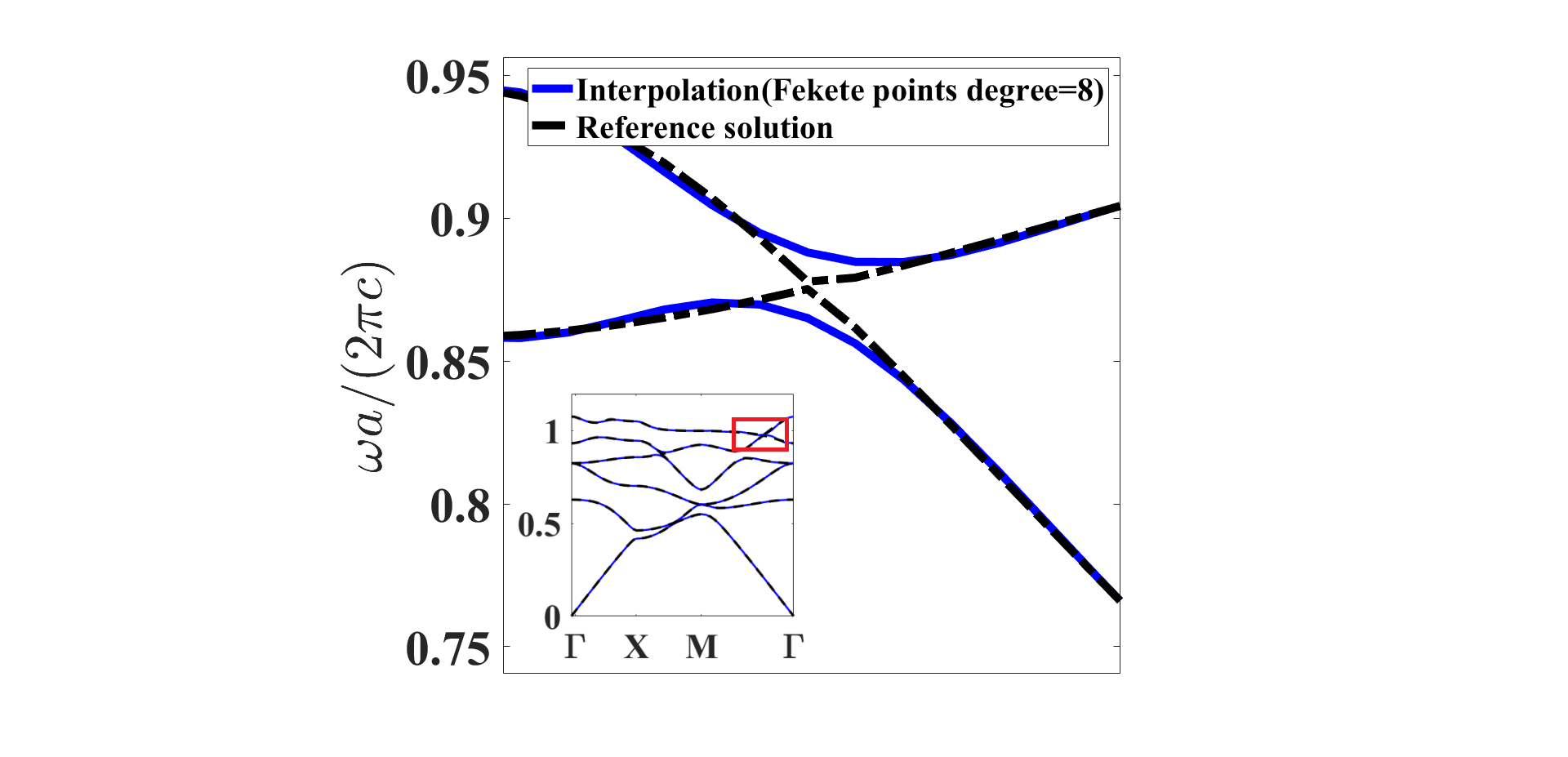}
}%
\quad
\subfigure[Square lattice TM mode]{\label{zoom1}
\includegraphics[width=0.32\textwidth,trim={11cm 1cm 12cm 1cm},clip]{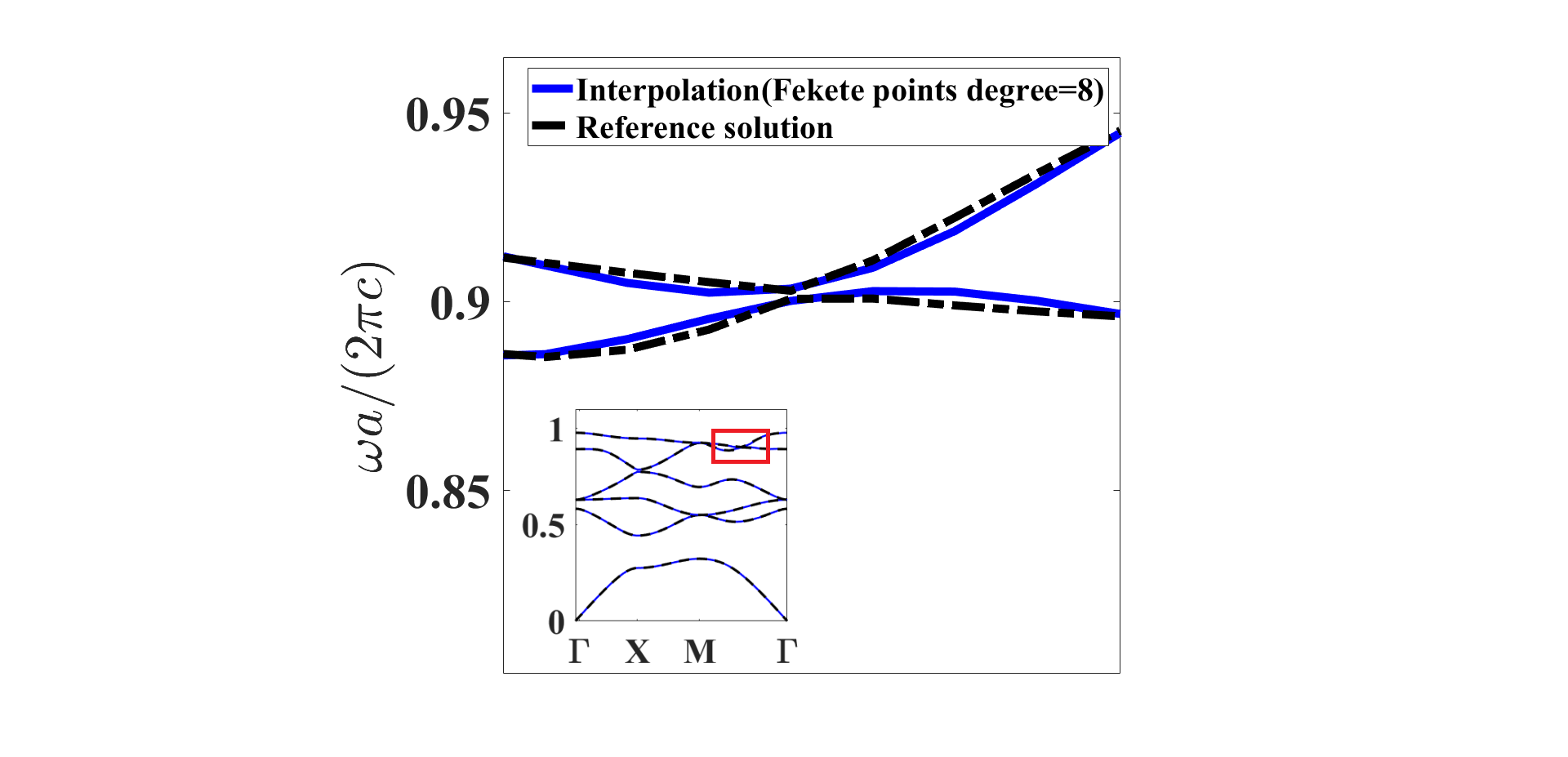}

}%
 \\            
\subfigure[Hexagonal lattice TE mode]{\label{zoom4}
\includegraphics[width=0.32\textwidth,trim={11cm 1cm 12cm 1cm},clip]{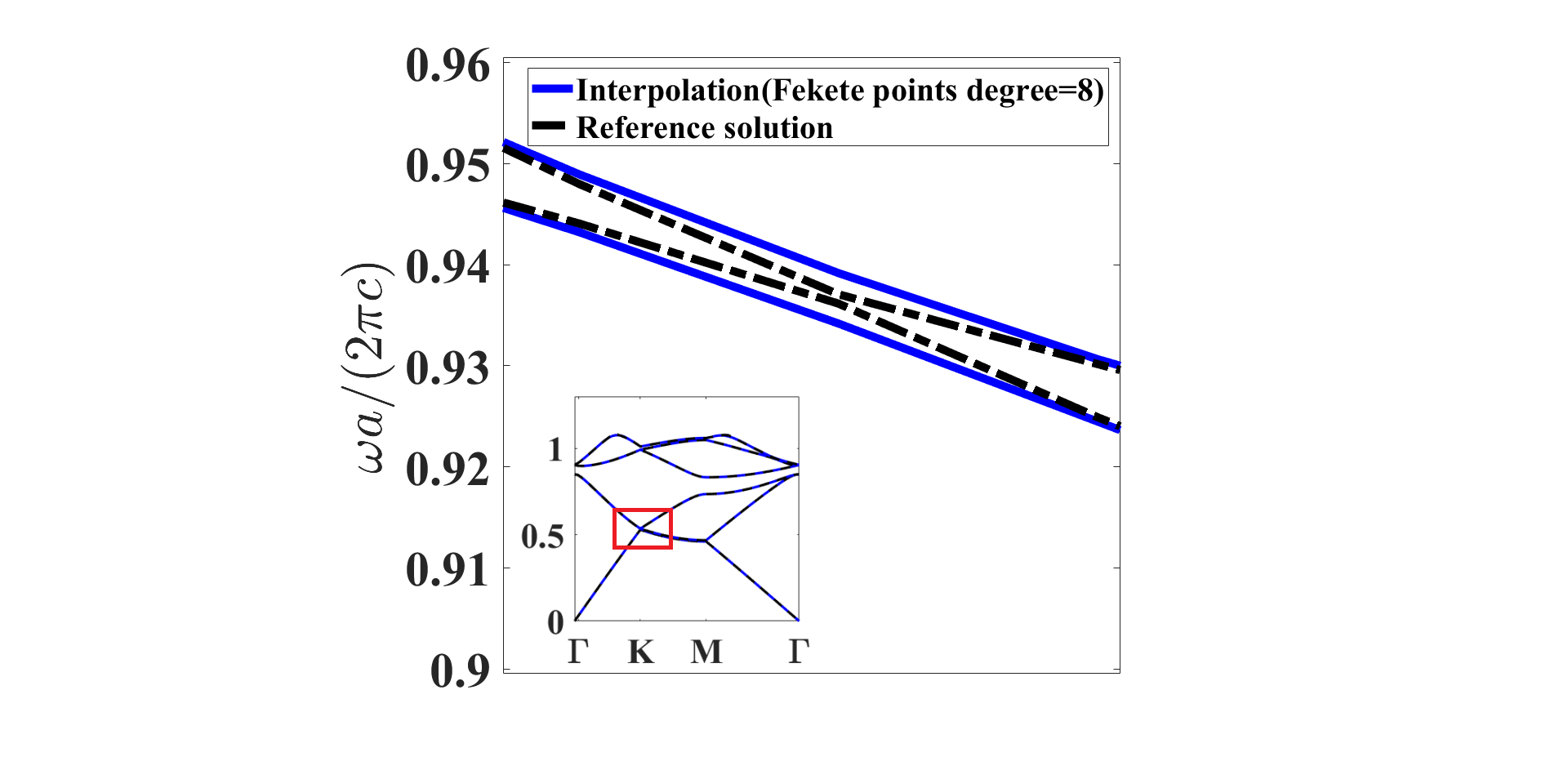}
}%
\quad
\subfigure[Hexagonal lattice TM mode]{\label{zoom3}
\includegraphics[width=0.32\textwidth,trim={11cm 1cm 12cm 1cm},clip]{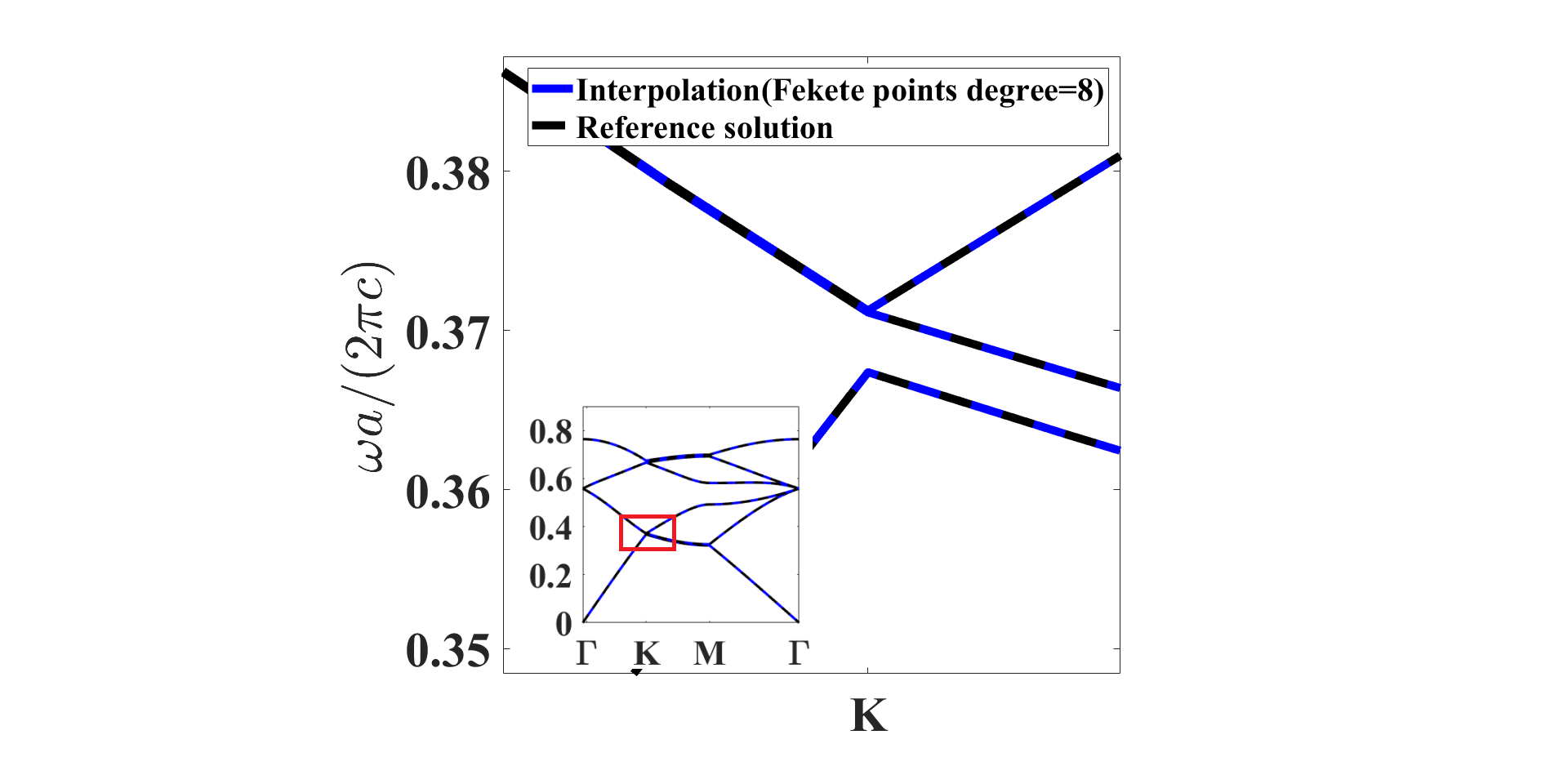}
}%
\caption{A magnified view of certain graphs in the red box from Figures \ref{fig:bandStructure} and \ref{fig:edge}.}\label{zoom in}
\end{figure} 
Figure \ref{fig:edge} illustrates the approximate band structure along $\partial\mathcal{B}_{\text{red}}$ using Fekete points and Cheb2 with degree $n=8$, showing that the approximate band functions attain a reasonable accuracy. Together with Figure \ref{fig:bandStructure}, we conclude that 
all those five methods can provide reasonably good reconstruction.
Figure \ref{zoom in} shows a magnified view of some special regions in Figures \ref{fig:bandStructure} and \ref{fig:edge}. We can observe that the areas with large interpolation errors are clearly the areas where the adjacent band functions are very close, which is consistent with the conclusion we have drawn before that the branch points are singular points. 

Note that it is quite difficult or even impossible to identify all branch points or distinguish from a fake branch point where two adjacent band functions are close but without intersecting due to many factors, for example the rounding error and numerical error resulting from \eqref{fem variational}. As a result, our numerical method can only yield approximative branch points where the adjacent band functions are close. Nevertheless, the overall performance demonstrates that our method is capable of approximating the band functions of 2D PhCs within a reasonable accuracy.

We present in Figure \ref{log log} the convergence of our proposed method for the crystals with both square lattice and hexagonal lattice for the TE mode and TM mode. We observe algebraic convergence and the slopes of these five sampling methods are similar as we expected from the approximation theory \cite{timan2014theory}.
\begin{figure}[H]
\centering
\subfigure[Square lattice TE mode]{\label{log2}
\includegraphics[width=.35\textwidth,trim={11cm 0cm 12.5cm 0cm},clip]{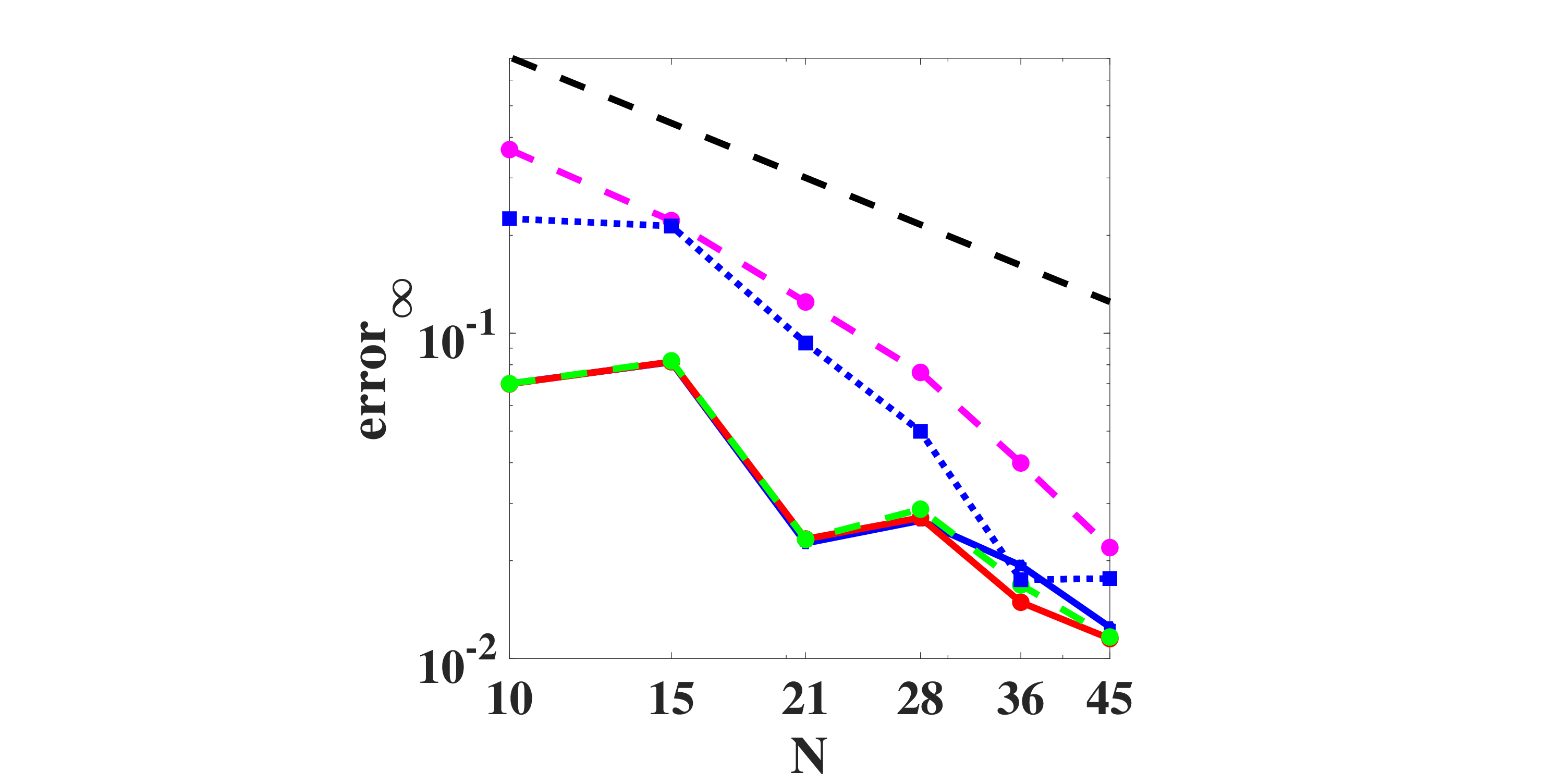}
}%
\quad
\subfigure[Square lattice TM mode]{\label{log1}
\includegraphics[width=.35\textwidth,trim={11cm 0cm 12.5cm 0cm},clip]{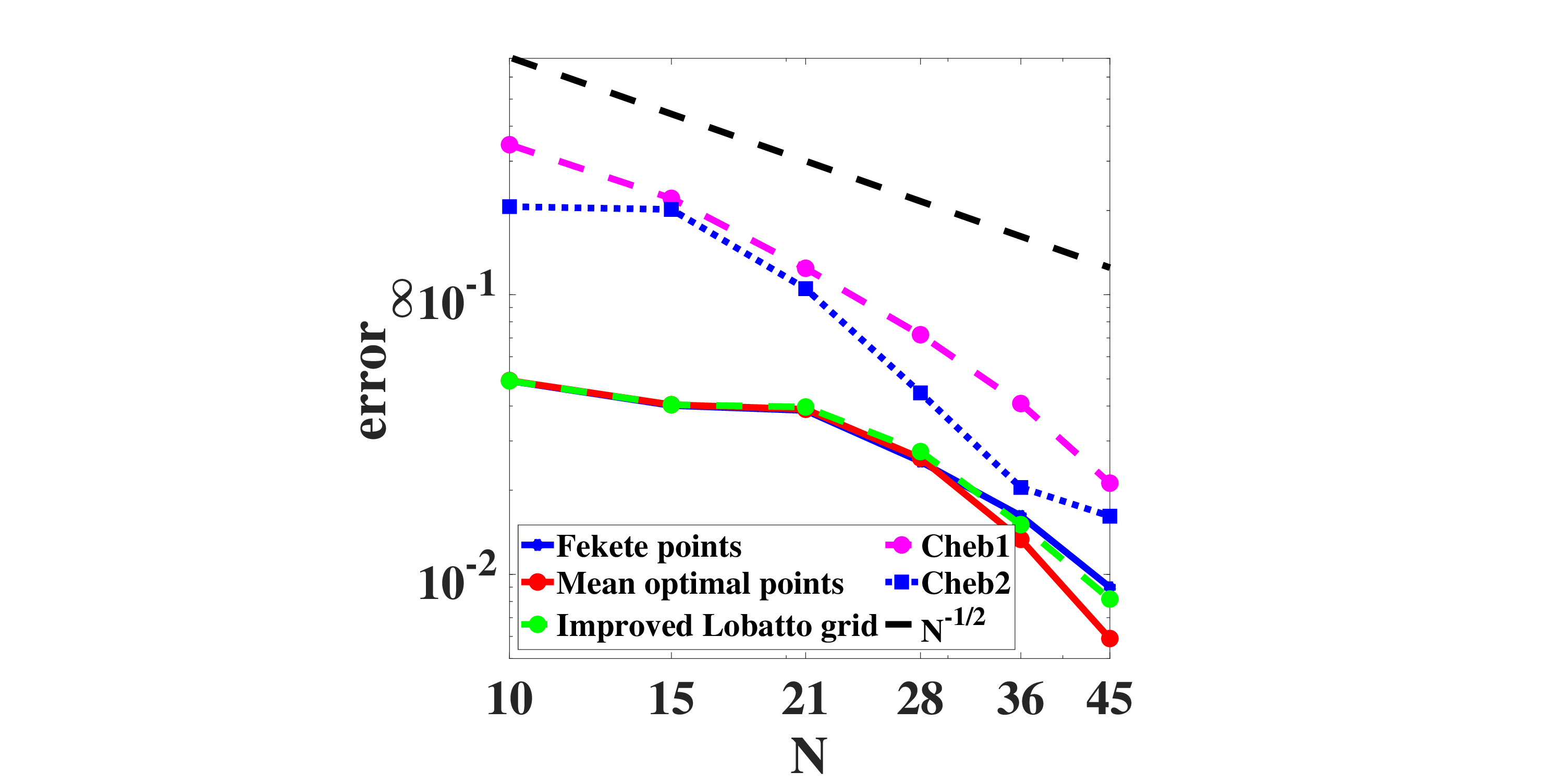}

}%
 \\             
\subfigure[Hexagonal lattice TE mode]{\label{log4}
\includegraphics[width=.35\textwidth,trim={11cm 0cm 12.5cm 0cm},clip]{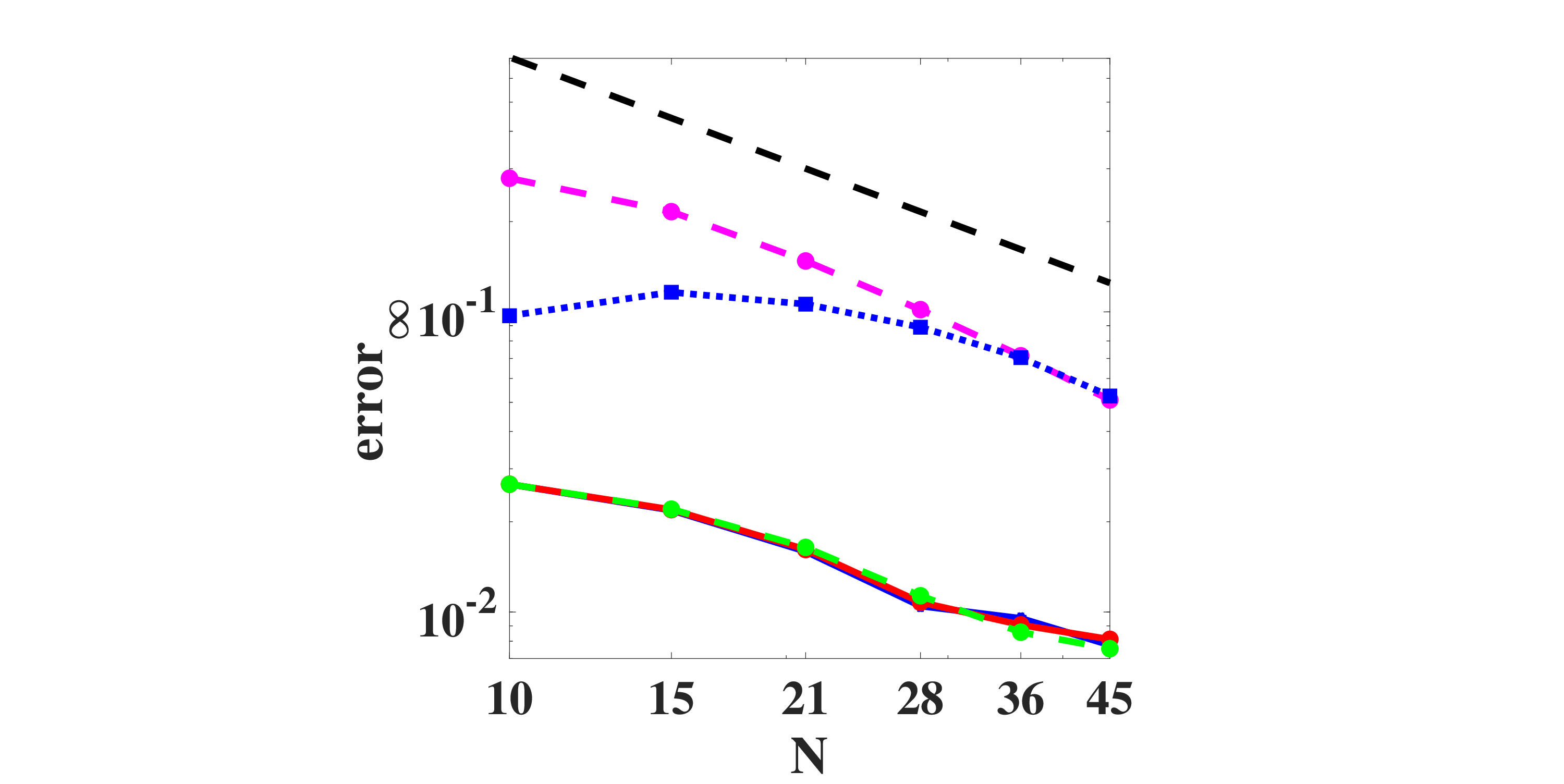}
}%
\quad
\subfigure[Hexagonal lattice TM mode]{\label{log3}
\includegraphics[width=.35\textwidth,trim={11cm 0cm 12.5cm 0cm},clip]{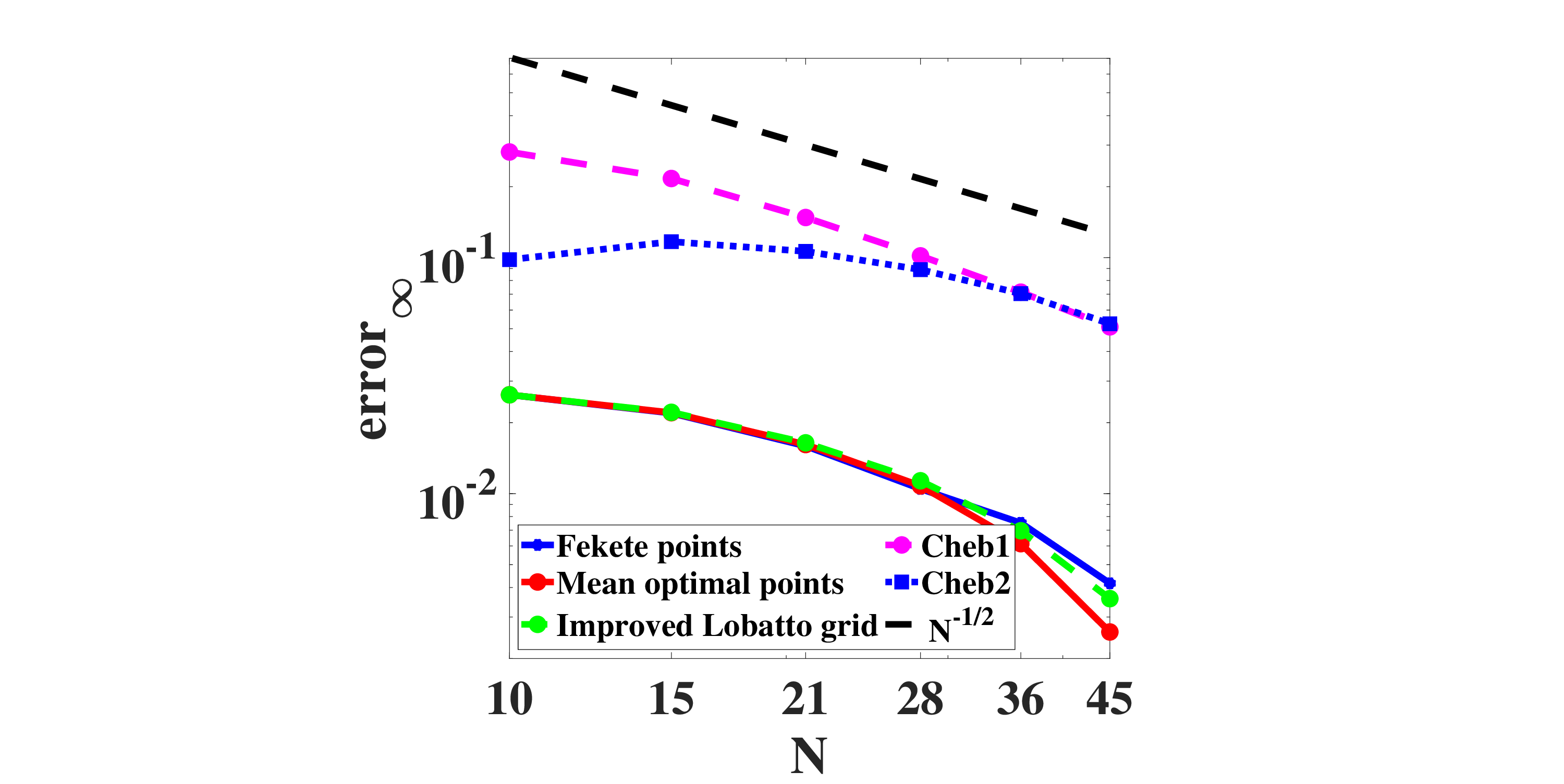}
}%
\caption{Convergence results.}\label{log log}
\end{figure} 
\section{Conclusion}\label{sec:conclusion}
In this paper, we analyze the properties of photonic band functions and consider the problem of band structure reconstruction in the context of 2D PhCs. The regularity of band functions is crucial for our proposed approximation method. In contrast to the traditional sampling algorithms based upon global polynomial interpolation and limited to the edges of the irreducible Brillouin zone, we propose an efficient and accurate global approximation algorithm based upon the Lagrange interpolation methods for computing band functions over the whole Brillouin zone. 
Regarding the selection of sampling points, we consider five different sampling algorithms to select suitable interpolation points in the irreducible Brillouin zone or a quadrilateral composed of the irreducible Brillouin zone and its mirror symmetry along one edge. We observe algebraic convergence rate and the numerical tests demonstrate that our method can approximate band functions efficiently. For example, our methods reach relative error below $1\%$ using only 45 sampling points for sampling methods defined on the Irreducible Brillouin zone. It should be noted that we focus on sampling algorithms based upon global polynomial interpolation in this paper since this is the current state of the art. However, this current method cannot identify branch points quite efficiently, since the global interpolation approximation has a relatively slow convergence rate due to the piecewise analyticity of the band functions. In order to make better use of this property, we will explore adaptive sampling algorithms based upon piecewise polynomial interpolation in the future. 
\section*{Acknowledgments}
Y. W. acknowledges support from the Research Grants Council (RGC) of Hong Kong via the Hong Kong PhD Fellowship Scheme (HKPFS). 
G.L. acknowledges support from Newton International Fellowships Alumni following-on funding awarded by The Royal Society and Early Career Scheme (Project number: 27301921), RGC, Hong Kong. We thank Richard Craster (Imperial College London) for fruitful discussion. We thank the anonymous referees for very detailed and constructive comments.
\bibliographystyle{abbrv}
\bibliography{refer}
\end{document}